\documentclass[nosumlimits,twoside]{amsart}
\usepackage{amsfonts, amsmath, amssymb}
\usepackage[bookmarksnumbered,plainpages,driverfallback=dvipdfm,backref=page]{hyperref}
\usepackage{graphicx}
\usepackage{float}
\usepackage{srcltx}
\usepackage[all]{xy}
\usepackage{version}
\usepackage[final]{showlabels}
\usepackage[T1]{fontenc}
\usepackage{xcolor}
\usepackage{enumerate}
\usepackage[normalem]{ulem}
\usepackage{bm}
\usepackage{latexsym}
\usepackage[2emode]{psfrag}
\usepackage{yhmath}
\usepackage{array}
\usepackage{dsfont}
\usepackage{tikz-cd}

\theoremstyle{plain}
\newtheorem{thm}{Theorem}[section]
\newtheorem{lem}[thm]{Lemma}
\newtheorem{prop}[thm]{Proposition}
\newtheorem{cor}[thm]{Corollary}

\theoremstyle{definition}

\theoremstyle{remark}
\newtheorem{oss}[thm]{Remark}
\allowdisplaybreaks
\excludeversion{invisible}
\excludeversion{proof?}

\newcommand{\Cc}{\mathcal{C}}

\newcommand{\Mm}{\mathcal{M}}
\newcommand{\mm}{\mathfrak{M}}

\def\Hc{{\mathrm{Hopf_{coc}}(\mathrm{Vec}_{G})}}

\begin{document}

\title[Semi-abelian condition for color Hopf algebras]{Semi-abelian condition for color Hopf algebras}

\author{Andrea Sciandra}

\address{
\parbox[b]{\linewidth}{University of Turin, Department of Mathematics ``G. Peano'', via
Carlo Alberto 10, I-10123 Torino, Italy}}
\email{andrea.sciandra@unito.it}

\thanks{This paper was written while the author was member of the "National Group for Algebraic and Geometric Structures and their Applications" (GNSAGA-INdAM). He was partially supported by MIUR within the National Research Project PRIN 2017.}

\maketitle

\begin{abstract}
   Recently in \cite{article1} it was shown that the category of cocommutative Hopf algebras over an arbitrary field $\Bbbk$ is semi-abelian. We extend this result to the category of cocommutative color Hopf algebras, i.e. of cocommutative Hopf monoids in the symmetric monoidal category of $G$-graded vector spaces with $G$ an abelian group, given an arbitrary skew-symmetric bicharacter on $G$, when $G$ is finitely generated and the characteristic of $\Bbbk$ is different from 2 (not needed if $G$ is finite of odd cardinality). We also prove that this category is action representable and locally algebraically cartesian closed, then algebraically coherent. In particular, these results hold for the category of cocommutative super Hopf algebras by taking $G=\mathbb{Z}_{2}$. Furthermore, we prove that, under the same assumptions on $G$ and $\Bbbk$, the abelian category of abelian objects in the category of cocommutative color Hopf algebras is given by those cocommutative color Hopf algebras which are also commutative.
\end{abstract}

\tableofcontents

\section{Introduction}

The notion of semi-abelian category was introduced by G. Janelidze et al. in \cite{article7} in order to capture typical algebraic properties valid for groups, rings and algebras. As it is said in \cite{article7}, semi-abelian categories provide a good categorical foundation for a meaningful treatment of radical and commutator theory and of (co)homology theory of non-abelian structures. Semi-abelian categories are rich in properties, for instance here the notions of semi-direct product, internal action and of crossed module are natural. Some examples of semi-abelian categories are the categories of groups, Lie algebras, (associative) rings and compact groups. In \cite{articlee} M. Gran et al. proved that the category of cocommutative Hopf algebras over a field $\Bbbk$, denoted by Hopf$_{\Bbbk,\mathrm{coc}}$, is semi-abelian when $\Bbbk$ has characteristic 0. Then the result was extended to arbitrary characteristic in \cite{article1}. Hence it becomes natural to ask if this is true also for the category of cocommutative color Hopf algebras, i.e. of cocommutative Hopf monoids in the category Vec$_{G}$ of $G$-graded vector spaces, which we denote by $\Hc$, where $G$ is an abelian group since we know that, in this case, Vec$_{G}$ becomes a symmetric monoidal category by using a skew-symmetric bicharacter on $G$ which modifies the braiding of Vec$_{\Bbbk}$ given by the usual tensor flip. We show that $\Hc$ is semi-abelian if the abelian group $G$ is finitely generated and the characteristic of the field $\Bbbk$ is not 2 (not needed if $G$ is finite of odd cardinality). This generalizes the result for ordinary cocommutative Hopf algebras since we can recover Hopf$_{\Bbbk,\mathrm{coc}}$ by taking $G$ as the trivial group in which case the symmetric monoidal category Vec$_{G}$ is exactly Vec$_{\Bbbk}$. Furthermore, if we consider $G=\mathbb{Z}_{2}$ we obtain that the category of cocommutative super Hopf algebras, extensively used in Mathematics and Physics, is semi-abelian if char$\Bbbk\not=2$. $\\$

The organization of the paper is the following. After calling back some basic notions and results about monoidal categories and (color) Hopf algebras, we prove the completeness and the cocompleteness of $\Hc$ by showing explicitly limits and colimits and the protomodularity of $\Hc$ by using a categorical result. We also observe that $\Hc$ is locally presentable, which is not guaranteed in general for the category of (cocommutative) Hopf monoids in a symmetric monoidal category. Then we show the regularity of $\Hc$ through the same steps of \cite{article1}. In particular, we obtain a generalization of a theorem by K. Newman \cite[Theorem 4.1]{article8} for cocommutative color Hopf algebras in the case char$\Bbbk\not=2$ and the abelian group $G$ is finitely generated, by using \cite[Theorem 3.10 (3)]{article2} about cocommutative super Hopf algebras together with a braided strong monoidal functor from the category Vec$_{G}$ to the category Vec$_{\mathbb{Z}_2}$ from \cite{article5}. Then, through an equivalent characterization given in \cite{article7}, we obtain that $\Hc$ is semi-abelian, still in case the abelian group $G$ is finitely generated and char$\Bbbk\not=2$. Finally, we also show that, under the same assumptions on $G$ and $\Bbbk$, the category $\Hc$ is action representable and locally algebraically cartesian closed (then algebraically coherent) and that the category of abelian objects in $\Hc$ consists of those cocommutative color Hopf algebras which are also commutative and then, as a consequence, this category is abelian.

\section{Preliminaries}

\subsection{Monoidal categories}
First we recall some basic facts about monoidal categories, which can be found in \cite{article9,l}. Let $(\mathcal{M},\otimes,\mathbf{I},a,l,r)$ be a monoidal category. We write $(\mathcal{M},\otimes,\mathbf{I})$ without the constraints $a$, $l$ and $r$ if these are clear from the context and we usually omit to write $a$ in the computations since it will be clear when it is needed, in order to have slightly more compact formulas. We know that we can consider the category Mon($\mathcal{M}$) of monoids in $\mathcal{M}$, whose objects will be denoted as $(A,m,u)$, and the category Comon($\mathcal{M}$) of comonoids in $\mathcal{M}$, whose objects will be denoted as $(C,\Delta,\epsilon)$. Recall that a monoid $M'$ is a $\textit{submonoid}$ of a monoid $M$, provided there exists a monoid morphism $i:M'\to M$ such that it is a monomorphism in $\Mm$. Analogously a comonoid $C'$ is a $\textit{subcomonoid}$ of a comonoid $C$, provided there exists a comonoid morphism $i:C'\to C$ such that it is a monomorphism in $\Mm$. In case $\Mm$ has a braiding $c$ the categories Mon$(\Mm)$ and Comon$(\Mm)$ become monoidal with the same constraints $a,l,r$. In this case, given monoids $(M_{1},m_{1},u_{1})$ and $(M_{2},m_{2},u_{2})$ in $\Mm$, the tensor product $\otimes$ is such that we have $(M_{1},m_{1},u_{1})\otimes(M_{2},m_{2},u_{2}):=(M_{1}\otimes M_{2},m,u)$ where 
\[
m:=(m_{1}\otimes m_{2})\circ(\mathrm{Id}_{M_{1}}\otimes c_{M_{2},M_{1}}\otimes\mathrm{Id}_{M_{2}})\ \text{ and }\ u:=(u_{1}\otimes u_{2})\circ r_{\mathbf{I}}^{-1}. 
\]
The unit object of Mon$(\Mm)$ is given by $(\mathbf{I},r_{\mathbf{I}},\mathrm{Id}_{\mathbf{I}})$. Similarly, given comonoids $(C_{1},\Delta_{1},\epsilon_{1})$ and $(C_{2},\Delta_{2},\epsilon_{2})$ in $\Mm$, $(C_{1},\Delta_{1},\epsilon_{1})\otimes(C_{2},\Delta_{2},\epsilon_{2}):=(C_{1}\otimes C_{2},\Delta,\epsilon)$ is a comonoid where
\[
\Delta:=(\mathrm{Id}_{C_{1}}\otimes c_{C_{1},C_{2}}\otimes\mathrm{Id}_{C_{2}})\circ(\Delta_{1}\otimes\Delta_{2})\ \text{ and }\ \epsilon:=r_{\mathbf{I}}\circ(\epsilon_{1}\otimes\epsilon_{2}). 
\]
The unit object of Comon($\Mm$) is given by $(\mathbf{I},r_{\mathbf{I}}^{-1},\mathrm{Id}_{\mathbf{I}})$. When Mon$(\Mm)$ and Comon$(\Mm)$ are monoidal we can consider monoids and comonoids in them. Hence we have that
\begin{equation}\label{iso}
\text{Bimon}(\Mm)\cong\text{Mon(Comon}(\Mm))\cong\text{Comon(Mon}(\Mm))
\end{equation}
where $\text{Bimon}(\Mm)$ is the category of bimonoids in $\Mm$, since for $(B,m,u,\Delta,\epsilon)$ the fact that $m$ and the $u$ are morphisms of comonoids is equivalent to that $\Delta$ and $\epsilon$ are morphisms of monoids (see e.g. \cite[Proposition 1.11]{l}) while
\begin{equation}\label{coco}
\text{Mon(Mon}(\Mm))\cong\text{Mon}_{\mathrm{c}}(\Mm)\ \text{ and }\ \text{Comon(Comon}(\Mm))\cong\text{Comon}_{\mathrm{coc}}(\Mm)
\end{equation}
which are the category of commutative monoids and of cocommutative comonoids in $\Mm$ respectively and this follows from the $\textit{Eckmann–Hilton argument}$: $\Delta_{C}$ is a morphism of comonoids if and only if $C$ is cocommutative and $m_{A}$ is a morphism of monoids if and only if $A$ is commutative (see e.g. \cite[Section 1.2.7]{l}). We recall that a monoid $(A,m,u)$ is $\textit{commutative}$ if $m=m\circ c_{A,A}$ and a comonoid $(C,\Delta,\epsilon)$ is $\textit{cocommutative}$ if $c_{C,C}\circ\Delta=\Delta$. Also recall that a bimonoid $B'$ is a $\textit{sub-bimonoid}$ of a bimonoid $B$, provided there exists a bimonoid morphism $i:B'\to B$ such that it is a monomorphism in $\Mm$. Given $(C,\Delta,\epsilon)\in\text{Comon}(\mathcal{M})$ and $(A,m,u)\in\text{Mon}(\mathcal{M})$, Hom$_{\mathcal{M}}(C,A)$ is an (ordinary) monoid with $\textit{convolution\ product}$ such that, given $f,g:C\rightarrow A$ in $\mathcal{M}$, the product is $f*g:=m\circ(f\otimes g)\circ\Delta$ and the unit is $u\circ\epsilon$. Hence we can consider the category Hopf($\Mm$) of Hopf monoids in $\mathcal{M}$, whose objects are bimonoids $B$ in $\mathcal{M}$ equipped with a morphism $S:B\rightarrow B$ ($\textit{antipode}$) which is the convolution inverse of $\text{Id}_{B}$. The monoidal categories Mon($\Mm$) and Comon($\Mm$) may fail to be braided and then the categories Hopf($\Mm$), Bimon($\Mm$), Mon$_{\mathrm{c}}(\Mm)$ and Comon$_{\mathrm{coc}}(\Mm)$ may fail to be monoidal but, when the braided category $\Mm$ is symmetric, i.e. $c_{X,Y}^{-1}=c_{Y,X}$ for every $X$ and $Y$ in $\Mm$, these categories are all braided and symmetric with the same braiding $c$ and the same constraints $a,l,r$ of $\Mm$ (see \cite[Section 1.2.7]{l}). Indeed, if $\Mm$ is symmetric, given $A$ and $B$ monoids in $\Mm$, then $c_{A,B}:A\otimes B\to B\otimes A$ is a morphism of monoids and then \text{Mon}($\Mm$) is a symmetric monoidal category and, dually, \text{Comon}($\Mm$) is a symmetric monoidal category. Iterating these results and applying \eqref{iso} and \eqref{coco}, one can deduce that \text{Bimon}($\Mm$), \text{Mon}$_{\mathrm{c}}(\Mm)$ and $\text{Comon}_{\mathrm{coc}}(\Mm)$ are symmetric monoidal categories as well. Furthermore, if $\Mm$ is symmetric, given $(B,S_{B})$ and $(B',S_{B'})$ in Hopf($\Mm$) we have that $(B,S_{B})\otimes(B',S_{B'}):=(B\otimes B',S_{B}\otimes S_{B'})$ is in Hopf($\Mm$). The antipode is a bimonoid morphism $S:B\to B^{\mathrm{op},\mathrm{cop}}$ where $(B^{\mathrm{op},\mathrm{cop}},m^{\mathrm{op}},u,\Delta^{\mathrm{cop}},\epsilon)$ is a bimonoid with $m^{\mathrm{op}}=m\circ c_{B,B}$ and $\Delta^{\mathrm{cop}}=c_{B,B}\circ\Delta$, hence it satisfies 
\[
m\circ c_{B,B}\circ(S\otimes S)=S\circ m\ \text{ and }\ (S\otimes S)\circ\Delta=c_{B,B}\circ\Delta\circ S 
\]
and so if $B$ is commutative then $S$ is a morphism of monoids while if $B$ is cocommutative then $S$ is a morphism of comonoids. Also note that if $B$ is commutative or cocommutative then $S^2=\mathrm{Id}_{B}$. Indeed, for instance, if $B$ is cocommutative then
\[
\begin{split}
m\circ(S\otimes S^{2})\circ\Delta&=m\circ(\mathrm{Id}_{B}\otimes S)\circ(S\otimes S)\circ\Delta=m\circ(\mathrm{Id}_{B}\otimes S)\circ c_{B,B}\circ\Delta\circ S=u\circ\epsilon\circ S=u\circ\epsilon
\end{split}
\]
and, analogously, $m\circ(S^{2}\otimes S)\circ\Delta=u\circ\epsilon$.
\begin{invisible}
\[
\begin{split}
m\circ(S\otimes S^{2})\circ\Delta&=m\circ(S\otimes S)\circ(\mathrm{Id}_{B}\otimes S)\circ\Delta=m\circ c_{B,B}\circ(S\otimes S)\circ(\mathrm{Id}_{B}\otimes S)\circ\Delta\\&=S\circ m\circ(\mathrm{Id}_{B}\otimes S)\circ\Delta=S\circ u\circ\epsilon=u\circ\epsilon
\end{split}
\]
and, analogously, $m\circ(S^{2}\otimes S)\circ\Delta=u\circ\epsilon$, then $S^{2}=\mathrm{Id}_{B}$. In the same way, with $B$ cocommutative
 then $S^{2}=\mathrm{Id}_{B}$. 
\end{invisible}
Since we use several times these facts in the following and, in particular, the fact that Comon$_{\mathrm{coc}}(\Mm)$ is a monoidal category is central for our proof of protomodularity, then we will work with a symmetric monoidal category. $\\$

Finally, recall that, given monoidal categories $(\Mm,\otimes,\mathbf{I},a,l,r)$ and $(\Mm',\otimes,\mathbf{I}',a',l',r')$ (where we do not use different notations for $\otimes$ for notation convenience), a monoidal functor $(F, \phi^{0},\phi^{2}):(\Mm,\otimes,\mathbf{I},a,l,r)\to(\Mm',\otimes,\mathbf{I}',a',l',r')$ consists of a functor $F:\Mm\to\Mm'$, a morphism $\phi^{2}_{X,Y}:F(X\otimes Y)\to F(X)\otimes F(Y)$ in $\Mm'$ for every $X,Y$ in $\Mm$ which is natural in $X$ and $Y$ and a morphism $\phi^0:F(\mathbf{I})\to\mathbf{I}'$ in $\Mm'$ such that 
\[
(\mathrm{Id}_{F(X)}\otimes\phi^{2}_{Y,Z})\circ\phi^{2}_{X,Y\otimes Z}\circ F(a_{X,Y,Z})=a'_{F(X),F(Y),F(Z)}\circ(\phi^2_{X,Y}\otimes\mathrm{Id}_{F(Z)})\circ\phi^{2}_{X\otimes Y,Z}
\]
and
\[
l'_{F(X)}\circ(\phi^{0}\otimes\mathrm{Id}_{F(X)})\circ\phi^{2}_{\mathbf{I},X}=F(l_{X})\ \text{ and }\ r'_{F(X)}\circ(\mathrm{Id}_{F(X)}\otimes\phi^{0})\circ\phi^{2}_{X,\mathbf{I}}=F(r_{X}).
\]
Furthermore, $F$ is called \textit{strong} if $\phi^0$ and $\phi^2_{X,Y}$ are isomorphisms for every $X,Y$ in $\Mm$ and \textit{strict} if $\phi^0$ and $\phi^2_{X,Y}$ are identities for every $X,Y$ in $\Mm$. If $\Mm$ and $\Mm'$ are (symmetric) braided with braidings $c$ and $c'$ respectively, $F$ is called (symmetric) \textit{braided} if $c'_{F(X),F(Y)}\circ\phi^2_{X,Y}=\phi^2_{Y,X}\circ F(c_{X,Y})$. $\\$

If $\mathcal{M}$ is the category $\mathrm{Vec}_{\Bbbk}$ of vector spaces over a field $\Bbbk$, we have the usual notions of $\Bbbk$-algebras, $\Bbbk$-coalgebras, $\Bbbk$-bialgebras and $\Bbbk$-Hopf algebras, usually denoted without $\Bbbk$. In the following we always omit $\Bbbk$ but it will be understood. For classical results and notions about the theory of Hopf algebras we refer to \cite{book4} and \cite{book5}.

\subsection{Semi-abelian categories}
Here we recall some definitions needed for the notion of semi-abelian category. For the notions of limits and colimits of a functor, as for other basic notions of category theory, we refer to \cite{book1,book3}. $\\$

A finitely complete category $\Cc$ is \textit{regular} if any arrow of $\Cc$ factors as a \textit{regular epimorphism} (i.e. the coequalizer of a pair of morphisms of $\mathcal{C}$) followed by a monomorphism and if, moreover, regular epimorphisms are stable under pullbacks along any morphism. A \textit{relation} on an object $X$ of $\Cc$ is an equivalence class of triples $(R,r_{1},r_{2})$, where $R$ is an object of $\Cc$ and $r_{1},r_{2}:R\to X$ is a pair of jointly monic morphisms of $\Cc$, and two triples $(R,r_{1},r_{2})$ and $(R',r'_{1},r'_{2})$ are identified when they both factor through each other. 
\begin{invisible}
Furthermore a relation $(R,r_{1},r_{2})$ on $X$ is called: $\\$ $\\$
1) \textit{reflexive} if there is a (necessarily unique since $r_{1}$ and $r_{2}$ are jointly monic) morphism $\delta:X\to R$ such that $r_{1}\circ\delta=\mathrm{Id}_{X}=r_{2}\circ\delta$; $\\$ $\\$
2) \textit{symmetric} if there is a (necessarily unique) morphism $\sigma:R\to R$ such that $r_{1}\circ\sigma=r_{2}$ and $r_{2}\circ\sigma=r_{1}$; $\\$ $\\$
3) \textit{transitive} if there is a (necessarily unique) morphism $\tau:R\times_{X}R\to R$ such that $r_{1}\circ\tau=r_{1}\circ\pi_{1}$ and $r_{2}\circ\tau=r_{2}\circ\pi_{2}$, where $(R\times_{X}R,\pi_{1},\pi_{2})$ is the pullback of $r_{1}$ and $r_{2}$. 
\end{invisible}
An \emph{equivalence relation} in $\Cc$ is a relation $R$ on an object $X$ which is reflexive, symmetric and transitive. A regular category $\Cc$ is (Barr)-$\textit{exact}$ if any equivalence relation $R$ in $\Cc$ is \textit{effective}, i.e. it is the kernel pair of some morphism in $\Cc$. Recall also that a category $\Cc$ is $\textit{protomodular}$, in the sense of \cite{book2}, if it has pullbacks of split epimorphisms along any morphism and all the inverse image functors of the fibration of points reflect isomorphisms. We know that, as it is said for instance in \cite[Proposition 3.1.2]{book2}, if $\Cc$ is pointed (i.e. it has a zero object) and finitely complete, the protomodularity can be expressed by simply asking that the Split Short Five Lemma holds in $\Cc$. Finally, a category $\Cc$ is $\textit{semi-abelian}$ if it is pointed, finitely cocomplete, (Barr)-exact and protomodular. Many details and properties about semi-abelian categories can be found in \cite{book2}.

\section{Color Hopf algebras}

In this section we recall what color Hopf algebras are and how they differ from common Hopf algebras. We consider the category Vec$_{G}$ of $G$-graded vector spaces over an arbitrary field $\Bbbk$ where $G$ is a group. We add conditions on the group $G$ along the way, to make it clear why these are needed. Objects in Vec$_{G}$ are vector spaces $V=\bigoplus_{g\in G}{V_{g}}$ where $V_{g}$ is a vector subspace of $V$ for every $g\in G$ and the morphisms in Vec$_{G}$ are linear maps $f:V\to W$ which preserve gradings, i.e. such that $f(V_{g})\subseteq W_{g}$ for every $g\in G$. We know that this category is monoidal with $\otimes$ the tensor product of Vec$_{\Bbbk}$ and unit object $\Bbbk=\bigoplus_{g\in G}{\Bbbk_{g}}$ where $\Bbbk_{g}=\{0\}$ if $g\not=1_{G}$ and $\Bbbk_{1_{G}}=\Bbbk$, with $1_{G}$ the identity of $G$. Indeed, given $V=\bigoplus_{g\in G}{V_{g}}$ and $W=\bigoplus_{g\in G}{W_{g}}$, we have that $V\otimes W=\bigoplus_{g\in G}{(V\otimes W)_{g}}$ where $(V\otimes W)_{g}=\bigoplus_{a\in G}({V_{a}\otimes W_{a^{-1}g}})$. Also the associativity constraint and left and right unit constraints are the usual ones of Vec$_{\Bbbk}$. 

\begin{oss}
Remember that the category Vec$_{G}$ is isomorphic to the category ${}^{\Bbbk G}\mm$ of left comodules over the group algebra $\Bbbk G$ with isomorphism given by $F:\mathrm{Vec}_{G}\to{}^{\Bbbk G}\mm$, $V=\bigoplus_{g\in G}{V_{g}}\mapsto(V,\rho)$ with $\rho(\sum_{g\in G}{v_{g}})=\sum_{g\in G}{g\otimes v_{g}}$ and $F(f)=f$ and inverse given by $G:{}^{\Bbbk G}\mm\to\mathrm{Vec}_{G}$, $(V,\rho)\mapsto V=\bigoplus_{g\in G}{V_{g}}$ where $V_{g}=\{w\in V\ |\ \rho(w)=g\otimes w\}$ and $G(f)=f$. It is known that ${}^{\Bbbk G}\mm$ is a Grothendieck category and then abelian, since this is true in general for $^{C}\mm$ (and $\mm^{C}$) with $C$ a coalgebra, not always true in general for a coalgebra over a ring (see e.g. \cite[3.13]{bre}). So here monomorphisms are exactly the injective maps and epimorphisms the surjective maps. Observe that, given a graded vector space $V=\bigoplus_{g\in G}{V_{g}}$ and a vector subspace $V'\subseteq V$, we can always consider the graded vector space $\bigoplus_{g\in G}{V'\cap V_{g}}\subseteq V'$. Furthermore, $V'$ is a \textit{graded subspace} of $V$ if it is a graded vector space such that the inclusion $i:V'\to V$ is in Vec$_{G}$ and this happens if and only if for every $x\in V'$, which is $x=\sum_{g\in G}{x_{g}}$ with $x_{g}\in V_{g}$, then $x_{g}\in V'$ for any $g\in G$; in this case $V'$ has the induced grading $V'=\bigoplus_{g\in G}{V'_{g}}$ where $V'_{g}=V'\cap V_{g}$. Furthermore, we can always consider the graded vector space $\bigoplus_{g\in G}{V_{g}/(V_{g}\cap V')}$ and there is a canonical isomorphism of vector spaces $\bigoplus_{g\in G}{\frac{V_{g}}{V_{g}\cap V'}}\cong\frac{\bigoplus_{g\in G}{V_{g}}}{\bigoplus_{g\in G}{V_{g}\cap V'}}$ which is $V/V'$ in the case $V'$ is a graded subspace of $V$. In this case, we can also consider $\bigoplus_{g\in G}{(V_{g}+V')/V'}$, where $(V_{g}+V')/V'$ is a vector subspace of $V/V'$ for every $g\in G$. Notice that $V_{g}/(V_{g}\cap V')\cong(V_{g}+V')/V'$ canonically as vector spaces for every $g\in G$ and then $\bigoplus_{g\in G}{(V_{g}+V')/V'}$ and $\bigoplus_{g\in G}{V_{g}/(V_{g}\cap V')}$ can be identified in Vec$_{G}$. Thus, when $V'$ is a graded subspace of $V$, we have that $V/V'=\bigoplus_{g\in G}{(V_{g}+V')/V'}$ is in Vec$_{G}$ and it is called \textit{quotient graded vector space}. 
\end{oss}

\begin{oss}\label{Im and qu}
We recall that if $f:A\to B$ is in Vec$_{G}$ then $\mathrm{ker}(f)$ and $\mathrm{Im}(f)$ are graded subspaces of $A$ and $B$ respectively. 
\begin{invisible}
In fact $\mathrm{ker}(f)$ is a graded subspace of $A$ since, given $x=\sum_{g\in G}{x_{g}}\in\mathrm{ker}(f)$, then $0=f(x)=\sum_{g\in G}{f(x_{g})}\in\bigoplus_{g\in G}{B_{g}}$ and then, since the sum is direct, we must have $f(x_{g})=0$ for every $g\in G$, i.e. $x_{g}\in\mathrm{ker}(f)$ for every $g\in G$. Furthermore $f(A)$ is a graded subspace of $B$, indeed, given $y=\sum_{g\in G}{y_{g}}\in f(A)$, then there is $x=\sum_{g\in G}{x_{g}}\in A$ such that $y=f(x)=\sum_{g\in G}{f(x_{g})}$. Thus, since the sum is direct, $y_{g}=f(x_{g})\in f(A)$ for every $g\in G$. 
\end{invisible}
If $f$ is surjective, the grading of $B=f(A)$ is the unique induced by $A$ through $f$, i.e. $B_{g}=f(A_{g})$ for every $g\in G$. 
\begin{invisible}
Indeed $f(A_{g})\subseteq B_{g}$ and given $y\in B_{g}$ there exists $x=\sum_{a\in G}{x_{a}}\in A$ such that $f(x)=\sum_{a\in G}{f(x_{a})}=y$ since $f$ is surjective and then $y=f(x_{g})$ by comparing the degrees. Note that, given a canonical projection $\pi:A\to A/I$ in Vec$_{G}$, this is consistent since $(A/I)_{g}=A_{g}/A_{g}\cap I=\pi(A_{g})$, for every $g\in G$.
\end{invisible}
\end{oss}

\subsection{Graded (co)algebras}
The objects of the categories Mon(Vec$_{G}$) and Comon(Vec$_{G}$) are called $G$-$\textit{graded\ algebras}$ and $G$-$\textit{graded\ coalgebras}$ respectively, which we usually call without $G$. Many details and properties about graded algebras and graded coalgebras can be found in \cite{Na,Na2}. Note that graded algebras and graded coalgebras are often used to denote algebras and coalgebras graded over $\mathbb{N}$, while here gradings will be always over $G$. \medskip

A graded algebra is an algebra $(A,m,u)$ where $A=\bigoplus_{g\in G}{A_{g}}$ is a graded vector space such that $m$ and $u$ preserve gradings, i.e. for every $h,k\in G$ we have $A_{h}A_{k}\subseteq A_{hk}$ and $u(\Bbbk)\subseteq A_{1_{G}}$ and a morphism of graded algebras is a morphism of algebras that preserves gradings. Since monomorphisms in Vec$_{G}$ are exactly the injective maps, a submonoid of a graded algebra $(A,m,u)$, called $\textit{graded subalgebra}$, is a graded subspace $V\subseteq A$ such that $1_{A}\in V$ and $m(V\otimes V)\subseteq V$. Indeed, in this case, $V$ is a graded vector space with $V_{g}=V\cap A_{g}$ for every $g\in G$, an algebra and 
\[
V_{g}V_{h}=m((V\cap A_{g})\otimes(V\cap A_{h}))=m((V\otimes V)\cap(A_{g}\otimes A_{h}))\subseteq m(V\otimes V)\cap m(A_{g}\otimes A_{h})\subseteq V\cap A_{gh}=V_{gh}
\]
for every $g,h\in G$. Furthermore, if we consider a graded two-sided ideal $I$ of $A$ such that $A/I=\bigoplus_{g\in G}{(A_{g}+I)/I}$ is a graded vector space, we know that $(A/I,u_{A/I},m_{A/I})$ is an algebra with $u_{A/I}=\pi\circ u_{A}$ and $m_{A/I}\circ(\pi\otimes\pi)=\pi\circ m_{A}$ where $\pi:A\to A/I$ is the canonical projection and it is graded since $u_{A/I}$ and $m_{A/I}$ are in Vec$_{G}$ with $\pi$, $u_{A}$ and $m_{A}$ in Vec$_{G}$; it is called \textit{quotient graded algebra}. 
\medskip

Similarly, a graded coalgebra is a coalgebra $(C,\Delta,\epsilon)$ where $C=\bigoplus_{g\in G}{C_{g}}$ is a graded vector space such that $\Delta$ and $\epsilon$ preserve gradings, i.e. $\Delta(C_{g})\subseteq\bigoplus_{h\in G}({C_{h}\otimes C_{h^{-1}g}})$ and $\epsilon(C_{g})\subseteq\delta_{g,1_{G}}\Bbbk$ for every $g\in G$ and a morphism of graded coalgebras is a morphism of coalgebras that preserves gradings. A subcomonoid of a graded coalgebra $(C,\Delta,\epsilon)$, called $\textit{graded subcoalgebra}$, is a graded vector subspace $V\subseteq C$ such that $\Delta(V)\subseteq V\otimes V$ ($\epsilon(V)\subseteq\Bbbk$ is automatic). Indeed, in this case, $V$ is a graded vector space, a coalgebra and 
\[
\Delta(V_{g})=\Delta(V\cap C_{g})\subseteq\Delta(V)\cap\Delta(C_{g})\subseteq(V\otimes V)\cap(C\otimes C)_{g}=(V\otimes V)_{g},
\]
for every $g\in G$. In fact observe that, since $V$ is a graded subspace of $C$, then $V\otimes V$ is a graded subspace of $C\otimes C$. If $I$ is a graded two-sided coideal of $C$, then $C/I$ is a graded vector space and it is a coalgebra with $\Delta_{C/I}\circ\pi=(\pi\otimes\pi)\circ\Delta_{C}$ and $\epsilon_{C/I}\circ\pi=\epsilon_{C}$, where $\pi:C\to C/I$ is the canonical projection. Thus $C/I$ is a graded coalgebra because $\Delta_{C/I}$ and $\epsilon_{C/I}$ clearly preserve gradings since $\epsilon_{C}$, $\Delta_{C}$ and $\pi$ are in Vec$_{G}$ and it is called \textit{quotient graded coalgebra}.

\subsection{Color bialgebras and color Hopf algebras}
We are interested in studying Hopf monoids in Vec$_{G}$ but, in order to do this, first we need that Vec$_{G}$ is braided. One can give to Vec$_{G}$ a braiding by using a bicharacter $\phi$ on $G$ (see for example \cite{article5}), i.e. a map $\phi:G\times G\rightarrow\Bbbk-\{0\}$ such that 
\[
\phi(gh,l)=\phi(g,l)\phi(h,l)\ \text{ and }  \phi(g,hl)=\phi(g,h)\phi(g,l)\ \text{ for every } g,h,l\in G.
\]
It follows immediately that $\phi(1_{G},g)=\phi(g,1_{G})=1_{\Bbbk}$ for all $g\in G$. The monoidal category Vec$_{G}$ is braided with braiding $c_{X,Y}:X\otimes Y\rightarrow Y\otimes X$ such that $c_{X,Y}(x\otimes y)=\phi(g,h)y\otimes x$ for $x\in X_{g}$, $y\in Y_{h}$ and $g,h\in G$, defined on the components of the grading and extended by linearity, for every $X$ and $Y$ in Vec$_{G}$. In order to obtain that the braiding is in Vec$_{G}$, the group $G$ needs to be $\textit{abelian}$ as it is said in \cite[Section 1.1]{article5} or in \cite[pag. 193]{articlee2}. Hence, from now on, we will always consider $G$ an abelian group. As we said before we also want that the category Vec$_{G}$ is symmetric and then we have to require that $\phi$ is a $\textit{commutation factor}$ on $G$ that is a skew-symmetric bicharacter on $G$, i.e. that $\phi$ satisfies further $\phi(g,h)\phi(h,g)=1_{\Bbbk}$ for $g,h\in G$. We will usually work on the components of the grading and all maps will be understood to be extended by linearity. For the braiding we use the same notation of \cite{article5} and we write $c(x\otimes y)=\phi(|x|,|y|)y\otimes x$ with $x\in X$ and $y\in Y$, intending to work on homogeneous components and extend by linearity. Note that, given $X,Y$ and $Z$ in Vec$_{G}$, the hexagon relations
\[
(\text{Id}_{Y}\otimes c_{X,Z})\circ(c_{X,Y}\otimes\text{Id}_{Z})=c_{X,Y\otimes Z}\ \text{ and }\ (c_{X,Z}\otimes\text{Id}_{Y})\circ(\text{Id}_{X}\otimes c_{Y,Z})=c_{X\otimes Y,Z}
\]
on elements $x\in X$, $y\in Y$ and $z\in Z$, in terms of $\phi$, are exactly
\[
\phi(|x|,|y|)\phi(|x|,|z|)=\phi(|x|,|y||z|)=\phi(|x|,|y\otimes z|)\ \text{and}\ \phi(|x|,|z|)\phi(|y|,|z|)=\phi(|x||y|,|z|)=\phi(|x\otimes y|,|z|).
\]
Also note that, if $X$ and $Y$ are graded coalgebras, then 
\[
l_{X}\circ(\epsilon_{Y}\otimes\mathrm{Id}_{X})\circ c_{X,Y}=l_{X}\circ c_{X,\Bbbk}\circ(\mathrm{Id}_{X}\otimes\epsilon_{Y})=r_{X}\circ(\mathrm{Id}_{X}\otimes\epsilon_{Y})
\]
and $r_{Y}\circ(\text{Id}_{Y}\otimes\epsilon_{X})\circ c_{X,Y}=
l_{Y}\circ(\epsilon_{X}\otimes\mathrm{Id}_{Y})$ which on elements $x\in X$ and $y\in Y$ are 
\begin{equation}\label{to}
\phi(|x|,|y|)\epsilon(y)x=x\epsilon(y)\ \text{ and }\ \phi(|x|,|y|)y\epsilon(x)=\epsilon(x)y.
\end{equation}
Note that if $y\in Y_{g}$ with $g\not=1_{G}$ then $\epsilon(y)=0$ and if $y\in Y_{1_{G}}$ then $\phi(|x|,|y|)=1_{\Bbbk}$, so we still have that $\epsilon(y)x=x\epsilon(y)$ (as clearly it must be) but these relations will be useful in the computations. 
A graded algebra $A$ is \textit{commutative} if $ab=\phi(|a|,|b|)ba$ for every $a,b\in A$ and a graded coalgebra $C$ is \textit{cocommutative} if $x_{1}\otimes x_{2}=\phi(|x_{1}|,|x_{2}|)x_{2}\otimes x_{1}$ for every $x\in C$, where we shall adapt Sweedler notation and write $\Delta(x)=x_{1}\otimes x_{2}$ always assuming homogeneous terms in the sum. Note that, given $A$ and $B$ in Mon(Vec$_{G}$), the multiplication of $A\otimes B$ is given by $(a\otimes b)\cdot(c\otimes d)=\phi(|b|,|c|)ac\otimes bd$ and, given $C$ and $D$ in Comon(Vec$_{G}$), the comultiplication of $C\otimes D$ is given by $\Delta_{C\otimes D}(c\otimes d)=\phi(|c_{2}|,|d_{1}|)c_{1}\otimes d_{1}\otimes c_{2}\otimes d_{2}$. $\\$

The objects of the categories Bimon(Vec$_{G}$) and Hopf(Vec$_{G}$) are called $\textit{color bialgebras}$ and $\textit{color Hopf algebras}$, respectively. A color bialgebra is a datum $(B,m,u,\Delta,\epsilon)$ where $(B,m,u)$ is a graded algebra, $(B,\Delta,\epsilon)$ is a graded coalgebra, and the two structures are compatible in the sense that $\Delta$ and $\epsilon$ are graded algebra maps or, equivalently, $m$ and $u$ are graded coalgebra maps. Hence $B=\bigoplus_{g\in G}{B_{g}}$ is an ordinary algebra and an ordinary coalgebra with $m,u,\Delta,\epsilon$ which preserve gradings, but the condition of compatibility between the two structures differs from that in Bialg$_{\Bbbk}$, only for the part that involves the braiding. So we have that 
\[
\epsilon(ab)=\epsilon(a)\epsilon(b),\ \epsilon(1_{B})=1_{\Bbbk},\ \Delta(1_{B})=1_{B}\otimes1_{B}\ \text{and}\ \Delta(ab)=\phi(|a_{2}|,|b_{1}|)a_{1}b_{1}\otimes a_{2}b_{2}
\]
for every $a,b\in B$. A morphism of color bialgebras is a morphism of algebras and of coalgebras which preserves gradings. Given a color bialgebra $B$, a sub-bimonoid $B'\subseteq B$, called \textit{color sub-bialgebra}, will be a graded subalgebra which is also a graded subcoalgebra (the compatibility between the two structures is that of $B$). Furthermore, given a color bialgebra $B$ and a graded bi-ideal $I$ (which is a two-sided ideal and two-sided coideal) we know that $B/I$ is a graded algebra and a graded coalgebra and we show that the compatibility between the two structures is automatically maintained. In fact, given $\pi:B\to B/I$ the canonical projection, we have that
\[
\Delta_{B/I}\circ m_{B/I}\circ(\pi\otimes\pi)=\Delta_{B/I}\circ\pi\circ m_{B}=(\pi\otimes\pi)\circ\Delta_{B}\circ m_{B}
\]
and 
\[
\begin{split}
&(m_{B/I}\otimes m_{B/I})\circ(\text{Id}_{B/I}\otimes c_{_{B/I,B/I}}\otimes\text{Id}_{B/I})\circ(\Delta_{B/I}\otimes\Delta_{B/I})\circ(\pi\otimes\pi)=\\&(m_{B/I}\otimes m_{B/I})\circ(\text{Id}_{B/I}\otimes c_{_{B/I,B/I}}\otimes\text{Id}_{B/I})\circ(\pi\otimes\pi\otimes\pi\otimes\pi)\circ(\Delta_{B}\otimes\Delta_{B})=\\&(m_{B/I}\otimes m_{B/I})\circ(\pi\otimes\pi\otimes\pi\otimes\pi)\circ(\text{Id}_{B}\otimes c_{B,B}\otimes\text{Id}_{B})\circ(\Delta_{B}\otimes\Delta_{B})=\\&(\pi\otimes\pi)\circ(m_{B}\otimes m_{B})\circ(\text{Id}_{B}\otimes c_{B,B}\otimes\text{Id}_{B})\circ(\Delta_{B}\otimes\Delta_{B})=(\pi\otimes\pi)\circ\Delta_{B}\circ m_{B}
\end{split}
\]
since $c$ is natural and $B$ is a color bialgebra. Now, since $\pi\otimes\pi$ is surjective, we have that
\[
(m_{B/I}\otimes m_{B/I})\circ(\text{Id}_{B/I}\otimes c_{_{B/I,B/I}}\otimes\text{Id}_{B/I})\circ(\Delta_{B/I}\otimes\Delta_{B/I})=\Delta_{B/I}\circ m_{B/I},
\]
hence $B/I$ is a color bialgebra, called \textit{quotient color bialgebra}. $\\$

Given $(C,\Delta,\epsilon)\in\text{Comon(Vec}_{G})$ and $(A,m,u)\in\text{Mon(Vec}_{G})$, we have the convolution product of two morphisms $f,g:C\rightarrow A$ in Vec$_{G}$ given by $f*g:=m\circ(f\otimes g)\circ\Delta$. A color Hopf algebra is a color bialgebra with a morphism $S:B\rightarrow B$ in Vec$_{G}$ (antipode) such that $S*\text{Id}_{B}=u\circ\epsilon=\text{Id}_{B}*S$, thus it is a linear map preserving gradings such that $b_{1}S(b_{2})=\epsilon(b)1_{B}=S(b_{1})b_{2}$ for every $b\in B$. A morphism of color Hopf algebras is simply a morphism of color bialgebras since the compatibility with antipodes is automatically guaranteed (see e.g. \cite[Propositon 1.16]{l}). Given a color Hopf algebra $H$, a \textit{color Hopf subalgebra} $H'\subseteq H$ will be simply a color sub-bialgebra such that $S_{H}(H')\subseteq H'$. Furthermore, given a graded bi-ideal $I$ such that $S_{H}(I)\subseteq I$, there is a unique linear map $S_{H/I}:H/I\to H/I$ such that $S_{H/I}\circ\pi=\pi\circ S_{H}$ which preserves gradings since $S_{H}$ and $\pi$ do. This is clearly the antipode of $H/I$ (which is a color bialgebra), in fact as usual
\[
\begin{split}
m_{H/I}\circ(S_{H/I}\otimes\mathrm{Id}_{H/I})\circ\Delta_{H/I}\circ\pi&=m_{H/I}\circ(S_{H/I}\otimes\mathrm{Id}_{H/I})\circ(\pi\otimes\pi)\circ\Delta_{H}\\&=m_{H/I}\circ(\pi\otimes\pi)\circ(S_{H}\otimes\mathrm{Id}_{H})\circ\Delta_{H}\\&=\pi\circ m_{H}\circ(S_{H}\otimes\mathrm{Id}_{H})\circ\Delta_{H}=\pi\circ u_{H}\circ\epsilon_{H}\\&=u_{H/I}\circ\epsilon_{H/I}\circ\pi
\end{split}
\]
and from the surjectivity of $\pi$ we obtain $m_{H/I}\circ(S_{H/I}\otimes\mathrm{Id}_{H/I})\circ\Delta_{H/I}=u_{H/I}\circ\epsilon_{H/I}$. Analogously for the other equality, so $H/I$ is a color Hopf algebra, called \textit{quotient color Hopf algebra}. Observe that the properties of the antipode $S$ of a color Hopf algebra $H$ on elements $x,y\in H$ are: 
\[
S(xy)= \phi(|x|,|y|)S(y)S(x),\ S(1_{B})=1_{B}\ \text{ and }\ \Delta(S(x))=\phi(|x_{1}|,|x_{2}|)S(x_{2})\otimes S(x_{1}),\ \epsilon(S(x))=\epsilon(x).
\]
If $H$ is commutative then $S(xy)=S(x)S(y)$ and $S^{2}=\mathrm{Id}_{H}$ and if $H$ is cocommutative then $\Delta(S(x))=S(x_{1})\otimes S(x_{2})$ and $S^{2}=\mathrm{Id}_{H}$. $\\$

Clearly the category of Vec$_{\Bbbk}$ is exactly Vec$_{G}$ with $G=\{1_{G}\}$ the trivial group. Hence, motivated by the fact that Hopf$_{\Bbbk,\mathrm{coc}}$ is a semi-abelian category (\cite[Theorem 2.10]{article1}), our question is now to establish whether the category $\Hc$ is semi-abelian.

\section{Limits, Colimits and Protomodularity of $\Hc$}
In this section we show that $\Hc$ is pointed, finitely complete, cocomplete and protomodular. Clearly $\Bbbk$ with trivial grading is in $\Hc$ and it is a zero object of the category. In fact, given $H$ in $\Hc$, we have that $\epsilon$ is the unique morphism of coalgebras from $H$ to $\Bbbk$ and it is also of algebras and preserving gradings. Similarly $u$ is the unique morphism of algebras from $\Bbbk$ to $H$, also of coalgebras and preserving gradings. Hence $\Bbbk$ is a terminal and initial object in $\Hc$, so a zero object and $\Hc$ is pointed. Note that this is true also for Hopf(Vec$_{G}$) and Bimon(Vec$_{G}$) while, with the same reasoning, $\Bbbk$ is initial in Mon(Vec$_{G}$) and terminal in Comon(Vec$_{G}$). $\\$

Now we show the finite completeness of $\Hc$, by constructing equalizers and binary products and by using \cite[Proposition 2.8.2]{book1}. Note that these limits have the same form, as vector spaces, of those of Hopf$_{\Bbbk,\mathrm{coc}}$, given for instance in \cite{article11} (see also \cite{article3}). The constructions given for Hopf$_{\Bbbk,\mathrm{coc}}$ fit with this more general context and the naturality of the braiding or the fact that the category is symmetric is often required to check what appears immediate in the Hopf$_{\Bbbk,\mathrm{coc}}$ case. Since we have not seen these computations in literature for $\Hc$, we give the explicit constructions of these limits, also because they will be used in the following. 

\begin{oss}\label{o1}
Recall that, given a color Hopf algebra $A$ and a graded subspace $V$ of $A$, then $V$ is a color Hopf subalgebra of $A$ if it contains $1_{A}$ and it is closed under $m_{A}$, $\Delta_{A}$ and $S_{A}$. Observe that if $A$ is (co)commutative, clearly also $V$ is (co)commutative. 
\end{oss}

\subsection{Equalizers} Let $f,g:A\rightarrow B$ in $\Hc$, we can consider $$K=\{x\in A\ |\ (\text{Id}_{A}\otimes f)\Delta(x)=(\text{Id}_{A}\otimes g)\Delta(x)\}\subseteq A.$$ Observe that, as vector space, $K=\mathrm{ker}((\text{Id}_{A}\otimes f-\text{Id}_{A}\otimes g)\circ\Delta)$ and $(\text{Id}_{A}\otimes f-\text{Id}_{A}\otimes g)\circ\Delta$ is in $\mathrm{Vec}_{G}$ since $\Delta$, $\text{Id}_{A}\otimes f$ and $\text{Id}_{A}\otimes g$ are in Vec$_{G}$. 
\begin{invisible}
But then if $x\in K\subseteq A$ we have that 
\[
0=(\text{Id}_{A}\otimes f-\text{Id}_{A}\otimes g)\Delta(x)=\sum_{g\in G}{(\text{Id}_{A}\otimes f-\text{Id}_{A}\otimes g)\Delta(x_{g})}\in\bigoplus_{g\in G}{(A\otimes B)_{g}}
\]
and then, since the sum is direct, we must have $(\text{Id}_{A}\otimes f-\text{Id}_{A}\otimes g)\Delta(x_{g})=0$, i.e. $x_{g}\in K$. 
\end{invisible}
Thus, by Remark \ref{Im and qu}, $K$ is a graded subspace of $A$, i.e. $K=\bigoplus_{g\in G}{K_{g}}$ with $K_{g}=K\cap A_{g}$. By Remark \ref{o1} we have that if it is closed under the operations of $A$ it will be automatically in $\Hc$. Now clearly $1_{A}\in K$ and, given $x,y\in K$, then $xy\in K$ since $(\mathrm{Id}_{A}\otimes f)\circ\Delta$ and $(\mathrm{Id}_{A}\otimes g)\circ\Delta$ are morphisms of graded algebras. Indeed, given $x,y\in K$, we have that
\[
(\text{Id}_{A}\otimes f)\Delta(xy)=(\text{Id}_{A}\otimes f)\Delta(x)\cdot(\text{Id}_{A}\otimes f)\Delta(y)=(\text{Id}_{A}\otimes g)\Delta(x)\cdot(\text{Id}_{A}\otimes g)\Delta(y)=(\text{Id}_{A}\otimes g)\Delta(xy), 
\]
denoting by $\cdot$ the multiplication $m_{A\otimes B}$, hence $K$ is closed under $m_{A}$. Furthermore, since by cocommutativity of $A$ we have that $x_{1}\otimes x_{2}=\phi(|x_{1}|,|x_{2}|)x_{2}\otimes x_{1}$ for every $x\in A$ (and then for every $x\in K$), if $\Delta(x)\in A\otimes K$ then we obtain that
\[
K\otimes A\ni c_{A,K}(\phi(|x_{1}|,|x_{2}|)x_{2}\otimes x_{1})=\phi(|x_{1}|,|x_{2}|)\phi(|x_{2}|,|x_{1}|)x_{1}\otimes x_{2}=x_{1}\otimes x_{2}, 
\]
since $\phi$ is a commutation factor, thus we only have to show $\Delta(K)\subseteq A\otimes K$. But now we have that $K=\mathrm{ker}((\mathrm{Id}_{A}\otimes(f-g))\circ\Delta)$, thus 
\[
A\otimes K=\mathrm{ker}(\mathrm{Id}_{A}\otimes(\mathrm{Id}_{A}\otimes(f-g))\circ\Delta), 
\]
hence we desire to show that $x\in K$ implies that $(\mathrm{Id}_{A}\otimes(\mathrm{Id}_{A}\otimes(f-g))\Delta)\Delta(x)=0$. This is equivalent to show that $(\mathrm{Id}_{A}\otimes\mathrm{Id}_{A}\otimes f)(\mathrm{Id}_{A}\otimes\Delta)\Delta(x)=(\mathrm{Id}_{A}\otimes\mathrm{Id}_{A}\otimes g)(\mathrm{Id}_{A}\otimes\Delta)\Delta(x)$, i.e. $x_{1}\otimes x_{2_{1}}\otimes f(x_{2_{2}})=x_{1}\otimes x_{2_{1}}\otimes g(x_{2_{2}})$. But $x_{1}\otimes x_{2_{1}}\otimes f(x_{2_{2}})=x_{1_{1}}\otimes x_{1_{2}}\otimes f(x_{2})$ and $x_{1}\otimes x_{2_{1}}\otimes g(x_{2_{2}})=x_{1_{1}}\otimes x_{1_{2}}\otimes g(x_{2})$ by coassociativity and so we have to prove $(\Delta\otimes\mathrm{Id}_{B})(\mathrm{Id}_{A}\otimes f)\Delta(x)=(\Delta\otimes\mathrm{Id}_{B})(\mathrm{Id}_{A}\otimes g)\Delta(x)$ and this is true because $x\in K$. Hence $K$ is closed under $\Delta_{A}$. Furthermore, since $A$ is cocommutative we have $\Delta\circ S_{A}=(S_{A}\otimes S_{A})\circ\Delta$ and then, given $x\in K$, we have
\[
(\mathrm{Id}_{A}\otimes f)\Delta(S_{A}(x))=(\mathrm{Id}_{A}\otimes f)(S_{A}\otimes S_{A})\Delta(x)=(S_{A}\otimes S_{B})(\mathrm{Id}_{A}\otimes f)\Delta(x)=(S_{A}\otimes S_{B})(\mathrm{Id}_{A}\otimes g)\Delta(x)
\]
which is exactly $(\mathrm{Id}_{A}\otimes g)\Delta(S_{A}(x))$.
Thus we have that $K$ is in $\Hc$. Now, since the inclusion $i:K\rightarrow A$ is in Vec$_{G}$ and it is a morphism of algebras and coalgebras, we obtain that $(K,i)$ is the equalizer in $\Hc$ of the pair $(f,g)$. In fact, with $x\in K$, from $x_{1}\otimes f(x_{2})=x_{1}\otimes g(x_{2})$ we immediately obtain that $f(x)=g(x)$ and if $h:C\to A$ in $\Hc$ is such that $f\circ h=g\circ h$, then the image of $h$ is in $K$. Indeed, since $h$ is a morphism of coalgebras, we obtain 
\[
\begin{split}
(\mathrm{Id}_{A}\otimes f)\circ\Delta_{A}\circ h&=(\mathrm{Id}_{A}\otimes f)\circ(h\otimes h)\circ\Delta_{C}=(\mathrm{Id}_{A}\otimes g)\circ(h\otimes h)\circ\Delta_{C}=(\mathrm{Id}_{A}\otimes g)\circ\Delta_{A}\circ h. 
\end{split}
\]
We denote the equalizer of the pair $(f,g)$ in $\Hc$ by (\text{Eq}($f,g$), $i$).

\subsection{Binary Products} If we take $A,B$ in $\Hc$ we can consider $(A\otimes B,\pi_{A},\pi_{B})$ where 
\[
\pi_{A}:=r_{A}\circ(\text{Id}_{A}\otimes\epsilon_{B})\ \text{ and }\ \pi_{B}:=l_{B}\circ(\epsilon_{A}\otimes\text{Id}_{B}). 
\]
In particular $(A\otimes B,m,u,\Delta,\epsilon,S)$ is a cocommutative color Hopf algebra since Hopf(Vec$_{G}$) and Comon$_{\mathrm{coc}}$(Vec$_{G}$) have a monoidal structure with Vec$_{G}$ symmetric and we recall that $m=(m_{A}\otimes m_{B})\circ(\text{Id}_{A}\otimes c_{B,A}\otimes\text{Id}_{B})$, $u=(u_{A}\otimes u_{B})\circ r_{\Bbbk}^{-1}$, $\Delta=(\text{Id}_{A}\otimes c_{A,B}\otimes\text{Id}_{B})\circ(\Delta_{A}\otimes\Delta_{B})$, $\epsilon=r_{\Bbbk}\circ(\epsilon_{A}\otimes\epsilon_{B})$ and $S=S_{A}\otimes S_{B}$. Furthermore $\pi_{A}$ and $\pi_{B}$ are algebra maps and coalgebra maps and they preserve gradings since this is true for $r_{A}$, $l_{B}$ and $\epsilon_{A}$, $\epsilon_{B}$ and then they are morphisms in $\Hc$. We only have to prove that, for every $H$ in $\Hc$, we have a bijection between the set of morphisms in $\Hc$ from $H$ to $A\otimes B$ and the cartesian product of the set of morphisms from $H$ to $A$ and that of morphisms from $H$ to $B$ in $\Hc$. Given a map $f:H\rightarrow A\otimes B$ we can consider the pair $(\pi_{A}\circ f,\pi_{B}\circ f)$ and given a pair $(g,h)$, with $g:H\rightarrow A$ and $h:H\rightarrow B$, we can consider the morphism $(g\otimes h)\circ\Delta_{H}$; this map will be the diagonal morphism of the pair $(g,h)$, usually denoted by $\langle g,h\rangle$. It is in $\Hc$ since $\Delta_{H}$ is a morphism of coalgebras with $H$ cocommutative (and only in this case). Hence it is clear that this construction is specific for the cocommutative case. Clearly, given $g:H\rightarrow A$ and $h:H\rightarrow B$, we have
\[
\pi_{A}\circ(g\otimes h)\circ\Delta_{H}=r_{A}\circ(\mathrm{Id}_{A}\otimes\epsilon_{B})\circ(g\otimes h)\circ\Delta_{H}=r_{A}\circ(g\otimes\mathrm{Id}_{\Bbbk})\circ(\mathrm{Id}_{H}\otimes\epsilon_{H})\circ\Delta_{H}=g\circ r_{H}\circ(\mathrm{Id}_{H}\otimes\epsilon_{H})\circ\Delta_{H}=g
\]
and, analogously, $\pi_{B}\circ(g\otimes h)\circ\Delta_{H}=h$. On the other hand, given $f:H\rightarrow A\otimes B$, we have that
\[
((\pi_{A}\circ f)\otimes(\pi_{B}\circ f))\circ\Delta_{H}=(\pi_{A}\otimes\pi_{B})\circ(f\otimes f)\circ\Delta_{H}=(\pi_{A}\otimes\pi_{B})\circ\Delta_{A\otimes B}\circ f
\]
using $f$ of coalgebras. Now we show that $(\pi_{A}\otimes\pi_{B})\circ\Delta_{A\otimes B}=\text{Id}_{A\otimes B}$. Indeed, we have
\[
\begin{split}
(\pi_{A}\otimes\pi_{B})\circ\Delta_{A\otimes B}&=(\pi_{A}\otimes\pi_{B})\circ(\mathrm{Id}_{A}\otimes c_{A,B}\otimes\mathrm{Id}_{B})\circ(\Delta_{A}\otimes\Delta_{B})\\&=(r_{A}\otimes l_{B})\circ(\mathrm{Id}_{A}\otimes\epsilon_{B}\otimes\epsilon_{A}\otimes\mathrm{Id}_{B})\circ(\mathrm{Id}_{A}\otimes c_{A,B}\otimes\mathrm{Id}_{B})\circ(\Delta_{A}\otimes\Delta_{B})\\&=(r_{A}\otimes l_{B})\circ(\mathrm{Id}_{A}\otimes c_{\Bbbk,\Bbbk}\otimes\mathrm{Id}_{B})\circ(\mathrm{Id}_{A}\otimes\epsilon_{A}\otimes\epsilon_{B}\otimes\mathrm{Id}_{B})\circ(\Delta_{A}\otimes\Delta_{B})\\&=(r_{A}\otimes l_{B})\circ(\mathrm{Id}_{A}\otimes\epsilon_{A}\otimes\epsilon_{B}\otimes\mathrm{Id}_{B})\circ(\Delta_{A}\otimes\Delta_{B})=\mathrm{Id}_{A}\otimes\mathrm{Id}_{B}=\mathrm{Id}_{A\otimes B}
\end{split}
\]
where we only use the naturality of $c$ and the fact that $c_{\Bbbk,\Bbbk}=\mathrm{Id}_{\Bbbk,\Bbbk}$. Hence $(A\otimes B,\pi_{A},\pi_{B})$ is the binary product of $A$ and $B$ in $\Hc$ and we denote the object by $A\times B$. $\\$

We have obtained that $\Hc$ is finitely complete and now we show the cocompleteness. To this aim we prove more generally the cocompleteness of Hopf(Vec$_{G}$) by constructing coequalizers and arbitrary coproducts and that the colimits are the same in the cocommutative case. As for limits, also colimits have the same form as vector spaces of those in Hopf$_{\Bbbk}$, which are reported for instance in \cite{article4}. 

\begin{oss}
The fact that colimits are the same in the cocommutative case should not surprise us. In fact we recall that, given a symmetric monoidal category $\Mm$, the forgetful functor $U_{a}:\text{Mon}(\Mm)\to\Mm$ creates limits and the forgetful functor $U_{c}:\text{Comon}(\Mm)\to\Mm$ creates colimits and then Mon$(\Mm)$ is closed under limits in $\Mm$ as Comon$(\Mm)$ is closed under colimits in $\Mm$ (see e.g. \cite[Fact 10]{article9},\cite[Fact 4]{article10}). Hence also Mon$_{\mathrm{c}}(\Mm)$ is closed under limits in Mon$(\Mm)$ and Comon$_{\mathrm{coc}}(\Mm)$ is closed under colimits in Comon$(\Mm)$. Furthermore $\text{Bimon}_{\mathrm{coc}}(\Mm)=\text{Comon(Comon(Mon}(\Mm)))$, so Bimon$_{\mathrm{coc}}(\Mm)$ is closed under colimits in Bimon($\Mm$). We will see that colimits in Hopf(Vec$_{G}$) are the same of those in Bimon(Vec$_{G}$) and then clearly $\Hc$ is closed under colimits in Hopf(Vec$_{G}$). Observe also that Bimon(Vec$_{G}$)=Comon(Mon(Vec$_{G}$)) is closed under colimits in Mon(Vec$_{G}$) and then colimits in Hopf(Vec$_{G}$) will derive from those of Mon(Vec$_{G}$). However we show all the details in the sequel.
\end{oss}

\begin{oss}\label{oso}
Recall that, given a color Hopf algebra $H$ and a graded bi-ideal $I$ such that $S(I)\subseteq I$, then $H/I$ is a color Hopf algebra. Observe also that if $H$ is (co)commutative then also $H/I$ is (co)commutative. Indeed, for instance, if $H$ is cocommutative, by naturality of $c$, we have that 
\[
c_{H/I,H/I}\circ\Delta_{H/I}\circ\pi=c_{H/I,H/I}\circ(\pi\otimes\pi)\circ\Delta_{H}=(\pi\otimes\pi)\circ c_{H,H}\circ\Delta_{H}=(\pi\otimes\pi)\circ\Delta_{H}=\Delta_{H/I}\circ\pi
\]
and then, since $\pi:H\rightarrow H/I$ is surjective, we obtain $c_{H/I,H/I}\circ\Delta_{H/I}=\Delta_{H/I}$.
\end{oss}

\subsection{Coequalizers} Let $f,g:A\rightarrow B$ in Hopf(Vec$_{G}$), we can consider $I=B((f-g)(A))B$, the two-sided ideal of $B$ generated by the graded subspace of $B$ given by $(f-g)(A):=\{f(a)-g(a)\ |\ a\in A\}$, which is graded by Remark \ref{Im and qu}, since $I=m_{B}(m_{B}\otimes\mathrm{Id}_{B})(B\otimes(f-g)(A)\otimes B)$. Thus, in order to prove that $B/I$ is a color Hopf algebra, we only have to check that $I$ is a two-sided coideal and that $S(I)\subseteq I$, by Remark \ref{oso}. Given $a\in A$, since $f$ and $g$ are morphisms of coalgebras, we obtain
\[
\begin{split}
\Delta(f(a)-g(a))&=\Delta(f(a))-\Delta(g(a))=f(a_{1})\otimes f(a_{2})-g(a_{1})\otimes g(a_{2})\\&=f(a_{1})\otimes f(a_{2})-g(a_{1})\otimes f(a_{2})+g(a_{1})\otimes f(a_{2})-g(a_{1})\otimes g(a_{2})\\&=(f(a_{1})-g(a_{1}))\otimes f(a_{2})+g(a_{1})\otimes(f(a_{2})-g(a_{2})).
\end{split}
\]
Hence we have that $\Delta((f-g)(A))\subseteq(f-g)(A)\otimes B+B\otimes(f-g)(A)$ and from this, using that $\Delta$ is a morphism of algebras and that $B$ is a color bialgebra, we have that
\[
\begin{split}
    \Delta_{B}(B((f-g)(A)))&=\Delta_{B}(m_{B}(B\otimes(f-g)(A)))\\&=(m_{B}\otimes m_{B})(\text{Id}_{B}\otimes c_{B,B}\otimes\text{Id}_{B})(\Delta_{B}\otimes\Delta_{B})(B\otimes(f-g)(A))\\&\subseteq(m_{B}\otimes m_{B})(\text{Id}_{B}\otimes c_{B,B}\otimes\text{Id}_{B})(B\otimes B\otimes(f-g)(A)\otimes B)\\&+(m_{B}\otimes m_{B})(\text{Id}_{B}\otimes c_{B,B}\otimes\text{Id}_{B})(B\otimes B\otimes B\otimes(f-g)(A))\\&\subseteq B((f-g)(A))\otimes B+B\otimes B((f-g)(A))
\end{split}
\]
and then
\begin{invisible}
\[
\begin{split}
\Delta(B((f-g)(A))B)&=\Delta_{B}(m_{B}(B((f-g)(A))\otimes B))\\&=(m_{B}\otimes m_{B})(\mathrm{Id}_{B}\otimes c_{B,B}\otimes\mathrm{Id}_{B})(\Delta_{B}\otimes\Delta_{B})(B((f-g)(A))\otimes B)\\&\subseteq(m_{B}\otimes m_{B})(\mathrm{Id}_{B}\otimes c_{B,B}\otimes\mathrm{Id}_{B})(B((f-g)(A))\otimes B\otimes B\otimes B)\\&+(m_{B}\otimes m_{B})(\mathrm{Id}_{B}\otimes c_{B,B}\otimes\mathrm{Id}_{B})(B\otimes B((f-g)(A))\otimes B\otimes B)\\&\subseteq B((f-g)(A))B\otimes B+B\otimes B((f-g)(A))B 
\end{split}
\]
\end{invisible}
$\Delta(I)\subseteq I\otimes B+B\otimes I$. Furthermore $\epsilon(I)=0$ since $\epsilon$ is a morphism of algebras and thus $I$ is a two-sided coideal. Furthermore we have that 
\[
\begin{split}
S_{B}(I)&=S_{B}m_{B}(m_{B}\otimes\mathrm{Id}_{B})(B\otimes(f-g)(A)\otimes B)\\&=m_{B}c_{B,B}(S_{B}\otimes S_{B})(m_{B}\otimes\mathrm{Id}_{B})(B\otimes(f-g)(A)\otimes B)\\&=m_{B}c_{B,B}(m_{B}\otimes\mathrm{Id}_{B})(c_{B,B}\otimes\mathrm{Id}_{B})(S_{B}\otimes S_{B}\otimes S_{B})(B\otimes(f-g)(A)\otimes B)\\&\subseteq m_{B}c_{B,B}(m_{B}\otimes\mathrm{Id}_{B})(c_{B,B}\otimes\mathrm{Id}_{B})(B\otimes (f-g)(S_{A}(A))\otimes B)\\&\subseteq m_{B}c_{B,B}(m_{B}\otimes\mathrm{Id}_{B})(c_{B,B}\otimes\mathrm{Id}_{B})(B\otimes(f-g)(A)\otimes B)\subseteq I.
\end{split}
\]
Hence $B/I$ is a color Hopf algebra and $\pi:B\to B/I$ is in Hopf(Vec$_{G}$) such that clearly $\pi\circ f=\pi\circ g$. Now, given $h:B\rightarrow H$ in Hopf(Vec$_{G}$) such that $h\circ f=h\circ g$, we have that $I\subseteq\mathrm{ker}(h)$ and then there exists a unique morphism of coalgebras $h':B/I\rightarrow H$ such that $h'\circ\pi=h$ which is also of algebras and preserving grading since this is true for $\pi$ and $h$, hence it is the unique morphism in Hopf(Vec$_{G}$) such that $h'\circ\pi=h$. Thus $(B/I,\pi)$ is the coequalizer in Hopf(Vec$_{G}$) of the pair $(f,g)$, which we denote by $(\text{Coeq}(f,g),\pi)$. Observe here that, clearly, this is also the coequalizer for $f,g$ in $\Hc$ since if $B$ is cocommutative also $B/I$ is cocommutative, as said in Remark \ref{oso}.
\begin{invisible}
In fact, by naturality of $c$, we have that 
\[
c_{_{B/I,B/I}}\circ\Delta_{B/I}\circ\pi=c_{_{B/I,B/I}}\circ(\pi\otimes\pi)\circ\Delta_{B}=(\pi\otimes\pi)\circ c_{B,B}\circ\Delta_{B}=(\pi\otimes\pi)\circ\Delta_{B}=\Delta_{B/I}\circ\pi
\]
and then, since $\pi:B\rightarrow B/I$ is surjective, $c_{_{B/I,B/I}}\circ\Delta_{B/I}=\Delta_{B/I}$.
\end{invisible}

\begin{oss}\label{tensor}
We know that, given $V=\bigoplus_{g\in G}{V_{g}}$ in Vec$_{G}$ and $T^{p}(V)=V\otimes\cdot\cdot\cdot\otimes V$ $p$-times with $p\in\mathbb{N}$ (where $T^{0}(V)=\Bbbk$), we have that $T^{p}(V)$ is graded with $T^{p}(V)_{g}=\bigoplus_{h_{1}\cdot\cdot\cdot h_{p}=g}{V_{h_{1}}\otimes\cdot\cdot\cdot\otimes V_{h_{p}}}$ for every $g\in G$ and then we have 
\[
T(V)=\bigoplus_{p\in\mathbb{N}}{T^{p}(V)}=\bigoplus_{p\in\mathbb{N}}\bigoplus_{g\in G}{T^{p}(V)_{g}}=\bigoplus_{g\in G}\bigoplus_{p\in\mathbb{N}}{T^{p}(V)_{g}}=\bigoplus_{g\in G}{T(V)_{g}},
\]
so $T(V)$ is graded as vector space. But $T(V)$ is also an algebra and it is graded since $1_{\Bbbk}\in\Bbbk_{1_{G}}=T^{0}(V)_{1_{G}}\subseteq T(V)_{1_{G}}$ and given $x\in T(V)_{g}$ and $y\in T(V)_{h}$ with $g,h\in G$ we have that $xy\in T(V)_{gh}$. In fact, given an element $x=x_{h_{1}}\otimes\cdot\cdot\cdot\otimes x_{h_{p}}\in T^{p}(V)_{g}$ with $h_{1}\cdot\cdot\cdot h_{p}=g$ and an element $y=y_{k_{1}}\otimes\cdot\cdot\cdot\otimes y_{k_{s}}\in T^{s}(V)_{h}$ with $k_{1}\cdot\cdot\cdot k_{s}=h$, we have that $xy=x_{h_{1}}\otimes\cdot\cdot\cdot\otimes x_{h_{p}}\otimes y_{k_{1}}\otimes\cdot\cdot\cdot\otimes y_{k_{s}}\in T^{p+s}(V)_{gh}$ since $h_{1}\cdot\cdot\cdot h_{p}\cdot k_{1}\cdot\cdot\cdot k_{s}=gh$. Note also that the canonical inclusion $i:V\to T(V)$ preserves the grading since $i(V_{g})=V_{g}=T^{1}(V)_{g}\subseteq T(V)_{g}$.
\end{oss}

\subsection{Coproducts} Let $\{H_{l}\}_{l\in I}$ be a family of color Hopf algebras, we can take $T(\bigoplus_{l\in I}{H_{l}})/L$ where $L$ is the two-sided ideal in $T(\bigoplus_{l\in I}{H_{l}})$ generated by the linear span of the set
\[
J:=\{i(j_{l}(x_{l}y_{l}))-i(j_{l}(x_{l}))i(j_{l}(y_{l})), 1_{T(\bigoplus{H_{l}})}-i(j_{l}(1_{H_{l}}))\ |\ x_{l},y_{l}\in H_{l},\ l\in I\},
\]
where $j_{t}:H_{t}\to\bigoplus_{l\in I}{H_{l}}$ sends $v$ to the element with $v$ as $t$-component, the only one not trivial and $i:\bigoplus_{l\in I}{H_{l}}\to T(\bigoplus_{l\in I}{H_{l}})$ is the canonical inclusion.  Now, since $H_{l}=\bigoplus_{g\in G}{H_{l,g}}$ for every $l\in I$, we have that $\bigoplus_{l\in I}{H_{l}}=\bigoplus_{l\in I}{\bigoplus_{g\in G}{H_{l,g}}}=\bigoplus_{g\in G}{\bigoplus_{l\in I}{H_{l,g}}}$, then $\bigoplus_{l\in I}{H_{l}}$ is in Vec$_{G}$ and so $T(\bigoplus_{l\in I}{H_{l}})$ is a graded algebra by Remark \ref{tensor}. But now, clearly, $L$ is graded since it is generated by homogeneous elements; indeed $i$, $j_{l}$,  
$m_{l}$ and $m_{T(\bigoplus{H_{l}})}$ are in Vec$_{G}$, for every $l\in I$. Thus also $T(\bigoplus_{l\in I}{H_{l}})/L$ is a graded algebra. For all $l\in I$ define $q_{l}:=\nu\circ i\circ j_{l}$, where $\nu:T(\bigoplus_{l\in I}{H_{l}})\to T(\bigoplus_{l\in I}{H_{l}})/L$ is the canonical projection. Then $q_{l}$ is a morphism of algebras for every $l\in I$ by the relations of $J$ and since $\nu$ is an algebra map and it preserves gradings since this is true for the three maps. Now, given a graded algebra $C$ and graded algebra morphisms $g_{l}:H_{l}\to C$ for $l\in I$, there exists a unique linear map $k:\bigoplus_{i\in I}{H_{i}}\to C$ such that $k\circ j_{l}=g_{l}$ for every $l\in I$ by the universal property of the coproduct of vector spaces and $k$ also preserves gradings since $j_{l}$ and $g_{l}$ do (it is the universal property of the coproduct in Vec$_{G}$). By the universal property of the tensor algebra, there is a unique algebra map $s:T(\bigoplus_{l\in I}{H_{l}})\to C$ such that $s\circ i=k$ and $s$ also preserves gradings since $i$ and $k$ do. Finally, we have 
\[
s(i(j_{l}(x_{l}y_{l})))=k(j_{l}(x_{l}y_{l}))=g_{l}(x_{l}y_{l})=g_{l}(x_{l})g_{l}(y_{l})=s(i(j_{l}(x_{l})))s(i(j_{l}(y_{l}))) 
\]
and $s(i(j_{l}(1_{H_{l}})))=k(j_{l}(1_{H_{l}}))=g_{l}(1_{H_{l}})=1_{C}=s(1_{T(\bigoplus{H_{l}})})$ since $g_{l}$ and $s$ are algebra maps, for every $l\in I$. So $L\subseteq\mathrm{ker}(s)$ and then there exists a unique algebra map $p:T(\bigoplus_{l\in I}{H_{l}})/L\to C$ such that $p\circ\nu=s$ which preserves gradings since $s$ and $\nu$ do. We have that $p\circ q_{l}=g_{l}$ and this morphism $p$ is the unique in Mon(Vec$_{G}$) such that $p\circ q_{l}=g_{l}$ for every $l\in I$. Indeed, if there is a morphism $\tilde{p}:T(\bigoplus_{l\in I}{H_{l}})/L\to C$ such that $\tilde{p}\circ q_{l}=g_{l}$ for every $l\in I$, then from $(\tilde{p}\circ\nu\circ i)\circ j_{l}=g_{l}$ we obtain $(\tilde{p}\circ\nu)\circ i=k$, so $\tilde{p}\circ\nu=s$ and hence $p=\tilde{p}$. We have shown that $(T(\bigoplus_{l\in I}{H_{l}})/L,(q_{l})_{l\in I})$ is the coproduct of the family $\{H_{l}\}_{l\in I}$ in Mon(Vec$_{G}$), and we denote $T(\bigoplus_{l\in I}{H_{l}})/L$ by $\coprod_{l\in I}{H_{l}}$. \medskip

Now, since $H_{l}$ is a color bialgebra for every $l\in I$, we can show that $\coprod_{l\in I}{H_{l}}$ is a color bialgebra and that it is the coproduct of the family $\{H_{l}\}_{l\in I}$ in Bimon(Vec$_{G}$). The comultiplication and the counit are given by the unique graded algebra maps such that the following diagrams commute
\begin{equation}\label{sp}
\begin{tikzcd}
  H_{l} \arrow[r, "q_{l}"] \arrow["\Delta_{l}"',d]
    & \coprod_{i\in I}{H_{i}} \arrow[d, "\Delta"] \\
  H_{l}\otimes H_{l}\arrow[r,"q_{l}\otimes q_{l}"']
& \coprod_{i\in I}{H_{i}}\otimes\coprod_{i\in I}{H_{i}} 
\end{tikzcd}
\begin{tikzcd}
H_{l}\arrow[rd, "\epsilon_{l}"']\arrow[r,"q_{l}"]
& \coprod_{i\in I}{H_{i}} \arrow[d,"\epsilon"]\\
& \Bbbk
\end{tikzcd}
\end{equation}
by the universal property of the coproduct in Mon(Vec$_{G}$). Thus we already have the compatibility and, if we prove that $\Delta$ is coassociative and counitary we will have that $\Delta$ and $\epsilon$ make $\coprod_{i\in I}{H_{i}}$ a color bialgebra and the two commutative diagrams \eqref{sp} will prove that $q_{l}$ is a coalgebra map for every $l\in I$ and then a color bialgebra map. In order to obtain $(\text{Id}\otimes\Delta)\circ\Delta=(\Delta\otimes\text{Id})\circ\Delta$ it is sufficient to show that $(\text{Id}\otimes\Delta)\circ\Delta\circ q_{l}=(\Delta\otimes\text{Id})\circ\Delta\circ q_{l}$ by using the universal property, since $(\text{Id}\otimes\Delta)\circ\Delta$ and $(\Delta\otimes\text{Id})\circ\Delta$ are both graded algebra maps and, for the same argument, if we show $l\circ(\epsilon\otimes\text{Id})\circ\Delta\circ q_{l}=\text{Id}\circ q_{l}$ we obtain that $l\circ(\epsilon\otimes\text{Id})\circ\Delta=\text{Id}$, where $l:\Bbbk\otimes\coprod_{i\in I}{H_{i}}\rightarrow\coprod_{i\in I}{H_{i}}$ is the canonical isomorphism. So, having in mind the two diagrams in \eqref{sp} and the fact that $H_{l}$ is a coalgebra for $l\in I$, we obtain that
\[
\begin{split}
(\text{Id}\otimes\Delta)\circ\Delta\circ q_{l}=&(\text{Id}\otimes\Delta)\circ(q_{l}\otimes q_{l})\circ\Delta_{l}=(q_{l}\otimes q_{l}\otimes q_{l})\circ(\text{Id}\otimes\Delta_{l})\circ\Delta_{l}\\=&(q_{l}\otimes q_{l}\otimes q_{l})\circ(\Delta_{l}\otimes\text{Id})\circ\Delta_{l}=(\Delta\otimes\text{Id})\circ(q_{l}\otimes q_{l})\circ\Delta_{l}\\=&(\Delta\otimes\text{Id})\circ\Delta\circ q_{l}
\end{split}
\]
and 
\[
\begin{split}
l\circ(\epsilon\otimes\text{Id})\circ\Delta\circ q_{l}=&l\circ(\epsilon\otimes\text{Id})\circ(q_{l}\otimes q_{l})\circ\Delta_{l}=l\circ(\epsilon_{l}\otimes q_{l})\circ\Delta_{l}\\=&l\circ(\text{Id}\otimes q_{l})\circ(\epsilon_{l}\otimes\text{Id})\circ\Delta_{l}=q_{l}\circ l_{H_{l}}\circ(\epsilon_{l}\otimes\text{Id})\circ\Delta_{l}=\text{Id}\circ q_{l}.
\end{split}  
\]   
Similarly $r\circ(\text{Id}\otimes\epsilon)\circ\Delta=\text{Id}$ where $r:\coprod_{i\in I}{H_{i}}\otimes\Bbbk\rightarrow\coprod_{i\in I}{H_{i}}$ is the canonical isomorphism. Hence $\coprod_{l\in I}{H_{l}}$ is a color bialgebra and $q_{l}$ is a color bialgebra map for every $l\in I$. Now, given a color bialgebra $C$ and color bialgebra maps $g_{l}:H_{l}\rightarrow C$, we have a unique graded algebra map $p:\coprod_{i\in I}{H_{i}}\rightarrow C$ such that $p\circ q_{l}=g_{l}$ for every $l\in I$ by the universal property of the coproduct in Mon(Vec$_{G}$). We show that $p$ is also a coalgebra map in order to obtain that $(\coprod_{l\in I}{H_{l}},(q_{l})_{l\in I})$ is the coproduct in Bimon(Vec$_{G}$) of the family $\{H_{l}\}_{l\in I}$. By the argument used above, it is enough to show that $(p\otimes p)\circ\Delta\circ q_{l}=\Delta_{C}\circ p\circ q_{l}$ and $\epsilon_{C}\circ p\circ q_{l}=\epsilon\circ q_{l}$. So, since $g_{l}$ is a coalgebra map for every $l\in I$, we have that
\[
(p\otimes p)\circ\Delta\circ q_{l}=(p\otimes p)\circ(q_{l}\otimes q_{l})\circ\Delta_{l}=(g_{l}\otimes g_{l})\circ\Delta_{l}=\Delta_{C}\circ g_{l}=\Delta_{C}\circ p\circ q_{l}
\]
and $\epsilon_{C}\circ p\circ q_{l}=\epsilon_{C}\circ g_{l}=\epsilon_{l}=\epsilon\circ q_{l}$. 
\medskip

Now we let $H:=\coprod_{i\in I}{H_{i}}$. Every $H_{l}$ has an antipode $S_{l}:H_{l}\rightarrow H_{l}$ which is a color bialgebra map from $H_{l}$ to $H_{l}^{\mathrm{op,cop}}$ where $x\cdot_{\mathrm{op}}y:=mc_{H_{l},H_{l}}(x\otimes y)=\phi(|x|,|y|)yx$ and $\Delta^{\mathrm{cop}}(x):=c_{H_{l},H_{l}}\Delta(x)=\phi(|x_{1}|,|x_{2}|)x_{2}\otimes x_{1}$, for every $x,y\in H_{l}$. Since $q_{l}$ is a color bialgebra map from $H_{l}^{\mathrm{op,cop}}$ to $H^{\mathrm{op,cop}}$, the universal property of the coproduct in Bimon(Vec$_{G}$) yields a unique color bialgebra map $S:H\rightarrow H^{\text{op,cop}}$ such that the following diagram commutes for all $l\in I$.
\[
\begin{tikzcd}
  H_{l} \arrow[r, "q_{l}"] \arrow["S_{l}"',d]
    & H \arrow[d, "S"] \\
  H_{l}^{\text{op,cop}} \arrow[r,"q_{l}"']
& H^{\text{op,cop}} \end{tikzcd}
\]
If we prove that $S$ is the antipode of $H$, then $H$ is a color Hopf algebra and $q_{l}$ is a morphism of color Hopf algebras for every $l\in I$. Furthermore, given $C$ a color Hopf algebra and $g_{l}:H_{l}\rightarrow C$ a color Hopf algebra map for every $l\in I$, there is a unique color bialgebra map $t:H\rightarrow C$ (a posteriori the unique color Hopf algebra map) such that $t\circ q_{l}=g_{l}$ for every $l\in I$. Hence, in this case, $(H,(q_{l})_{l\in I})$ is the coproduct in Hopf(Vec$_{G}$) of the family of color Hopf algebras $\{H_{l}\}_{l\in I}$. Thus, in order to conclude, we prove that $m\circ(\text{Id}\otimes S)\circ\Delta=m\circ(S\otimes\text{Id})\circ\Delta=u\circ\epsilon$. Since $S:H\rightarrow H^{\text{op,cop}}$ is a color bialgebra map we only need to prove these on the generators of $H$ as a graded algebra. Indeed, let $h,k$ be generators in $H$ for which the relations hold, we obtain
\[
\begin{split}
    m(\mathrm{Id}_{H}\otimes S)\Delta(hk)&=m(\mathrm{Id}_{H}\otimes S)(\phi(|h_{2}|,|k_{1}|)h_{1}k_{1}\otimes h_{2}k_{2})=\phi(|h_{2}|,|k_{1}|)h_{1}k_{1}S(h_{2}k_{2})\\&=\phi(|h_{2}|,|k_{1}|)\phi(|h_{2}|,|k_{2}|)h_{1}k_{1}S(k_{2})S(h_{2})=\phi(|h_{2}|,|k|)h_{1}\epsilon(k)S(h_{2})\\&=h_{1}S(h_{2})\epsilon(k)=\epsilon(h)\epsilon(k)1_{H}=\epsilon(hk)1_{H}=u\epsilon(hk)
\end{split}
\]
and similarly $m(S\otimes\text{Id})\Delta(hk)=u\epsilon(hk)$, so the relations hold for $hk$ and thus for all the elements in $H$. So, having in mind that $H:=T(\bigoplus_{l\in I}{H_{l}})/L$, we only need to prove the relations for the elements $\bar{x}=i(x)+L\in H$ with $x\in\bigoplus_{l\in I}{H_{l}}$ whereas the tensor algebra $T(\bigoplus_{l\in I}{H_{l}})$ is the free algebra on $\bigoplus_{l\in I}{H_{l}}$. Moreover, since elements $x\in\bigoplus_{l\in I}{H_{l}}$ are such that $x_{l}=0$ for every $l\in I$ except for a finite number, by linearity it is enough to show that the relations hold for every $x_{l}\in H_{l}$ with $l\in I$. Using the commutativity of the three diagrams before, the fact that $H_{l}$ is a color Hopf algebra and that $q_{l}$ is an algebra map for $l\in I$, we obtain
\[
\begin{split}
    m\circ(\text{Id}\otimes S)\circ\Delta\circ q_{l}&=m\circ(\text{Id}\otimes S)\circ(q_{l}\otimes q_{l})\circ\Delta_{l}=m\circ(q_{l}\otimes q_{l})\circ(\text{Id}\otimes S_{l})\circ\Delta_{l}\\&=q_{l}\circ m_{l}\circ(\text{Id}\otimes S_{l})\circ\Delta_{l}=q_{l}\circ u_{l}\circ\epsilon_{l}=u\circ\epsilon_{l}=u\circ\epsilon\circ q_{l},
\end{split}
\]
hence $m\circ(\text{Id}\otimes S)\circ\Delta=u\circ\epsilon$. In the same way it can be shown that $m\circ(S\otimes\text{Id})\circ\Delta=u\circ\epsilon$. Thus $S$ is the antipode of $H$ and then $(H,(q_{l})_{l\in I})$ is the coproduct of the family $\{H_{l}\}_{l\in I}$ in Hopf(Vec$_{G}$). It is clear that if we consider $H_{l}$ a cocommutative color Hopf algebra for every $l\in I$ then $(H,(q_{l})_{l\in I})$ will be the coproduct in $\Hc$, since $H$ is cocommutative. In fact, since Vec$_{G}$ is symmetric, $c$ is a braiding for Mon(Vec$_{G}$) and, in particular, $c_{H,H}$ is a morphism of graded algebras, so the same is true for $c_{H,H}\circ\Delta$. Thus from
\[
c_{H,H}\circ\Delta\circ q_{l}=c_{H,H}\circ(q_{l}\otimes q_{l})\circ\Delta_{l}=(q_{l}\otimes q_{l})\circ c_{H_{l},H_{l}}\circ\Delta_{l}=(q_{l}\otimes q_{l})\circ\Delta_{l}=\Delta\circ q_{l}
\]
we obtain that $c_{H,H}\circ\Delta=\Delta$ by universal property of the coproduct. $\\$

Hence we have obtained that $\Hc$ (and also Hopf(Vec$_{G}$)) is cocomplete. Note that, even if only finite cocompleteness is required in the definition of a semi-abelian category, the fact that this category has all small colimits will be used to obtain that it is semi-abelian, through an equivalent characterization. 

\begin{oss}\label{lp}
Note that in \cite[Proposition 4.1.1]{article12} it has been proven that Hopf($\Mm$), $\mathrm{Hopf_{coc}}(\Mm)$ and $\mathrm{Hopf_{c}}(\Mm)$ are always accessible categories for every symmetric monoidal category $\Mm$. Hence we have that $\Hc$ is accessible and then, since we have shown that it is cocomplete, we obtain that it is complete and locally presentable. In fact we know that, as reported in \cite[Corollary 2.47]{book6}, a category is locally presentable if and only if it is accessible and complete if and only if it is accessible and cocomplete. Observe that, while accessibility is always true for the category of Hopf monoids in a symmetric monoidal category, this is not the same for local presentability. As it is said in \cite[Propositions 49,52,53]{article9} this is true when the forgetful functor $U_{a}:\mathrm{Mon}(\Mm)\to\Mm$ is an extremally monadic functor or when the forgetful functor $U_{c}:\mathrm{Comon}(\Mm)\to\Mm$ is an extremally comonadic functor, since in these cases we have that the category Hopf$(\Mm)$ is closed under colimits and limits in Bimon$(\Mm)$, respectively.
\end{oss}

\subsection{Protomodularity} Recall that if $\Mm$ is a category with binary products, i.e. there exists the binary product $A\times B$ for every objects $A$ and $B$ in $\Mm$, and with terminal object $\mathbf{I}$, the monoidal category $(\Mm,\times,\mathbf{I})$ is called \textit{cartesian} and the category of internal groups in $\Mm$, denoted by $\mathrm{Grp}(\Mm)$, has objects which are monoids $(G,m,u)$ in $\Mm$ equipped with a morphism $i:G\to G$ in $\Mm$ (called inversion) such that $m\circ\langle\mathrm{Id}_{G},i\rangle=u\circ t_{G}=m\circ\langle i,\mathrm{Id}_{G}\rangle$, where $t_{G}$ is the unique morphism from $G$ to $\mathbf{I}$ and $\langle\mathrm{Id}_{G},i\rangle$, $\langle i,\mathrm{Id}_{G}\rangle$ are the diagonal morphisms. $\\$

In \cite[Propositon 3.24]{article6} it is proved that, given a cartesian monoidal category $\Mm$ with finite limits, then the category $\mathrm{Grp}(\Mm)$ is protomodular. Note that the same terminal object, equalizers and binary products given before say that $\mathrm{Comon_{coc}}(\mathrm{Vec}_{G})$ is finitely complete. This category is also cartesian since its unit object $\Bbbk$ is the terminal object and the tensor product is the binary product and then we have that $\mathrm{Grp}(\mathrm{Comon_{coc}}(\mathrm{Vec_{G}}))$ is protomodular. Furthermore, as it is said for instance in \cite[Remark 3.3]{article12}, for every symmetric monoidal category $\Mm$ we have that Hopf$_{\mathrm{coc}}(\Mm)$=Grp(Comon$_{\mathrm{coc}}(\Mm)$) and then $\Hc=\mathrm{Grp}(\mathrm{Comon_{coc}}(\mathrm{Vec_{G}}))$. This is easy to show indeed
\[
\begin{split}
\text{Mon}(\text{Comon}_{\text{coc}}(\text{Vec}_{G}))&=\text{Mon(Comon(Comon}(\text{Vec}_{G})))=\text{Comon(Comon(Mon}(\text{Vec}_{G})))=\text{Bimon}_{\text{coc}}(\text{Vec}_{G})
\end{split}
\]
so monoids in $\mathrm{Comon_{coc}}(\mathrm{Vec}_{G})$ are given by cocommutative color bialgebras. Hence an object in $\mathrm{Grp}(\mathrm{Comon_{coc}}(\mathrm{Vec}_{G}))$ is a cocommutative color bialgebra $(B,m,u,\Delta,\epsilon)$ equipped with a morphism $i:B\to B$ in $\mathrm{Comon_{coc}}(\mathrm{Vec}_{G})$ such that $m\circ\langle\text{Id}_{B},i\rangle=u\circ t_{B}=m\circ\langle i,\text{Id}_{B}\rangle$. But we have seen before that $t_{B}=\epsilon$ and $\langle\text{Id}_{B},i\rangle=(\text{Id}_{B}\otimes i)\circ\Delta$, $\langle i,\text{Id}_{B}\rangle=(i\otimes\text{Id}_{B})\circ\Delta$, thus we obtain
\[
m\circ(\mathrm{Id}_{B}\otimes i)\circ\Delta=u\circ\epsilon=m\circ(i\otimes\mathrm{Id}_{B})\circ\Delta,
\]
so $i$ is the antipode of $B$. Hence we have that $\mathrm{Grp}(\mathrm{Comon_{coc}}(\mathrm{Vec}_{G}))$ is exactly the category $\Hc$. Thus we have that $\Hc$ is protomodular. Recall that here the fact that Vec$_{G}$ is symmetric ensures that $\text{Comon}_{\text{coc}}(\text{Vec}_{G})$ is monoidal.

\section{Regularity of $\Hc$}
The most delicate point is the regularity, as in the case of Hopf$_{\Bbbk,\mathrm{coc}}$. Following \cite{article1} the regularity will be shown through the following characterization:

\begin{lem}\label{l1}
Let $\Cc$ be a finitely complete category. Then $\Cc$ is a regular category if and only if $\\$ $\\$
(1) any arrow in $\Cc$ factors as a regular epimorphism followed by a monomorphism; $\\$ $\\$
(2) given any regular epimorphism $f:A\to B$ in $\Cc$ and any object $E$ in $\Cc$, the induced arrow $\mathrm{Id}_{E}\times f:E\times A\to E\times B$ is a regular epimorphism; $\\$ $\\$
(3) regular epimorphisms are stable under pullbacks along split monomorphisms.
\end{lem}
Since the zero morphism in $\Hc$ between $A$ and $B$ is $u_{B}\circ\epsilon_{A}$, the categorical kernel of $f:A\rightarrow B$ in $\Hc$, i.e. the equalizer of the pair $(f,u_{B}\circ\epsilon_{A})$, is given by $(\text{Hker}(f),j:\mathrm{Hker}(f)\to A)$ with 
\[
\mathrm{Hker}(f)=\{x\in A\ |\ (\mathrm{Id}_{A}\otimes f)\Delta(x)=(\mathrm{Id}_{A}\otimes u_{B}\epsilon_{A})\Delta(x)\}=\{x\in A\ |\ x_{1}\otimes f(x_{2})=x\otimes1_{B}\} 
\]
and $j$ is the canonical inclusion, while the categorical cokernel of $f$ in $\Hc$, i.e. the coequalizer of the pair $(f,u_{B}\circ\epsilon_{A})$, is given by $(B/I,\pi:B\rightarrow B/I$) where $$I=Bf(A)^{+}B$$ and $\pi$ is the canonical projection and where, for any coalgebra $C$, we write $C^{+}=\{x\in C\ |\ \epsilon(x)=0\}$. First note that for any coalgebra morphism $f:C\to D$, $f(C^{+})=f(C)^{+}$. Indeed, $y\in f(C)^{+}$ if and only if $y=f(x)$ for some $x\in C$ and $0=\epsilon_{D}(y)=\epsilon_{D}(f(x))=\epsilon_{C}(x)$ if and only if $y\in f(C^{+})$. Now $I$ is the two-sided ideal of $B$ generated by the set $\{f(a)-u_{B}\epsilon_{A}(a)\ |\ a\in A\}=\{f(a)-\epsilon_{A}(a)1_{B}\ |\ a\in A\}$. But now we have that $(f-u_{B}\circ\epsilon_{A})(A)=f(A^{+})=f(A)^{+}$ since if $x=f(a)-\epsilon_{A}(a)1_{B}$ we have $x=f(a-\epsilon_{A}(a)1_{A})$ with $\epsilon_{A}(a-\epsilon_{A}(a)1_{A})=0$ so that $(f-u_{B}\circ\epsilon_{A})(A)\subseteq f(A^{+})$ and the other inclusion is trivial. Now, given $f:A\rightarrow B$ in $\Hc$, we can consider the categorical cokernel of its categorical kernel in $\Hc$ that is given, as a map, by $p:A\rightarrow A/A(\text{Hker}(f))^{+}A$. Since $j$ is the kernel of $f$ we have that $f\circ j=u_{B}\circ\epsilon_{\mathrm{Hker}(f)}=f\circ u_{A}\circ\epsilon_{\mathrm{Hker}(f)}$, thus, by the universal property of the cokernel, there exists a unique morphism $i$ in $\Hc$ such that $f=i\circ p$. 
\[
\begin{tikzcd}
\mathrm{Hker}(f)\ar[r,hook,"j"]
  &
A \ar[r,"p"] \ar[dr,swap,"f"]
&
\frac{A}{A(\mathrm{Hker}(f))^{+}A} \ar[d,densely dotted,"i"]
\\
& & B
\end{tikzcd}
\]
If we show that $i$ is a monomorphism we obtain the decomposition regular epimorphism-monomorphism of $f$ in $\Hc$. In the case of Hopf$_{\Bbbk,\mathrm{coc}}$, Newman's Theorem \cite[Theorem 4.1]{article8} tells us that for a cocommutative Hopf algebra $H$ there is a bijective correspondence between the set of Hopf subalgebras of $H$ and that of left ideals which are also two-sided coideals of $H$: given a Hopf subalgebra $K$ of $H$ and a left ideal, two-sided coideal $I$ of $H$ the two maps are
\[
K\mapsto HK^{+} \text{ and } I\mapsto H^{\mathrm{co}\frac{H}{I}}:=\{x\in H\ |\ (\mathrm{Id}_{H}\otimes\pi)\Delta(x)=x\otimes\pi(1_{H})\}\ 
\]
which is also $\{x\in H\ |\ (\mathrm{Id}_{H}\otimes\pi)\Delta(x)=(\mathrm{Id}_{H}\otimes\pi u_{H}\epsilon_{H})\Delta(x)\}$, where $\pi:H\to H/I$ is the canonical projection and this result is used in \cite{article1} to deduce that the vector space ker$(f)$ is exactly $A(\text{Hker}(f))^{+}A$ and then that the morphism $i$ of the previous factorization is injective and so a monomorphism. We would like to obtain the same fact in the graded case.

\begin{oss}\label{ms}
Recall that given a graded algebra $A=\bigoplus_{g\in G}{A_{g}}$, i.e. an object in Mon(Vec$_{G}$), we can consider the category $_{A}\mathrm{Vec}_{G}$, whose objects are graded vector spaces $V=\bigoplus_{g\in G}{V_{g}}$ that are also left $A$-modules such that the left $A$-action $\mu:A\otimes V\to V$ is in Vec$_{G}$ and then $\mu(A_{g}\otimes V_{h})\subseteq V_{gh}$ for every $g,h\in G$ and whose morphisms are linear maps preserving grading which are also left $A$-linear. If $A$ is in Bimon($\mathrm{Vec}_{G}$), i.e. it is a color bialgebra, then the category $_{A}\mathrm{Vec}_{G}$ is monoidal with the same tensor product, unit object and constraints of $\mathrm{Vec}_{G}$ and then of $\mathrm{Vec}_{\Bbbk}$. Here the unit object $\Bbbk$ has left $A$-action such that $a\cdot k=\epsilon(a)k$ for $a\in A$ and $k\in\Bbbk$ and, given $V$ and $W$ in $_{A}\mathrm{Vec}_{G}$, $V\otimes W$ has left $A$-action given by $a\cdot(v\otimes w)=\phi(|a_{2}|,|v|)a_{1}\cdot v\otimes a_{2}\cdot w$, for $a\in A$, $v\in V$ and $w\in W$. With quotient color left $A$-module coalgebras we mean quotient objects in Comon($_{A}\mathrm{Vec}_{G}$), thus quotient graded vector spaces which are left $A$-modules with left $A$-action in Vec$_{G}$, which are also coalgebras with $\Delta$ and $\epsilon$ in $_{A}\mathrm{Vec}_{G}$; in particular, as coalgebras, they are quotients of a graded coalgebra with a graded two-sided coideal.
\end{oss}

Given $A$ in $\Hc$ we define in Vec$_{G}$ the morphism
\[
\xi_{A}:A\otimes A\to A,\ a\otimes x\mapsto\phi(|a_{2}|,|x|)a_{1}xS(a_{2}),
\]
i.e. $\xi_{A}=m_{A}\circ(m_{A}\otimes S_{A})\circ(\mathrm{Id}_{A}\otimes c_{A,A})\circ(\Delta_{A}\otimes\mathrm{Id}_{A})$, with $c$ the braiding of Vec$_{G}$. 
By analogy with Theorem \ref{ma}, we say that a color Hopf subalgebra $K\subseteq A$ is $\textbf{normal}$ if $\xi_{A}(A\otimes K)\subseteq K$. Now we show some properties of the map $\xi_{A}$.

\begin{oss}\label{f(A)}
If $f:A\to B$ is in $\Hc$ then $f(A)$ is in $\Hc$. In fact, by Remark \ref{Im and qu}, we know that $f(A)$ is a graded subspace of $B$ and, as in the usual case, it contains $1_{B}=f(1_{A})$ and it is closed under $m_{B}$ ($m_{B}\circ(f\otimes f)=f\circ m_{A}$) since $f$ is an algebra map, it is closed under $\Delta_{B}$ since $f$ is a coalgebra map ($\Delta_{B}\circ f=(f\otimes f)\circ\Delta$)  and under $S_{B}$ (since $S_{B}\circ f=f\circ S_{A}$). So $f(A)$ is in $\Hc$ by Remark \ref{o1}. Similarly $\mathrm{ker}(f)$, which is graded by Remark \ref{Im and qu}, is a two-sided ideal of $A$ (since $f$ is an algebra map), a two-sided coideal of $A$ (since $f$ is a coalgebra map) and it is closed under $S_{A}$, so that $A/\mathrm{ker}(f)$ is in $\Hc$ by Remark \ref{oso}. 
\end{oss}

\begin{lem}\label{xi}
Let $A$ and $B$ in $\Hc$. Then the following properties hold: $\\$ $\\$
1) $\xi_{A}$ is a morphism of coalgebras. $\\$ $\\$
2) Given $p:A\to B$ in $\Hc$ then $\xi_{B}\circ(p\otimes p)=p\circ\xi_{A}$. As a consequence with $p$ surjective, if $D$ is a normal color Hopf subalgebra of $A$ then $p(D)$ is a normal color Hopf subalgebra of $B$. $\\$ $\\$
3) $A$ is commutative if and only if $\xi_{A}=l_{A}\circ(\epsilon_{A}\otimes\mathrm{Id}_{A})$. As a consequence if $A$ is commutative every $K\subseteq A$ color Hopf subalgebra is normal.
\end{lem}

\begin{proof}
In order to prove 1) we have to show that $\Delta_{A}\circ\xi_{A}=(\xi_{A}\otimes\xi_{A})\circ\Delta_{A\otimes A}$ where $\Delta_{A\otimes A}=(\mathrm{Id}_{A}\otimes c_{A,A}\otimes\mathrm{Id}_{A})\circ(\Delta_{A}\otimes\Delta_{A})$ and $\epsilon_{A}\circ\xi_{A}=\epsilon_{A\otimes A}$ where $\epsilon_{A\otimes A}=r_{\Bbbk}\circ(\epsilon_{A}\otimes\epsilon_{A})$. Given $a,x\in A$, since $\Delta$ is a morphism of graded algebras and $A$ is cocommutative, we have that 
\[
\begin{split}
\Delta_{A}\xi_{A}(a\otimes x)&=\Delta_{A}(\phi(|a_{2}|,|x|)a_{1}xS(a_{2}))=\phi(|a_{2}|,|x|)(m_{A}\otimes m_{A})(\mathrm{Id}\otimes c_{A,A}\otimes\mathrm{Id})(\Delta_{A}(a_{1}x)\otimes\Delta_{A}(S(a_{2})))\\&=\phi(|a_{2_{1}}||a_{2_{2}}|,|x_{1}||x_{2}|)(m_{A}\otimes m_{A})(\mathrm{Id}\otimes c_{A,A}\otimes\mathrm{Id})(\phi(|a_{1_{2}}|,|x_{1}|)a_{1_{1}}x_{1}\otimes a_{1_{2}}x_{2}\otimes S(a_{2_{1}})\otimes S(a_{2_{2}}))\\&=\phi(|a_{2_{1}}||a_{2_{2}}|,|x_{1}||x_{2}|)\phi(|a_{1_{2}}|,|x_{1}|)\phi(|a_{1_{2}}||x_{2}|,|a_{2_{1}}|)a_{1_{1}}x_{1}S(a_{2_{1}})\otimes a_{1_{2}}x_{2}S(a_{2_{2}})\\&=\phi(|a_{2_{1}}|,|x_{1}|)\phi(|a_{2_{2}}|,|x_{1}|)\phi(|a_{2_{2}}|,|x_{2}|)\phi(|a_{1_{2}}|,|x_{1}|)\phi(|a_{1_{2}}|,|a_{2_{1}}|)a_{1_{1}}x_{1}S(a_{2_{1}})\otimes a_{1_{2}}x_{2}S(a_{2_{2}})\\&\overset{(*)}{=}\phi(|a_{2}|,|x_{1}|)\phi(|a_{2_{2}}|,|x_{2}|)\phi(|a_{1_{2}}|,|x_{1}|)a_{1_{1}}x_{1}S(a_{1_{2}})\otimes a_{2_{1}}x_{2}S(a_{2_{2}})\\&=(\xi_{A}\otimes\xi_{A})(\phi(|a_{2}|,|x_{1}|)a_{1}\otimes x_{1}\otimes a_{2}\otimes x_{2})=(\xi_{A}\otimes\xi_{A})\Delta_{A\otimes A}(a\otimes x),
\end{split}
\]
where $(*)$ follows since $A$ is cocommutative and then $\Delta_{A}$ is of graded coalgebras, i.e. $(\Delta_{A}\otimes\Delta_{A})\circ\Delta_{A}=\Delta_{A\otimes A}\circ\Delta_{A}=(\mathrm{Id}_{A}\otimes c_{A,A}\otimes\mathrm{Id}_{A})\circ(\Delta_{A}\otimes\Delta_{A})\circ\Delta_{A}$, which on $a\in A$ is $a_{1_{1}}\otimes a_{1_{2}}\otimes a_{2_{1}}\otimes a_{2_{2}}=\phi(|a_{1_{2}}|,|a_{2_{1}}|)a_{1_{1}}\otimes a_{2_{1}}\otimes a_{1_{2}}\otimes a_{2_{2}}$. Furthermore we have that
\[
\begin{split}
\epsilon_{A}(\xi_{A}(a\otimes x))&=\epsilon_{A}(\phi(|a_{2}|,|x|)a_{1}xS(a_{2}))=\phi(|a_{2}|,|x|)\epsilon_{A}(a_{1})\epsilon_{A}(x)\epsilon_{A}(S(a_{2}))\\&=\epsilon_{A}(a_{1})\epsilon_{A}(a_{2})\epsilon_{A}(x)=\epsilon_{A}(a)\epsilon_{A}(x)=\epsilon_{A\otimes A}(a\otimes x).
\end{split}
\]
Now, since $p$ is a morphism of algebras and of coalgebras, we have that
\[
\begin{split}
\xi_{B}(p(a)\otimes p(x))&=\phi(|p(a_{2})|,|p(x)|)p(a_{1})p(x)S(p(a_{2}))=\phi(|a_{2}|,|x|)p(a_{1}x)p(S(a_{2}))\\&=p(\phi(|a_{2}|,|x|)a_{1}xS(a_{2}))=p(\xi_{A}(a\otimes x)).
\end{split}
\]
Suppose that $p$ is surjective and that $D$ is a normal color Hopf subalgebra of $A$. We already know that $p(D)$ is a color Hopf subalgebra of $B$ by Remark \ref{f(A)}, we have to show that it is normal, i.e. that, given $b\in B$ and $d\in D$, then $\xi_{B}(b\otimes p(d))\in p(D)$. By surjectivity of $p$ there exists $a\in A$ such that $p(a)=b$. Hence, since $D$ is normal, we obtain
\[
\xi_{B}(b\otimes p(d))=\xi_{B}(p\otimes p)(a\otimes d)=p\xi_{A}(a\otimes d)\in p(D),
\]
thus $p(D)$ is normal in $B$ and then also 2) is proved. Finally if we suppose that $A$ is commutative clearly we obtain that $\xi_{A}(a\otimes x)=\phi(|a_{2}|,|x|)a_{1}xS(a_{2})=a_{1}S(a_{2})x=\epsilon(a)x$, so that $\xi_{A}(A\otimes K)\subseteq K$ for every $K\subseteq A$ color Hopf subalgebra, while if $\phi(|a_{2}|,|x|)a_{1}xS(a_{2})=\epsilon(a)x$ then we have that 
\[
x\epsilon(a)=\phi(|x|,|a|)\epsilon(a)x=\phi(|x|,|a_{1}|)\phi(|x|,|a_{2}|)\phi(|a_{2}|,|x|)a_{1}xS(a_{2})=\phi(|x|,|a_{1}|)a_{1}xS(a_{2}),
\]
hence
\[
x\otimes a=x\otimes\epsilon(a_{1})a_{2}=x\epsilon(a_{1})\otimes a_{2}=\phi(|x|,|a_{1_{1}}|)a_{1_{1}}xS(a_{1_{2}})\otimes a_{2}=\phi(|x|,|a_{1}|)a_{1}xS(a_{2_{1}})\otimes a_{2_{2}}
\]
and then
\[
xa=\phi(|x|,|a_{1}|)a_{1}xS(a_{2_{1}})a_{2_{2}}=\phi(|x|,|a_{1}|)a_{1}x\epsilon(a_{2})=\phi(|x|,|a_{1}|)\phi(|x|,|a_{2}|)a_{1}\epsilon(a_{2})x=\phi(|x|,|a|)ax,
\]
thus $xa=\phi(|x|,|a|)ax$ and then $A$ is commutative and also 3) is shown. 
\end{proof}

\begin{lem}\label{31}
Let $i:B\to A$ be an inclusion in $\Hc$. If $i$ is the categorical kernel of some morphism in $\Hc$, then $B$ is a normal color Hopf subalgebra of $A$.
\end{lem}

\begin{proof}
    Suppose that $B=\mathrm{Hker}(f)$ for some morphism $f:A\to C$ in $\Hc$. We already know that $B$ is a color Hopf subalgebra of $A$, we have to prove that it is normal, i.e. that, given $x\in\mathrm{Hker}(f)$ and $a\in A$, then $\xi_{A}(a\otimes x)\in\mathrm{Hker}(f)$, i.e. $(\mathrm{Id}\otimes f)\Delta_{A}\xi_{A}(a\otimes x)=\xi_{A}(a\otimes x)\otimes1_{C}$. But now $f(x)=\epsilon(x)1_{C}$ with $x\in\mathrm{Hker}(f)$ and, since $f$ is a morphism of graded algebras, we obtain
\[
\begin{split}
f\xi_{A}(a\otimes x)&=\phi(|a_{2}|,|x|)f(a_{1})f(x)f(S(a_{2}))=\phi(|a_{2}|,|x|)f(a_{1})\epsilon(x)f(S(a_{2}))=f(a_{1})f(S(a_{2}))\epsilon(x)\\&=f(a_{1}S(a_{2}))\epsilon(x)=f(\epsilon(a)1_{A})\epsilon(x)=\epsilon(a)\epsilon(x)1_{C}=u_{C}\epsilon_{A\otimes A}(a\otimes x)
\end{split}
\]
and then, using 1) of Lemma \ref{xi}, we have that
\[
\begin{split}
    (\mathrm{Id}\otimes f)\Delta_{A}\xi_{A}(a\otimes x)&=(\mathrm{Id}\otimes f)(\xi_{A}\otimes\xi_{A})\Delta_{A\otimes A}(a\otimes x)=(\xi_{A}\otimes u_{C})(\mathrm{Id}\otimes\mathrm{Id}\otimes\epsilon_{A\otimes A})\Delta_{A\otimes A}(a\otimes x)\\&=(\xi_{A}\otimes u_{C})(a\otimes x\otimes1_{\Bbbk})=\xi_{A}(a\otimes x)\otimes1_{C},
\end{split}
\]
thus $\mathrm{Hker}(f)$ is normal.
\end{proof}

A generalization of Newman's Theorem for the category Hopf$_{\mathrm{coc}}(\mathrm{Vec}_{\mathbb{Z}_2})$ of cocommutative super Hopf algebras is  proved by A. Masuoka in the case char$\Bbbk\not=2$. The result is the following:

\begin{thm}\label{ma}$($c.f. \cite[Theorem 3.10 (3)]{article2}$)$
Let $H$ be a cocommutative super Hopf algebra. Then the super Hopf subalgebras $K\subseteq H$ and the quotient super left $H$-module coalgebras $Q$ of $H$ are in 1-1 correspondence, under $K\mapsto H/HK^{+}$, $Q\mapsto{}^{\mathrm{co}Q}H(=H^{\mathrm{co}Q})$. This restricts to a 1-1 correspondence between those super Hopf subalgebras $K$ which are normal in the sense that $(-1)^{|h_{2}||x|}h_{1}xS(h_{2})\in K$ for every $h\in H$ and $x\in K$ and the quotient super Hopf algebras.
\end{thm}

We call the bijections in analogy with those given for Newman's Theorem in \cite{article1}, i.e. 
\[
\phi_{H}:K\mapsto H/HK^{+}\ \text{  and  }\ \psi_{H}:Q\mapsto{}^{\mathrm{co}Q}H(=H^{\mathrm{co}Q}),
\]
where $H^{\mathrm{co}Q}$ is defined as before. Observe that last statement in Theorem \ref{ma} is a generalization of the equivalence between (1) and (2) of \cite[Corollary 2.3]{article1}. Here we obtain immediately a complete generalization of \cite[Corollary 2.3]{article1} for cocommutative super Hopf algebras.

\begin{cor}\label{c1}
For a super Hopf subalgebra $B\subseteq A$ of a cocommutative super Hopf algebra $A$, the following conditions are equivalent: $\\$ $\\$
(1) $B$ is a normal super Hopf subalgebra; $\\$ $\\$
(2) $A/AB^{+}$ is a quotient super Hopf algebra; $\\$ $\\$
(3) the inclusion morphism $B\to A$ is the categorical kernel of some morphism in $\mathrm{Hopf_{coc}}(\mathrm{Vec}_{\mathbb{Z}_{2}})$.
\end{cor}

\begin{proof}
We already know that (1) and (2) are equivalent by Theorem \ref{ma}. $\\$ 
(2)$\implies$(3). Since $A/AB^{+}$ is a quotient super Hopf algebra, the canonical projection $\pi:A\to A/AB^{+}$ is a morphism of cocommutative super Hopf algebras and then clearly $A^{\mathrm{co}\frac{A}{AB^{+}}}$ is exactly $\mathrm{Hker}(\pi)$ since $x\otimes\pi(1_{A})=x\otimes1_{A/AB^{+}}$, for $x\in A$. But now, using Theorem \ref{ma}, we obtain
\[
\mathrm{Hker}(\pi)=A^{\mathrm{co}\frac{A}{AB^{+}}}=\psi_{A}(A/AB^{+})=\psi_{A}(\phi_{A}(B))=B.
\]
Hence $(B,j)$ is the kernel of $\pi$ in $\mathrm{Hopf_{coc}}(\mathrm{Vec}_{\mathbb{Z}_{2}})$, where $j:B\to A$ is the canonical inclusion. 
We already know that (3)$\implies$(1) by Lemma \ref{31} and then we are done.
\begin{invisible}
Suppose that $B=\mathrm{Hker}(f)$ for some morphism $f:A\to C$ in $\mathrm{Hopf_{coc}}(\mathrm{Vec}_{\mathbb{Z}_{2}})$. We already know that it is a super Hopf subalgebra of $A$, we have to prove that it is normal, i.e. that, given $x\in\mathrm{Hker}(f)$ and $a\in A$, then $\xi_{A}(a\otimes x)\in\mathrm{Hker}(f)$, i.e. $(\mathrm{Id}\otimes f)\Delta_{A}\xi_{A}(a\otimes x)=\xi_{A}(a\otimes x)\otimes1_{C}$. But now $f(x)=\epsilon(x)1_{C}$ with $x\in\mathrm{Hker}(f)$ and, since $f$ is a morphism of super algebras, we obtain
\[
\begin{split}
f\xi_{A}(a\otimes x)&=(-1)^{|a_{2}||x|}f(a_{1})f(x)f(S(a_{2}))=(-1)^{|a_{2}||x|}f(a_{1})\epsilon(x)f(S(a_{2}))=f(a_{1})f(S(a_{2}))\epsilon(x)\\&=f(a_{1}S(a_{2}))\epsilon(x)=f(\epsilon(a)1_{A})\epsilon(x)=\epsilon(a)\epsilon(x)1_{C}=u_{C}\epsilon_{A\otimes A}(a\otimes x)
\end{split}
\]
and then, using 1) of Lemma \ref{xi}, we have that
\[
\begin{split}
    (\mathrm{Id}\otimes f)\Delta_{A}\xi_{A}(a\otimes x)&=(\mathrm{Id}\otimes f)(\xi_{A}\otimes\xi_{A})\Delta_{A\otimes A}(a\otimes x)=(\xi_{A}\otimes u_{C})(\mathrm{Id}\otimes\mathrm{Id}\otimes\epsilon_{A\otimes A})\Delta_{A\otimes A}(a\otimes x)\\&=(\xi_{A}\otimes u_{C})(a\otimes x\otimes1_{\Bbbk})=\xi_{A}(a\otimes x)\otimes1_{C},
\end{split}
\]
thus $\mathrm{Hker}(f)$ is normal.
\end{invisible}
\end{proof}
\begin{invisible}
\begin{oss}
It is clear that $(3)\implies(1)$ holds also for color cocommutative Hopf algebras with the same proof and an arbitrary commutation factor $\phi$ on $G$ instead of that on $\mathbb{Z}_2$, since the same is true for Lemma \ref{xi}.
\end{oss}
\end{invisible}

We will obtain a generalization of Theorem \ref{ma} and of Corollary \ref{c1} for $\Hc$ that will be used for the regularity and the semi-abelian condition of $\Hc$.

\subsection{From color Hopf algebras to super Hopf algebras}

In order to use Theorem \ref{ma} we are interested in obtaining a braided strong monoidal functor from the category Vec$_{G}$ to the category Vec$_{\mathbb{Z}_2}$. In this subsection $G$ and $L$ will denote arbitrary abelian groups. 

\begin{oss}\label{F}
 As it is said in \cite[Example 2.5.2]{po}, given $f:G\to L$ a morphism of groups, any $G$-graded vector space is naturally $L$-graded (by pushforward of grading) and we have a natural strict monoidal functor $(F,\phi^0,\phi^2):\mathrm{Vec}_{G}\to\mathrm{Vec}_{L}$ (also denoted by $f_{*}$). The functor $F:\mathrm{Vec}_{G}\to\mathrm{Vec}_{L}$ is defined, given $V=\bigoplus_{g\in G}{V_{g}}$ and $f$ in Vec$_{G}$, such that  
\[
F(\bigoplus_{g\in G}{V_{g}})=\bigoplus_{f(g)\in L}{V_{f(g)}}\ \text{ where }\ V_{f(g)}=\bigoplus_{g'\in f^{-1}f(g)}{V_{g'}}\ \text{ and }\ F(f)=f,
\]
so $V_{f(g)}$ is the direct sum of all the $V_{g'}$'s such that $f(g')$ is the same element $f(g)$ in $L$ and then $F(V)$ is still $V$ as vector space but with a grading over $L$ in which $V_{l}=\{0\}$ if $l\notin\mathrm{Im}(f)$. Observe that $F(\Bbbk)=\Bbbk$ with $\Bbbk_{1_{L}}=\Bbbk$ and $\Bbbk_{l}=0$ if $l\not=1_{L}$, so that one can define $\phi^{0}:=\mathrm{Id}_{\Bbbk}$ and, given $V,W$ in Vec$_{G}$, we have  
\[
\begin{split}
F(V\otimes W)&=F(\bigoplus_{g\in G}{(V\otimes W)_{g}})=\bigoplus_{f(g)\in L}{(V\otimes W)_{f(g)}}=\bigoplus_{f(g)\in L}{V_{f(g)}}\otimes\bigoplus_{f(g)\in L}{W_{f(g)}}
\end{split}
\]
which is $F(V)\otimes F(W)$, so $F(V\otimes W)$ and $F(V)\otimes F(W)$ are the same $L$-graded vector space and then one can define $\phi^{2}_{V,W}:=\mathrm{Id}$ for every $V,W$ in $\mathrm{Vec}_{G}$. Clearly this remark is true also for groups $G$ and $L$ not abelian in which case Vec$_{G}$ and Vec$_{L}$ are not braided.
\end{oss}

In \cite[Remark 1.2]{article5} it is said how to obtain a braided strong monoidal functor from Vec$_{G}$ to Vec$_{L}$ when $G$ and $L$ are finite abelian groups and it is not difficult to see that this works also in the case $G$ and $L$ are not necessarily finite. We recall here how to do it. Clearly, if we define $\phi^2_{V,W}:=\mathrm{Id}$ for every $V$ and $W$ in Vec$_{G}$, we can not obtain in general a braided monoidal functor from Vec$_{G}$ to Vec$_{L}$ since the braiding of Vec$_{G}$ and that of Vec$_{L}$ are different. Thus we define $\phi^0:=\mathrm{Id}_{\Bbbk}$ but we modify the morphisms $\phi^2_{V,W}$ that we want to be isomorphisms in Vec$_{L}$ in order to have a strong monoidal functor and we recall that $F(V\otimes W)=F(V)\otimes F(W)$ in Vec$_{L}$.
Given a map $\gamma:G\times G\to\Bbbk-\{0\}$, one can define, for every $V$ and $W$ in Vec$_{G}$, isomorphisms in Vec$_{G}$ given by $f_{V,W}:V\otimes W\to V\otimes W$, $v\otimes w\mapsto\gamma(g,h)v\otimes w$, for $v\in V_{g}$ and $w\in W_{h}$ and $g,h\in G$, defined on the components of the grading and extended by linearity. We define $\phi^2_{V,W}:=F(f_{V,W})=f_{V,W}$, which are isomorphisms in Vec$_{L}$ for every $V$ and $W$ in Vec$_{G}$. In order to obtain a monoidal functor we need that
\begin{invisible}
Observe that to give isomorphisms $\phi^{2}_{V,W}:F(V\otimes W)\to F(V)\otimes F(W)$ in Vec$_{L}$ for every $V$ and $W$ in Vec$_{G}$ is equivalent to give isomorphisms of vector spaces between $F(V\otimes W)_{l}$ and $(F(V)\otimes F(W))_{l}$ for every $l\in L$ (which are trivial if $l\notin\mathrm{Im}(f)$) and then to give isomorphisms between $F(V_{g}\otimes W_{h})$ and $F(V_{g})\otimes F(W_{h})$ for every $g,h\in G$ with $f(gh)=l$ and every $l\in L$ and then for every $g,h\in G$. Hence we need isomorphisms $\phi^{2}_{V_{g},W_{h}}:F(V_{g}\otimes W_{h})\to F(V_{g})\otimes F(W_{h})$ for every $g,h\in G$. In \cite[Remark 1.2]{article5} it is said how to obtain a braided strong monoidal functor when $G$ and $L$ are finite abelian groups; we sum up the result in Lemma \ref{braided} of which we show the proof by completeness, in the case $G$ and $L$ are abelian groups not necessarily finite.
\end{invisible}
$\gamma$ is a 2-\textit{cocycle} on $G$, i.e. that it satisfies 
\begin{equation}\label{mon1}
\gamma(gh,k)\gamma(g,h)=\gamma(g,hk)\gamma(h,k)\ \text{ for every }\ g,h,k\in G.
\end{equation}
We immediately have that $\gamma(g,1_{G})=\gamma(1_{G},g)=\gamma(1_{G},1_{G})$ for every $g\in G$ and $\gamma$ is said to be normalized if $\gamma(1_{G},1_{G})=1_{\Bbbk}$. Observe also that a bicharacter $\phi$ on $G$ is automatically a normalized 2-cocycle on $G$ since
\[
\phi(gh,k)\phi(g,h)=\phi(g,k)\phi(h,k)\phi(g,h)=\phi(g,hk)\phi(h,k)\ \text{ for every }\ g,h,k\in G
\]
and $\phi(g,1_{G})=\phi(1_{G},g)=1_{\Bbbk}$ for every $g\in G$. We have the following result:

\begin{lem}\label{braided}$($c.f. \cite[Remark 1.2]{article5}$)$
Given $f:G\to L$ a morphism of abelian groups, $F:\mathrm{Vec}_{G}\to\mathrm{Vec}_{L}$ as in Remark \ref{F}, $\phi^0:=\mathrm{Id}_{\Bbbk}$ and $\phi^{2}_{V,W}$ as above for every $V$ and $W$ in $\mathrm{Vec}_{G}$. Then $(F,\phi^0,\phi^2)$ is a strong monoidal functor if and only if $\gamma$ is a normalized 2-cocycle on $G$. Furthermore if $\phi$ and $\theta$ are bicharacters over $G$ and $L$ respectively (which give the respective braidings over $\mathrm{Vec}_{G}$ and $\mathrm{Vec}_{L}$), then $(F,\phi^0,\phi^2)$ is braided if and only if
\begin{equation}\label{mon2}
\phi(g,h)=\theta(f(g),f(h))\frac{\gamma(g,h)}{\gamma(h,g)}\ \text{ for every }\ g,h\in G.
\end{equation}  
\end{lem}

\begin{invisible}
\begin{proof}
Clearly $\phi^{2}_{V_{g},W_{h}}:F(V_{g}\otimes W_{h})\to F(V_{g})\otimes F(W_{h})$, $v\otimes w\mapsto\gamma(g,h)v\otimes w$ are isomorphisms of vector spaces for every $g,h\in G$. If we consider vector spaces $U_{g},V_{h},W_{k}$ with $g,h,k\in G$ we have that
\[
(\phi^{2}_{U_{g},V_{h}}\otimes\mathrm{Id})\phi^{2}_{U_{g}\otimes V_{h},W_{k}}(u\otimes v\otimes w)=(\phi^{2}_{U_{g},V_{h}}\otimes\mathrm{Id})(\gamma(gh,k)u\otimes v\otimes w)=\gamma(g,h)\gamma(gh,k)u\otimes v\otimes w
\]
while
\[
(\mathrm{Id}\otimes\phi^{2}_{V_{h},W_{k}})\phi^{2}_{U_{g},V_{h}\otimes W_{k}}(u\otimes v\otimes w)=(\mathrm{Id}\otimes\phi^{2}_{V_{h},W_{k}})(\gamma(g,hk)u\otimes v\otimes w)=\gamma(g,hk)\gamma(h,k)u\otimes v\otimes w
\]
and they are equal for every $u\in U_{g}$, $v\in V_{h}$ and $w\in W_{k}$ and every $g,h,k\in G$ if and only if \eqref{mon1} holds true, i.e. $\gamma$ is a 2-cocycle on $G$. Furthermore
\[
F(l_{U_{g}})(k\otimes u)=k\cdot u\ \text{ and }\ l_{F(U_{g})}(\phi^{0}\otimes\mathrm{Id})\phi^{2}_{\Bbbk_{1_{G}},U_{g}}(k\otimes u)=\gamma(1_{G},g)k\cdot u
\]
are equal for every $k\in\Bbbk$ and $u\in U_{g}$ and every $g\in G$ if and only if $\gamma(1_{G},g)=1_{\Bbbk}$ for all $g\in G$, i.e. $\gamma$ is normalized ($\gamma(g,1_{G})=1_{\Bbbk}$ for every $g\in G$ if we consider $r_{U_{g}}$). Then we have that $(F,\phi^0,\phi^2)$ is a strong monoidal functor if and only if $\gamma$ is a normalized 2-cocycle. Furthermore $(F,\phi^0,\phi^2)$ is braided if and only if for every $g,h\in G$ we have that 
\[
\sigma_{F(V_{g}),F(W_{h})}\phi^{2}_{V_{g},W_{h}}(v\otimes w)=\sigma_{V_{f(g)},W_{f(h)}}(\gamma(g,h)v\otimes w)=\gamma(g,h)\theta(f(g),f(h))w\otimes v
\]
and
\[
\phi^{2}_{W_{h},V_{g}}F(\sigma_{V_{g},W_{h}})(v\otimes w)=\phi^{2}_{W_{h},V_{g}}\sigma_{V_{g},W_{h}}(v\otimes w)=\phi^{2}_{W_{h},V_{g}}(\phi(g,h)w\otimes v)=\gamma(h,g)\phi(g,h)w\otimes v
\]
are equal for every $v\in V_{g}$ and $w\in W_{h}$. Thus $(F,\phi^0,\phi^2)$ is braided if and only if \eqref{mon2} holds true.
\end{proof}
\end{invisible}

\begin{oss}\label{u}
Consider the map $\bar{u}:G\to\Bbbk-\{0\}$ given by $g\mapsto\phi(g,g)$. Note that, since $\phi$ is a commutation factor, we obtain that 
\[
\bar{u}(gg')=\phi(gg',gg')=\phi(g,g)\phi(g,g')\phi(g',g)\phi(g',g')=\phi(g,g)\phi(g',g')=\bar{u}(g)\bar{u}(g')
\]
and then $\bar{u}$ is a morphism of groups. But we also have that, again since $\phi$ is a commutation factor, $\phi(g,g)\phi(g,g)=1_{\Bbbk}$, i.e. $\phi(g,g)^{2}=1_{\Bbbk}$ for every $g\in G$, thus $\phi(g,g)\in\{\pm1_{\Bbbk}\}$. Thus $\bar{u}(G)\subseteq\{\pm1_{\Bbbk}\}$ and then we have that, with $\phi$ a commutation factor, $\bar{u}:G\to\{\pm1_{\Bbbk}\}$ is a morphism of groups. If we consider the subgroup
\[
H:=\mathrm{ker}(\bar{u})=\{g\in G\ |\ \bar{u}(g)=1_{\Bbbk}\}=\{g\in G\ |\ \phi(g,g)=1_{\Bbbk}\}
\]
the possibilities are two: $H=G$ or $\{\pm1_{\Bbbk}\}\cong G/H$ and then $|G/H|=2$. In the second case, if $G$ is finite (and so also $H$ is finite), then $|G|=2\cdot|H|$. In the finite case if the cardinality of $G$ is odd we always have $G=H$ and, with $G$ non trivial, if $H=\{1_{G}\}$ then $G=\mathbb{Z}_{2}$. Thus if the cardinality of $G$ is bigger than 2, $H$ can not be trivial and hence it never happens that $\phi(h,h)=-1_{\Bbbk}$ for every $h\not=1_{G}$. One can find, for example in \cite[Section 3.2]{a}, a complete classification of a nondegenerate skew-symmetric bicharacter defined on a finite abelian group.
\end{oss}

In \cite[Section 1.5]{article5} it is said how to obtain a braided strong monoidal functor from $\mathrm{Vec}_{G}$ to $\mathrm{Vec}_{\mathbb{Z}_{2}}$, in the case $G$ is a finite abelian group; this also works when $G$ is finitely generated abelian and we report here the proof for completeness. If we consider $\mathbb{Z}_{2}:=\{0,1\}$, by Remark \ref{u} we already have a morphism of groups $\bar{u}:G\to\mathbb{Z}_{2}$, where $\bar{u}(g)=0$ if $g\in G$ is such that $\phi(g,g)=1_{\Bbbk}$ and $\bar{u}(g)=1$ if $\phi(g,g)=-1_{\Bbbk}$. We need a normalized 2-cocycle $\gamma$ on $G$ such that \eqref{mon2} is satisfied for $\phi$ and $\eta$ where $\eta:\mathbb{Z}_{2}\times\mathbb{Z}_{2}\to\{\pm1_{\Bbbk}\}\subseteq\Bbbk-\{0\}$, $(x,y)\mapsto(-1)^{x\cdot y}$ is the commutation factor on $\mathbb{Z}_{2}$. 

\begin{prop}\label{braided strong}
    If $G$ is a finitely generated abelian group then we have a braided strong monoidal functor $(F,\phi^0,\phi^2):\mathrm{Vec}_{G}\to\mathrm{Vec}_{\mathbb{Z}_{2}}$.
\end{prop}

\begin{proof}
Define $\kappa:G\times G\to\{\pm1_{\Bbbk}\}$ such that $\kappa(g,h)=1_{\Bbbk}$ if $\phi(g,g)=1_{\Bbbk}$ or $\phi(h,h)=1_{\Bbbk}$ and $\kappa(g,h)=-1_{\Bbbk}$ if $\phi(g,g)=\phi(h,h)=-1_{\Bbbk}$, for every $g,h\in G$. By definition of $\kappa$, which is clearly a commutation factor, $\phi\kappa:G\times G\to\Bbbk-\{0\},(g,h)\mapsto\phi(g,h)\kappa(g,h)$ is a commutation factor on $G$. Observe that $\phi\kappa(g,g)=1_{\Bbbk}$ for every $g\in G$. If we consider $\hat{G}:=\text{Hom}_{\text{Grp}}(G,\Bbbk-\{0\})$, the character group of $G$, and we define the group morphism $\chi:G\to\hat{G}$, $g\mapsto\chi_g$ by setting $\chi_{g}(h)=\phi\kappa(g,h)$ for $g,h\in G$, we have that $\phi\kappa$ induces a non-degenerate bicharacter $\phi'$ on the group $G':=G/\mathrm{ker}(\chi)$ such that $\phi'(\bar{g},\bar{g})=1_{\Bbbk}$ for every $\bar{g}\in G'$. In fact we can define $\phi'(\bar{g},\bar{g'})=\phi\kappa(g,g')$ for every $\bar{g},\bar{g'}\in G'$, which is well defined since, for $g,g'\in G$ and $h,h'\in\mathrm{ker}(\chi)$, we have that
\[
\phi\kappa(h,g')=\phi\kappa(h,h')=1_{\Bbbk}\ \text{ and }\ 1_{\Bbbk}=\phi\kappa(g,h')\phi\kappa(h',g)=\phi\kappa(g,h')
\]
and then $\phi\kappa(gh,g'h')=\phi\kappa(g,g')$. Clearly $\phi'$ is a commutation factor and $\phi'(\bar{g},\bar{g})=1_{\Bbbk}$ for every $\bar{g}\in G'$. 
Since $G$ is a finitely generated abelian group the same is true for $G'$. Thus by \cite[Lemma 2]{Sc}, since $\phi'$ is a commutation factor on $G'$ such that $\phi'(\bar{g},\bar{g})=1_{\Bbbk}$ for every $\bar{g}\in G'$, there exists a bicharacter $\gamma':G'\times G'\to\Bbbk-\{0\}$ on $G'$ (so, in particular, a normalized 2-cocycle on $G'$) such that
\[
\phi'(\bar{g},\bar{h})=\frac{\gamma'(\bar{g},\bar{h})}{\gamma'(\bar{h},\bar{g})}\ \text{ for every }\ \bar{g},\bar{h}\in G'.
\]
If we call $p:G\to G'$ the canonical projection, we have that $\phi'\circ(p\times p)=\phi\kappa$ and we can consider the bicharacter (so a normalized 2-cocycle) on $G$ given by $\gamma:=\gamma'\circ(p\times p)$. Clearly
\[
\phi\kappa(g,h)=\phi'(\bar{g},\bar{h})=\frac{\gamma'(\bar{g},\bar{h})}{\gamma'(\bar{h},\bar{g})}=\frac{\gamma(g,h)}{\gamma(h,g)}\ \text{ for every }g,h\in G.
\]
But now, given $\eta$ the bicharacter on $\mathbb{Z}_{2}$  of before, we have that $\kappa=\eta\circ(\bar{u}\times\bar{u})$, i.e. $\kappa(g,h)=\eta(\bar{u}(g),\bar{u}(h))$ for every $g,h\in G$. Thus we obtain that 
\[
\frac{\gamma(g,h)}{\gamma(h,g)}=\phi\kappa(g,h)=\phi(g,h)\kappa(g,h)=\phi(g,h)\eta(\bar{u}(g),\bar{u}(h))\ \text{ for every }\ g,h\in G
\]
and then that
\[
\phi(g,h)=\eta(\bar{u}(g),\bar{u}(h))\frac{\gamma(g,h)}{\gamma(h,g)}\ \text{ for every }\ g,h\in G.
\]
Hence, by Lemma \ref{braided}, we have a braided strong monoidal functor $(F,\phi^{2},\phi^{0}):\mathrm{Vec}_{G}\to\mathrm{Vec}_{\mathbb{Z}_{2}}$. 
\end{proof}

Thus, from now on, we suppose that the abelian group $G$ is $\textit{finitely generated}$. We know that a braided strong monoidal functor preserves Hopf monoids (see \cite[Propositions 3.46, 3.50]{l}), thus, via $(F,\phi^{0},\phi^{2})$, every color Hopf algebra becomes a super Hopf algebra and every morphism of color Hopf algebras becomes a morphism of super Hopf algebras (we already know that it is automatically in Vec$_{\mathbb{Z}_{2}}$, so it will be also a morphism of algebras and of coalgebras with respect to new products and new coproducts). Given a color Hopf algebra $(H:=\bigoplus_{g\in G}{H_{g}},m,u,\Delta,\epsilon,S)$, the super Hopf algebra will be given by 
\[
(F(H),F(m)\circ(\phi^{2}_{H,H})^{-1},F(u)\circ(\phi^0)^{-1},\phi^{2}_{H,H}\circ F(\Delta),\phi^0\circ F(\epsilon),F(S))=(H_{0}\oplus H_{1},m\circ(\phi^{2}_{H,H})^{-1},u,\phi^{2}_{H,H}\circ\Delta,\epsilon,S) 
\]
since $F(f)=f$ for $f$ in Vec$_{G}$ and $\phi^0$ is the identity, where $H_{0}=\bigoplus_{g\in G_{0}}{H_{g}}$ and $H_{1}=\bigoplus_{g\in G_{1}}{H_{g}}$ by setting $G_{0}:=\{g\in G\ |\ \phi(g,g)=1_{\Bbbk}\}$ and $G_{1}:=\{g\in G\ |\ \phi(g,g)=-1_{\Bbbk}\}$. 

\begin{lem}
Given a faithful braided strong monoidal functor $F:\Mm\to\Mm'$, then $A$ is a (co)commutative (co)monoid in $\Mm$ if and only if $F(A)$ is a (co)commutative (co)monoid in $\Mm'$.
\end{lem}

\begin{proof}
Since $F$ is braided we have that
\[
m_{F(A)}\circ c'_{F(A),F(A)}=F(m_{A})\circ(\phi^2_{A,A})^{-1}\circ c'_{F(A),F(A)}=F(m_{A}\circ c_{A,A})\circ(\phi^2_{A,A})^{-1} 
\]
so that if $A$ is a commutative monoid in $\Mm$ then $F(A)$ is a commutative monoid in $\Mm'$ and
\[
c'_{F(A),F(A)}\circ\Delta_{F(A)}=c'_{F(A),F(A)}\circ\phi^2_{A,A}\circ F(\Delta_{A})=\phi^2_{A,A}\circ F(c_{A,A}\circ\Delta_{A})
\]
so that if $A$ is a cocommutative comonoid in $\Mm$ then $F(A)$ is a cocommutative comonoid in $\Mm'$. 
\begin{invisible}
If $A$ commutative, i.e. $m_{A}=m_{A}\circ c_{A,A}$, then $F(A)$ is commutative, indeed
\[
F(m_{A})\circ(\phi^2_{A,A})^{-1}\circ c'_{F(A),F(A)}=F(m_{A})\circ F(c_{A,A})\circ(\phi^2_{A,A})^{-1}=F(m_{A})\circ(\phi^2_{A,A})^{-1}, 
\]
i.e. $m_{F(A)}\circ c'_{F(A),F(A)}=m_{F(A)}$, so $F(A)$ is commutative. Vice versa, if $F(B)$ is commutative, then 
\[
F(m_{B})\circ F(c_{B,B})=m_{F(B)}\circ\phi^2_{B,B}\circ F(c_{B,B})=m_{F(B)}\circ c'_{F(B),F(B)}\circ\phi^2_{B,B}=m_{F(B)}\circ\phi^2_{B,B}=F(m_{B})
\]
and then, since $F$ is faithful, we obtain $m_{B}\circ c_{B,B}=m_{B}$, i.e. $B$ is commutative. Similarly for the cocommutative case.
\end{invisible}
If $F(A)$ is a (co)commutative (co)monoid in $\Mm'$ then, from the previous two computations, we obtain that $A$ is a (co)commutative (co)monoid in $\Mm$ by using that $F$ is faithful.
\end{proof}

\begin{cor}\label{commutative}
Given $A$ in $\Hc$, it is (co)commutative if and only if $F(A)$ in $\mathrm{Hopf}_{\mathrm{coc}}(\mathrm{Vec}_{\mathbb{Z}_2})$ is (co)commutative.
\end{cor}
\begin{invisible}
Clearly, if we consider a color Hopf algebra $H$ which is cocommutative, i.e. such that $c_{H,H}\circ\Delta=\Delta$, then $F(H)$ will be a cocommutative super Hopf algebra, indeed 
\[
c'_{F(H),F(H)}\circ\phi^2_{H,H}\circ F(\Delta)=\phi^2_{H,H}\circ F(c_{H,H})\circ F(\Delta)=\phi^2_{H,H}\circ F(c_{H,H}\circ\Delta)=\phi^2_{H,H}\circ F(\Delta)
\]
since $F$ is braided, where we denote by $c$ and $c'$ the braidings of Vec$_{G}$ and Vec$_{\mathbb{Z}_{2}}$, respectively. Similarly, if $H$ is a commutative color Hopf algebra then $F(H)$ is a commutative super Hopf algebra. 
\end{invisible}
In the following we will often refer to the functor $F$ restricted to the category of cocommutative color Hopf algebras, still calling it $F$. 
\begin{invisible}
Observe that, explicitly, the product and the coproduct of $F(H)$ computed on elements and referred to the product and the coproduct of $H$ are
\[
x\cdot_{F(H)}y:=\gamma(|x|,|y|)^{-1}xy\ \text{ and }\ \Delta_{F(H)}(x):=\gamma(|x_{1}|,|x_{2}|)x_{1}\otimes x_{2}\ \text{ for every }x,y\in F(H)
\]
for a given normalized 2-cocycle $\gamma$ over $G$, for which we use the same convention that we use for the bicharacter $\phi$. We will write $\Delta_{F(H)}(x)=\bar{x}_{1}\otimes\bar{x}_{2}$ and $x\cdot_{F(H)}y=x\cdot y$ in the computations.
\end{invisible}

\begin{oss}
In order to avoid confusion, here we denote $\mathrm{Vec}_{G}^{\phi}$ by indicating the bicharacter associated to the braiding at the top. As it is said in \cite[Section 1.5]{article5}, the normalized 2-cocycle $\gamma:G\times G\to\Bbbk-\{0\}$ induces an equivalence of braided monoidal categories from $\mathrm{Vec}_{G}^{\phi\kappa}$ to $\mathrm{Vec}_{G}^{t}$, where $t:G\times G\to\Bbbk-\{0\}$ is the trivial bicharacter such that $t(g,h)=1_{\Bbbk}$ for every $g,h\in G$. Indeed, if we consider the morphism of groups $\mathrm{Id}_{G}$ and $\gamma$, we have
\[
\phi\kappa(g,h)=t(g,h)\frac{\gamma(g,h)}{\gamma(h,g)}\ \text{ for every }g,h\in G
\]
and then a braided strong monoidal functor $\mathrm{Vec}_{G}^{\phi\kappa}\to\mathrm{Vec}_{G}^{t}$ by Lemma \ref{braided}. But now, clearly, we can consider the normalized 2-cocycle $\gamma^{-1}:G\times G\to\Bbbk-\{0\}$, $(g,h)\to\gamma(g,h)^{-1}$ and then
\[
t(g,h)=\phi\kappa(g,h)\frac{\gamma(h,g)}{\gamma(g,h)}=\phi\kappa(g,h)\frac{\gamma^{-1}(g,h)}{\gamma^{-1}(h,g)}\ \text{ for every }\ g,h\in G
\]
so that we have a braided strong monoidal functor $\mathrm{Vec}_{G}^{t}\to\mathrm{Vec}_{G}^{\phi\kappa}$, again by Lemma \ref{braided}. These two functors give an equivalence of (symmetric) braided monoidal categories between $\mathrm{Vec}_{G}^{\phi\kappa}$ and $\mathrm{Vec}_{G}^{t}$ 
\begin{invisible}
This means, in particular, that given a Hopf monoid in $\mathrm{Vec}_{G}^{\phi\kappa}$ one can obtain a Hopf monoid in $\mathrm{Vec}_{G}^{t}$ and vice versa 
\end{invisible}
and now the possibilities are two. If we have that $\phi(g,g)=1_{\Bbbk}$ for every $g\in G$, then clearly $\kappa(g,h)=1_{\Bbbk}$ for every $g,h\in G$ and then $\phi\kappa=\phi$, so that we have an equivalence of symmetric monoidal categories between Vec$_{G}^{\phi}$ and Vec$_{G}^{t}$. Indeed observe that, in this case, $G_{0}=G$ and then, given $V$ in Vec$_{G}^{\phi}$, we have $V=V_{0}$ and $V_{1}=0$ and $\eta(0,0)=1_{\Bbbk}$. The objects of Hopf$_{\mathrm{coc}}(\mathrm{Vec}_{G}^{t})$, the category of $G$-graded cocommutative Hopf algebras, are ordinary cocommutative Hopf algebras graded over $G$ as vector spaces and with $m,u,\Delta,\epsilon,S$ which preserve gradings (thus $G$-graded algebras and coalgebras) and morphisms are algebra and coalgebra maps which preserve gradings. In particular from a cocommutative color Hopf algebra we can obtain an ordinary cocommutative Hopf algebra and vice versa. Otherwise if $\phi(g,g)=-1_{\Bbbk}$ for some $g\in G$, we can return to the braided strong monoidal functor $(F,\phi^{0},\phi^{2}):\mathrm{Vec}_{G}^{\phi}\to\mathrm{Vec}_{\mathbb{Z}_{2}}^{\eta}$ of before and, given $H$ and $f$ in Hopf$_{\mathrm{coc}}(\mathrm{Vec}_{G}^{\phi})$, we obtain $F(H)$ and $F(f)$ in Hopf$_{\mathrm{coc}}(\mathrm{Vec}_{G}^{\eta})$, the category of $G$-graded cocommutative super Hopf algebras, where objects are $G$-graded algebras and coalgebras (since also $\phi^2_{H,H}$ is in Vec$_{G}$), then also $\mathbb{Z}_{2}$-graded algebras and coalgebras by considering the new grading, with respect to which, considering the braiding of super vector spaces, they are cocommutative super Hopf algebras and morphisms are algebra and coalgebra maps which preserve the $G$-grading (and then that over $\mathbb{Z}_{2}$). 
\end{oss}

Hence every $H$ in $\Hc$ can be seen as a cocommutative super Hopf algebra. If $\phi(g,g)=1_{\Bbbk}$ for every $g\in G$ we have that this is effectively an ordinary cocommutative Hopf algebra and Newman's Theorem holds true; this always happens if we have a finite group $G$ of odd cardinality by Remark \ref{u}, for example. If $\phi(g,g)=-1_{\Bbbk}$ for some $g\in G$ we can use the more general Theorem \ref{ma} for cocommutative super Hopf algebras, where char$\Bbbk\not=2$ is needed, which allows us to deal with the more general case. 

\subsection{Generalized Newman's theorem for color Hopf algebras}

Now we can generalize Theorem \ref{ma} and Corollary \ref{c1} to the case of cocommutative color Hopf algebras by using the functor $F:\mathrm{Vec}_{G}\to\mathrm{Vec}_{\mathbb{Z}_{2}}$, in the case char$\Bbbk\not=2$ and $G$ is a finitely generated abelian group. 

\begin{lem}\label{injsub}
    The forgetful functor $K:\mathrm{Vec}_{G}\to\mathrm{Vec}_{\Bbbk}$ is injective on subobjects and on quotients of the same object. As a consequence, the same holds true if $K$ is restricted to the categories $\mathrm{Mon}(\mathrm{Vec}_{G})$, $\mathrm{Comon}(\mathrm{Vec}_{G})$, $\mathrm{Bimon}(\mathrm{Vec}_{G})$ and $\mathrm{Hopf}(\mathrm{Vec}_{G})$. 
\end{lem}

\begin{proof}
    Given $A$ in Vec$_{G}$ and $B$, $C$ graded subspaces of $A$, then $B_{g}=B\cap A_{g}$ and $C_{g}=C\cap A_{g}$ for every $g\in G$. Thus, if $K(B)=K(C)$, i.e. $B$ and $C$ are the same vector space, then they must be the same object in Vec$_{G}$. Furthermore, if we consider $A$ in Mon$(\mathrm{Vec}_{G})$, Comon$(\mathrm{Vec}_{G})$, Bimon$(\mathrm{Vec}_{G})$ or Hopf$(\mathrm{Vec}_{G})$ and $B$, $C$ subobjects of $A$ in these categories, we have that $B$ and $C$ are the same object in these categories if and only if they are the same object in Vec$_{G}$ because their operations are the restrictions of those of $A$ and then this happens if $K(B)=K(C)$. Moreover, given $A/B$ and $A/C$ in Vec$_{G}$ such that $K(A/B)=K(A/C)$, i.e. $A/B$ and $A/C$ are the same vector space, then $B=0_{A/B}=0_{A/C}=C$. As a consequence, $(A/B)_{g}=(A_{g}+B)/B=(A_{g}+C)/C=(A/C)_{g}$ for every $g\in G$ and then $A/B$ and $A/C$ are the same object in Vec$_{G}$. The same result holds true when $A/B$ and $A/C$ are in Mon$(\mathrm{Vec}_{G})$, Comon$(\mathrm{Vec}_{G})$, Bimon$(\mathrm{Vec}_{G})$ or Hopf$(\mathrm{Vec}_{G})$, since $A/B$ and $A/C$ are the same object in these categories if and only if they are the same object in Vec$_{G}$ because their operations are induced by those of $A$ through the canonical projection.
\end{proof}

\begin{lem}\label{sub}
    The functor $F:\mathrm{Vec}_{G}\to\mathrm{Vec}_{\mathbb{Z}_{2}}$ preserves and reflects submonoids and subcomonoids, then sub-bimonoids and Hopf submonoids.
\end{lem}

\begin{proof}
We consider the case of Hopf submonoids which includes all the others in itself. We already know that if $C$ is a color Hopf subalgebra of $A$, i.e. the inclusion $i:C\to A$ is in $\Hc$, then $F(i):F(C)\to F(A)$ is in $\mathrm{Hopf}_{\mathrm{coc}}(\mathrm{Vec}_{\mathbb{Z}_{2}})$, i.e. $F(C)$ is a super Hopf subalgebra of $F(A)$, so we show the other direction, assuming $A$ in $\Hc$ and $C\subseteq A$ a graded subspace such that $F(C)$ is a super Hopf subalgebra of $F(A)$. But now we only have to observe that
\[
m_{F(A)}(F(C)\otimes F(C))=F(m_{A})(\phi^2_{A,A})^{-1}(F(C)\otimes F(C))=F(m_{A})(F(C\otimes C))=F(m_{A}(C\otimes C)),
\]
\[
\Delta_{F(A)}(F(C))=\phi^2_{A,A}F(\Delta_{A})(F(C))=\phi^2_{A,A}F(\Delta_{A}(C)) 
\]
and also $u_{F(A)}(\Bbbk)=F(u_{A}(\Bbbk))$, $S_{F(A)}(F(C))=F(S_{A}(C))$ so that, since $F(C)$ is a super Hopf subalgebra of $F(A)$ (i.e. $F(C)$ is closed under the operations of $F(A)$), then $C$ is a color Hopf subalgebra of $A$ (i.e. $C$ is closed under the operations of $A$), since $F$ does not change the structure of vector space.
\begin{invisible}
We can also see this as: if an inclusion $i:C\to A$ is in $\Hc$ then $F(i):F(C)\to F(A)$ is in $\mathrm{Hopf}_{\mathrm{coc}}(\mathrm{Vec}_{\mathbb{Z}_{2}})$ and viceversa.
In fact, for example, if $F(i)\circ m_{F(C)}=m_{F(A)}\circ(F(i)\otimes F(i))$ then $F(i\circ m_{C})\circ(\phi^2_{C,C})^{-1}=F(m_{A}\circ(i\otimes i))\circ(\phi^2_{C,C})^{-1}$ and $F$ is the identity on morphisms.
\end{invisible}
\end{proof}

\begin{oss}\label{quotient}
Observe that, given $\pi:A\to A/I$ in Vec$_{G}$, then $F(\pi):F(A)\to F(A/I)$ is in Vec$_{\mathbb{Z}_2}$ and it is still surjective, then by Remark \ref{Im and qu} $F(A/I)$ has the unique grading induced by $F(A)$ through the surjection, i.e. $F(A/I)_{i}=F(\pi)(F(A)_{i})=(F(A)_{i}+I)/I=(F(A)_{i}+F(I))/F(I)$ for $i=0,1$ and this is exactly the grading in Vec$_{\mathbb{Z}_2}$ of the quotient of $F(A)$ with its super subspace $F(I)$. Thus 
\begin{invisible}
Recall that, given $A$ in Vec$_{G}$ and $I\subseteq A$ a graded subspace, then $A/I=\bigoplus_{g\in G}{A_{g}/A_{g}\cap I}$ in Vec$_{G}$. So we have that
\[
F(A/I)=\big(\bigoplus_{g\in G_{0}}{\frac{A_{g}}{A_{g}\cap I}}\big)\oplus\big(\bigoplus_{g\in G_{1}}{\frac{A_{g}}{A_{g}\cap I}}\big)\cong\frac{\bigoplus_{g\in G_{0}}{A_{g}}\oplus\bigoplus_{g\in G_{1}}{A_{g}}}{\bigoplus_{g\in G_{0}}{A_{g}\cap I}\oplus\bigoplus_{g\in G_{1}}{A_{g}\cap I}}=\frac{F(A)}{F(I)}
\]
as vector spaces through a canonical isomorphism, so when we consider $F(A)/F(I)$ in Vec$_{\mathbb{Z}_2}$ (quotient of $F(A)$ in Vec$_{\mathbb{Z}_2}$ with $F(I)\subseteq F(A)$ super subspace) we are meaning exactly $F(A/I)$ and so, given $\pi:A\to A/I$ in Vec$_{G}$, 
\end{invisible}
$F(\pi):F(A)\to F(A/I)$ is the projection $F(A)\to F(A)/F(I)$ in Vec$_{\mathbb{Z}_2}$. Also note that, if $\pi:A\to A/I$ is in $\Hc$, then $F(\pi)$ is in $\mathrm{Hopf}_{\mathrm{coc}}(\mathrm{Vec}_{\mathbb{Z}_{2}})$ and $F(A/I)$ has the unique structure in $\mathrm{Hopf}_{\mathrm{coc}}(\mathrm{Vec}_{\mathbb{Z}_{2}})$ induced by $F(A)$. 
\begin{invisible}
Thus, in the hypothesis of Lemma \ref{lemmino}, if $A/I$ is in $\Hc$ then $F(A)/F(I)=F(A/I)$ in $\mathrm{Hopf}_{\mathrm{coc}}(\mathrm{Vec}_{\mathbb{Z}_{2}})$. In fact, taking $\pi:A\to A/I$ in $\Hc$, then $F(\pi):F(A)\to F(A/I)$ is still surjective and then the grading on $F(A/I)$ is the unique induced by $F(A)$ through the surjection, so that $F(A/I)_{g}=F(\pi)(F(A)_{g})=(I+F(A)_{g})/I$, which is the grading on $F(A)/F(I)$ and then they are the same object in Vec$_{\mathbb{Z}_2}$. Furthermore, since $F(\pi)$ is a morphism in $\mathrm{Hopf}_{\mathrm{coc}}(\mathrm{Vec}_{\mathbb{Z}_{2}})$, $F(A/I)$ has the unique structure of super Hopf algebras induced by $F(A)$, so that $F(A/I)=F(A)/F(I)$ in $\mathrm{Hopf}_{\mathrm{coc}}(\mathrm{Vec}_{\mathbb{Z}_{2}})$.
\end{invisible}
\end{oss}

\begin{invisible}
\begin{lem}\label{lemmino}
The functor $F:\mathrm{Vec}_{G}\to\mathrm{Vec}_{\mathbb{Z}_{2}}$ preserves and reflects quotient graded algebras and coalgebras, then quotient color bialgebras and Hopf algebras.
\end{lem}

\begin{proof}
We consider the case of quotient color Hopf algebras which includes all the others in itself. We already know that if $A/I$ is in $\Hc$ then $F(A)/F(I)$ is in $\mathrm{Hopf}_{\mathrm{coc}}(\mathrm{Vec}_{\mathbb{Z}_{2}})$, we show the other direction and we assume that $I\subseteq A$ is a graded subspace such that $F(A)/F(I)$ is in $\mathrm{Hopf}_{\mathrm{coc}}(\mathrm{Vec}_{\mathbb{Z}_{2}})$. By Remark \ref{oso}, $A/I$ is in $\Hc$ if we show that the graded subspace $I$ is a two-sided ideal and two-sided coideal of $A$ and it is closed under $S_{A}$. Now, as for Lemma \ref{sub}, we observe that
\[
F(IA)=F(m_{A})(F(I\otimes A))=F(m_{A})(\phi^2_{A,A})^{-1}(F(I)\otimes F(A))=F(I)F(A)\ \text{ and }\ F(AI)=F(A)F(I)
\]
and $S_{F(A)}(F(I))=F(S_{A}(I))$, thus, since $F(I)$ is a two-sided ideal of $F(A)$ and it is closed under $S_{F(A)}$, then $I$ is a two-sided ideal of $A$ and it is closed under $S_{A}$ because $F$ does not change the structure of vector space. We also have that 
\[
\phi^2_{A,A}F(\Delta_{A}(I))=\Delta_{F(A)}(F(I))\subseteq F(I)\otimes F(A)+F(A)\otimes F(I)\ \text{and}\ \epsilon_{A}(F(I))=0
\]
since $F(I)$ is a two-sided coideal of $F(A)$. Now we consider the forgetful functor $G:\mathrm{Vec}_{\mathbb{Z}_{2}}\to\mathrm{Vec}_{\Bbbk}$ which forgets the grading. So $G(\phi^2_{A,A})=\phi^2_{A,A}$ is considered as an isomorphism of vector spaces so  
\[
G(\phi^2_{A,A})(\Delta_{A}(I))=G(\phi^2_{A,A})(G(F(\Delta_{A}(I))))=G(\phi^2_{A,A}F(\Delta_{A}(I)))\subseteq I\otimes A+A\otimes I 
\]
and then $\Delta_{A}(I)\subseteq G((\phi^2_{A,A})^{-1})(A\otimes I+I\otimes A)\subseteq A\otimes I+I\otimes A$
by remembering the explicit form of $\phi^2_{A,A}$.
\end{proof}
\end{invisible}

\begin{prop}\label{ossi2}
The functor $F:\Hc\to\mathrm{Hopf}_{\mathrm{coc}}(\mathrm{Vec}_{\mathbb{Z}_{2}})$ preserves and reflects equalizers.
\end{prop}

\begin{proof}
 If we take $f,g:A\to B$ in $\Hc$ and $j:\mathrm{Eq}(f,g)\to A$ the equalizer of the pair $(f,g)$ in $\Hc$, then we can consider $F(j):F(\mathrm{Eq}(f,g))\to F(A)$ in Hopf$_{\mathrm{coc}}(\mathrm{Vec}_{\mathbb{Z}_{2}})$ and we can show that $F(\mathrm{Eq}(f,g))=\mathrm{Eq}(F(f),F(g))$. We know that $\mathrm{Eq}(F(f),F(g))$ in Hopf$_{\mathrm{coc}}(\mathrm{Vec}_{\mathbb{Z}_{2}})$ is given by those $x\in F(A)$ such that $(\mathrm{Id}_{F(A)}\otimes F(f))\Delta_{F(A)}(x)=(\mathrm{Id}_{F(A)}\otimes F(g))\Delta_{F(A)}(x)$. But now we have that
\[
(\mathrm{Id}_{F(A)}\otimes F(f))\circ\Delta_{F(A)}=(F(\mathrm{Id}_{A})\otimes F(f))\circ\phi^{2}_{A,A}\circ F(\Delta_{A})=\phi^{2}_{A,B}\circ F(\mathrm{Id}_{A}\otimes f)\circ F(\Delta_{A})=\phi^{2}_{A,B}\circ F((\mathrm{Id}_{A}\otimes f)\circ\Delta_{A})
\]
and similarly $(\mathrm{Id}_{F(A)}\otimes F(g))\circ\Delta_{F(A)}=\phi^{2}_{A,B}\circ F((\mathrm{Id}_{A}\otimes g)\circ\Delta_{A})$. Then, equivalently, $\mathrm{Eq}(F(f),F(g))$ is composed by elements $x\in F(A)$ such that $F((\mathrm{Id}_{A}\otimes f)\Delta_{A})(x)=F((\mathrm{Id}_{A}\otimes g)\Delta_{A})(x)$, i.e. such that $(\mathrm{Id}_{A}\otimes f)\Delta_{A}(x)=(\mathrm{Id}_{A}\otimes g)\Delta_{A}(x)$ and these are exactly the elements of $F(\mathrm{Eq}(f,g))$. Hence we have that $(F(\mathrm{Eq}(f,g)),F(j))$ is the equalizer of the pair $(F(f),F(g))$ in Hopf$_{\mathrm{coc}}(\mathrm{Vec}_{\mathbb{Z}_{2}})$, so $F$ preserves equalizers. The fact that $F$ reflects equalizers follows using that $F$ preserves equalizers and that $F$ reflects isomorphisms (see \cite[Proposition 2.9.7]{book1}).
\begin{invisible}
Now consider $f,g:A\to B$ and $i:C\to A$ in $\Hc$ such that $f\circ i=g\circ i$ and such that $F(i):F(C)\to F(A)$ is the equalizer of the pair $(F(f),F(g))$ in Hopf$_{\mathrm{coc}}(\mathrm{Vec}_{\mathbb{Z}_{2}})$. Then if we consider $k:\mathrm{Eq}(f,g)\to A$ the equalizer of the pair $(f,g)$ in $\Hc$, since $f\circ i=g\circ i$, there exists a unique $\lambda:C\to\mathrm{Eq}(f,g)$ in $\Hc$ such that $k\circ\lambda=i$.
\[
\begin{tikzcd}
	\mathrm{Eq}(f,g) && A & B \\
	& C
	\arrow[from=1-1, to=1-3, "k"]
	\arrow[from=2-2, to=1-1,dotted,"\lambda"]
	\arrow[from=2-2, to=1-3, "i"']
	\arrow[shift left=2, from=1-3, to=1-4, "f"]
	\arrow[shift right=2, from=1-3, to=1-4, "g"']
\end{tikzcd}
\]
Now, since $F(\mathrm{Eq}(f,g))=\mathrm{Eq}(F(f),F(g))$ and $F(i)$ is the equalizer of the pair $(F(f),F(g))$ in Hopf$_{\mathrm{coc}}(\mathrm{Vec}_{\mathbb{Z}_{2}})$, $F(\lambda)$ must be an isomorphism in Hopf$_{\mathrm{coc}}(\mathrm{Vec}_{\mathbb{Z}_{2}})$ and then $\lambda$ must be an isomorphism in $\Hc$ since $F$ is faithful. Thus $i$ is the equalizer of the pair $(f,g)$ in $\Hc$, so $F$ reflects equalizers.
\end{invisible}
\end{proof}

Clearly the previous result holds true by considering $F:\mathrm{Comon}_{\mathrm{coc}}(\mathrm{Vec}_{G})\to\mathrm{Comon}_{\mathrm{coc}}(\mathrm{Vec}_{\mathbb{Z}_{2}})$. 
\begin{invisible}
We also have the same result for coequalizers:

\begin{prop}\label{coequalizer}
    The functor $F:\Hc\to\mathrm{Hopf}_{\mathrm{coc}}(\mathrm{Vec}_{\mathbb{Z}_{2}})$ preserves and reflects coequalizers.
\end{prop}

\begin{proof}
Let $f,g:A\to B$ in $\Hc$ and $\pi:B\to B/I$ the coequalizer of the pair $(f,g)$ in $\Hc$, where $I=B((f-g)(A))B$. We can consider $F(\pi):F(B)\to F(B/I)$ in $\mathrm{Hopf}_{\mathrm{coc}}(\mathrm{Vec}_{\mathbb{Z}_{2}})$, where $F(B/I)=F(B)/F(I)$ in $\mathrm{Hopf}_{\mathrm{coc}}(\mathrm{Vec}_{\mathbb{Z}_{2}})$ by Remark \ref{quotient}. Furthermore
\[    
    F(B((f-g)(A))B)=F(B)F((f-g)(A))F(B)=F(B)((F(f)-F(g))(F(A)))F(B),
\]
thus $F(\pi):F(B)\to F(B/I)$ is the coequalizer of the pair $(F(f),F(g))$ in $\mathrm{Hopf}_{\mathrm{coc}}(\mathrm{Vec}_{\mathbb{Z}_{2}})$, so $F$ preserves coequalizers. The fact that $F$ reflects coequalizers follows, as before for equalizers, using that $F(\mathrm{Coeq}(f,g))=\mathrm{Coeq}(F(f),F(g))$ and that $F$ reflects isomorphisms (see \cite[Exercise 3 p.155]{book3}).
\end{proof}
Clearly Propositions \ref{ossi2} and \ref{coequalizer} hold true when we consider kernels and cokernels.
\end{invisible}
 
\begin{lem}\label{module}
Given a graded algebra $H$ and a color left $H$-module $A$, then $F(A)$ is a super left $F(H)$-module. 
\end{lem}

\begin{proof}
    Given the action $\mu:H\otimes A\to A$ in Vec$_{G}$, we define $\mu':=F(\mu)\circ(\phi^2_{H,A})^{-1}:F(H)\otimes F(A)\to F(A)$ in Vec$_{\mathbb{Z}_{2}}$. Thus we can compute
\[
\begin{split}
    \mu'\circ(m_{F(H)}\otimes\mathrm{Id}_{F(A)})&=F(\mu)\circ(\phi^2_{H,A})^{-1}\circ(F(m_{H})\otimes F(\mathrm{Id}_{A}))\circ((\phi^2_{H,H})^{-1}\otimes\mathrm{Id}_{F(A)})\\&=F(\mu)\circ F(m_{H}\otimes\mathrm{Id}_{A})\circ(\phi^2_{H\otimes H,A})^{-1}\circ((\phi^2_{H,H})^{-1}\otimes\mathrm{Id}_{F(A)})\\&=F(\mu\circ(\mathrm{Id}_{H}\otimes\mu))\circ((\phi^2_{H,H}\otimes\mathrm{Id}_{F(A)})\circ\phi^2_{H\otimes H,A})^{-1}\\&=F(\mu)\circ F(\mathrm{Id}_{H}\otimes\mu)\circ(\phi^2_{H,H\otimes A})^{-1}\circ(\mathrm{Id}_{F(H)}\otimes\phi^2_{H,A})^{-1}\\&=F(\mu)\circ(\phi^2_{H,A})^{-1}\circ(\mathrm{Id}_{F(H)}\otimes F(\mu))\circ(\mathrm{Id}_{F(H)}\otimes(\phi^2_{H,A})^{-1})\\&=\mu'\circ(\mathrm{Id}_{F(H)}\otimes\mu')
\end{split}
\]
and
\[
\begin{split}
    \mu'\circ(u_{F(H)}\otimes\mathrm{Id}_{F(A)})&=F(\mu)\circ(\phi^2_{H,A})^{-1}\circ(F(u_{H})\otimes F(\mathrm{Id}_{A}))=F(\mu)\circ F(u_{H}\otimes\mathrm{Id}_{A})\circ(\phi^2_{\Bbbk,A})^{-1}\\&=F(l_{A})\circ(\phi^2_{\Bbbk,A})^{-1}=l_{F(A)}.
\end{split}
\]
\end{proof}

\begin{thm}\label{magen}
    Let $H$ be a cocommutative color Hopf algebra. Then the color Hopf subalgebras $K\subseteq H$ and the quotient color left $H$-module coalgebras $Q$ of $H$ are in 1-1 correspondence, under $K\mapsto H/HK^{+}$, $Q\mapsto H^{\mathrm{co}Q}$. 
\end{thm}

\begin{proof}
First of all we show that the two maps are well-defined. So, given a color Hopf subalgebra $K\subseteq H$, we know that $HK^{+}$ is a graded left ideal of $H$ and also a two-sided coideal of $H$ since $HK^{+}=\mathrm{Im}(m\circ(\mathrm{Id}_{H}\otimes i)-l\circ(\mathrm{Id}_{H}\otimes\epsilon))$, where $i:K\to H$ is the inclusion (see \cite[Proposition 1.4.8]{book5}). Furthermore, recalling the monoidal structure of $_{H}\mathrm{Vec}_{G}$ given in Remark \ref{ms}, it is not difficult to see that $\Delta_{H/HK^{+}}$ and $\epsilon_{H/HK^{+}}$ are morphisms of left $H$-modules. Indeed, given $\pi:H\to H/HK^{+}$ in $\mathrm{Comon}_{\mathrm{coc}}(\mathrm{Vec}_{G})$, for every $a,b\in H$ we have that
\[
\begin{split}
    \Delta_{H/HK^{+}}(h\cdot(a+HK^{+}))&=\Delta_{H/HK^{+}}(ha+HK^{+})=(\pi\otimes\pi)\Delta_{H}(ha)\\&=\phi(|h_{2}|,|a_{1}|)(h_{1}a_{1}+HK^{+})\otimes(h_{2}a_{2}+HK^{+})\\&=h\cdot\Delta_{H/HK^{+}}(a+HK^{+})
\end{split}
\]
and 
\[
\epsilon_{H/HK^{+}}(h\cdot(a+HK^{+}))=\epsilon_{H/HK^{+}}(ha+HK^{+})=\epsilon_{H}(ha)=\epsilon_{H}(h)\epsilon_{H}(a)=h\cdot\epsilon_{H/HK^{+}}(a+HK^{+}),
\]
thus $H/HK^{+}$ is a quotient color left $H$-module coalgebra of $H$. Furthermore, let $Q$ be a quotient color left $H$-module coalgebra of $H$; we show that
\[
H^{\mathrm{co}Q}=\{h\in H\ |\ (\mathrm{Id}_{H}\otimes\pi)\Delta(h)=h\otimes\pi(1_{H})\}=\{h\in H\ |\ (\mathrm{Id}_{H}\otimes\pi)\Delta(h)=(\mathrm{Id}_{H}\otimes\pi u\epsilon)\Delta(h)\}
\]
is a color Hopf subalgebra of $H$. It is in Comon$_{\mathrm{coc}}(\mathrm{Vec}_{G})$ since it is the equalizer object of the pair $(\pi,\pi\circ u_{H}\circ\epsilon_{H})$ in Comon$_{\mathrm{coc}}(\mathrm{Vec}_{G})$. In addition, by Proposition \ref{ossi2}, we know that $F(H^{\mathrm{co}Q})$ is the equalizer of the pair $(F(\pi),F(\pi\circ u_{H}\circ\epsilon_{H}))=(F(\pi),F(\pi)\circ u_{F(H)}\circ\epsilon_{F(H)})$ in Comon$_{\mathrm{coc}}(\mathrm{Vec}_{\mathbb{Z}_{2}})$, i.e. $F(H^{\mathrm{co}Q})=F(H)^{\mathrm{co}F(Q)}$. But now $F(H)$ is a super Hopf algebra and $F(Q)$ is a quotient super coalgebra and a quotient super left $F(H)$-module of $F(H)$ by Lemma \ref{module}. Furthermore, $\Delta_{F(Q)}$ and $\epsilon_{F(Q)}$ are morphisms of left $F(H)$-modules. Indeed, recalling the hexagon relation
\[
\phi^2_{H\otimes Q,H\otimes Q}\circ F(\mathrm{Id}_{H}\otimes c_{H,Q}\otimes\mathrm{Id}_{Q})\circ(\phi^2_{H\otimes H,Q\otimes Q})^{-1}=((\phi^2_{H,Q})^{-1}\otimes(\phi^2_{H,Q})^{-1})\circ(\mathrm{Id}_{F(H)}\otimes c'_{F(H),F(Q)}\otimes\mathrm{Id}_{F(Q)})\circ(\phi^2_{H,H}
\otimes\phi^2_{Q,Q})
\]
which holds true since $F$ is a braided strong monoidal functor, we obtain that
\[
\begin{split}
    &(\mu'\otimes\mu')\circ(\mathrm{Id}_{F(H)}\otimes c'_{F(H),F(Q)}\otimes\mathrm{Id}_{F(Q)})\circ(\Delta_{F(H)}\otimes\mathrm{Id}_{F(Q)}\otimes\mathrm{Id}_{F(Q)})\circ(\mathrm{Id}_{F(H)}\otimes\Delta_{F(Q)})=\\&(F(\mu)\otimes F(\mu))\circ((\phi^2_{H,Q})^{-1}\otimes(\phi^2_{H,Q})^{-1})\circ(\mathrm{Id}_{F(H)}\otimes c'_{F(H),F(Q)}\otimes\mathrm{Id}_{F(Q)})\circ(\phi^2_{H,H}\otimes\phi^2_{Q,Q})\circ(F(\Delta_{H})\otimes F(\Delta_{Q}))=\\&(F(\mu)\otimes F(\mu))\circ\phi^2_{H\otimes Q,H\otimes Q}\circ F(\mathrm{Id}_{H}\otimes c_{H,Q}\otimes\mathrm{Id}_{Q})\circ(\phi^2_{H\otimes H,Q\otimes Q})^{-1}\circ(F(\Delta_{H})\otimes F(\Delta_{Q}))=\\&\phi^2_{Q,Q}\circ F(\mu\otimes\mu)\circ F(\mathrm{Id}_{H}\otimes c_{H,Q}\otimes\mathrm{Id}_{Q})\circ F(\Delta_{H}\otimes\Delta_{Q})\circ(\phi^2_{H,Q})^{-1}=\\&\phi^2_{Q,Q}\circ F(\Delta_{Q}\circ\mu)\circ(\phi^2_{H,Q})^{-1}=\Delta_{F(Q)}\circ\mu'
\end{split}
\]
and
\[
\epsilon_{F(Q)}\circ\mu'=F(\epsilon_{Q})\circ F(\mu)\circ(\phi^2_{H,Q})^{-1}=F(r_{\Bbbk})\circ F(\epsilon_{H}\otimes\epsilon_{Q})\circ(\phi^2_{H,Q})^{-1}=r_{\Bbbk}\circ(\epsilon_{F(H)}\otimes\epsilon_{F(Q)}).
\]
Thus $F(Q)$ is a quotient super left $F(H)$-module coalgebra of $F(H)$ and then, by Theorem \ref{ma}, we have that $F(H)^{\mathrm{co}F(Q)}=F(H^{\mathrm{co}Q})$ is a super Hopf subalgebra of $F(H)$; thus $H^{\mathrm{co}Q}$ is a color Hopf subalgebra of $H$ by Lemma \ref{sub}. Thus the two maps are well-defined and now we want to prove that they are inverse to each other. So we compute 
\[
K\mapsto H/HK^{+}\mapsto H^{\mathrm{co}\frac{H}{HK^{+}}}
\]
and 
\[
F(H^{\mathrm{co}\frac{H}{HK^{+}}})=F(H)^{\mathrm{co}F(\frac{H}{HK^{+}})}=F(H)^{\mathrm{co}\frac{F(H)}{F(H)F(K)^{+}}}=\psi_{F(H)}(\phi_{F(H)}(F(K)))=F(K)
\]
since $F(K)$ is a super Hopf subalgebra of $F(H)$ by Lemma \ref{sub}. Thus, since $K$ and $H^{\mathrm{co}\frac{H}{HK^{+}}}$ are color Hopf subalgebras of $H$ and they are the same vector space by the previous equality because $F$ does not change the structure of vector space, they must be the same object in $\Hc$ by Lemma \ref{injsub}. Furthermore we compute
\[
Q\mapsto H^{\mathrm{co}Q}\mapsto H/H(H^{\mathrm{co}Q})^{+}
\]
and
\[
F\big(\frac{H}{H(H^{\mathrm{co}Q})^{+}}\big)=\frac{F(H)}{F(H)F(H^{\mathrm{co}Q})^{+}}=\frac{F(H)}{F(H)(F(H)^{\mathrm{co}F(Q)})^{+}}=\phi_{F(H)}(\psi_{F(H)}(F(Q)))=F(Q)
\]
since $F(Q)$ is a quotient super left $F(H)$-module coalgebra of $F(H)$. We know that $Q$ and $H/H(H^{\mathrm{co}Q})^{+}$ are quotient color left $H$-module coalgebras of $H$ and so, since they are the same vector space by the previous equality, they must be the same quotient color left $H$-module coalgebra of $H$ by Lemma \ref{injsub}.
\end{proof}

We call the two maps $\phi_{H}:K\mapsto H/HK^{+}$ and $\psi_{H}:Q\mapsto H^{\mathrm{co}Q}$ as for Theorem \ref{ma}. The bijection restricts to a 1-1 correspondence between normal color Hopf subalgebras and quotient color Hopf algebras as it is shown in the following result. Thus we extend Theorem \ref{ma} and Corollary \ref{c1} to the case of cocommutative color Hopf algebras.

\begin{cor}\label{kernel}
For a color Hopf subalgebra $B\subseteq A$ of a cocommutative color Hopf algebra $A$, the following conditions are equivalent: $\\$ $\\$
(1) $B$ is a normal color Hopf subalgebra, i.e. $\phi(|a_{2}|,|b|)a_{1}bS(a_{2})\in B$ for every $a\in A$ and $b\in B$; $\\$ $\\$
(2) $A/AB^{+}$ is a quotient color Hopf algebra; $\\$ $\\$
(3) the inclusion morphism $B\to A$ is the categorical kernel of some morphism in $\Hc$.
\end{cor}

\begin{proof}
(1)$\implies$(2). Let $B$ be a normal color Hopf subalgebra of $A$ and consider the quotient color left $A$-module coalgebra $A/AB^{+}$. In order to show that this is a quotient color Hopf algebra we have to prove that $AB^{+}$ is a right ideal of $A$ and that it is closed under the antipode of $A$. First we show that it is a right ideal. For any $a,a'\in A$ and $b\in B^{+}$, since $(S\otimes S)\circ\Delta=\Delta\circ S$ with $A$ cocommutative, we obtain that  
\[
\begin{split}
(ab)a'&=ab\epsilon(a'_{1})a'_{2}=\phi(|b|,|a'_{1}|)a\epsilon(a'_{1})ba'_{2}=\phi(|b|,|a'_{1}|)aa'_{1_{1}}S(a'_{1_{2}})bS(S(a'_{2}))\\&=\phi(|b|,|a'_{1}|)\phi(|b|,|S(a'_{2_{1}})|)aa'_{1}S(a'_{2_{1}})bS(S(a'_{2_{2}}))=\phi(|b|,|a'_{1}|)\phi(|b|,|S(a'_{2})_{1}|)aa'_{1}S(a'_{2})_{1}bS(S(a'_{2})_{2})\\&=
\phi(|b|,|a'|)\phi(|S(a'_{2})_{2}|,|b|)aa'_{1}S(a'_{2})_{1}bS(S(a'_{2})_{2})=\phi(|b|,|a'|)aa'_{1}\xi_{A}(S(a'_{2})\otimes b)\in AB^{+}. 
\end{split}
\]
To see that $AB^{+}$ is stable under the antipode, note that $S(ab)=\phi(|a|,|b|)S(b)S(a)$. From $\epsilon(S(b))=\epsilon(b)=0$ with $b\in B^{+}$ it follows that $S(b)\in B^{+}\subseteq AB^{+}$ and therefore $\phi(|a|,|b|)S(b)S(a)\in AB^{+}$ since we have just shown that $AB^{+}$ is a right ideal. $\\$
(2)$\implies$(3). Since $A/AB^{+}$ is a quotient color Hopf algebra and then $\pi:A\to A/AB^{+}$ is in $\Hc$, we have that $A^{\mathrm{co}\frac{A}{AB^{+}}}=\{a\in A\ |\ a_{1}\otimes\pi(a_{2})=a\otimes1_{A/AB^{+}}\}=\mathrm{Hker}(\pi)$ in $\Hc$. But now $A^{\mathrm{co}\frac{A}{AB^{+}}}=\psi_{A}(\phi_{A}(B))=B$ by Theorem \ref{magen} and then $B=\mathrm{Hker}(\pi)$. Hence $(B,j)$ is the kernel of $\pi$ in $\Hc$, where $j:B\to A$ is the canonical inclusion. $\\$
We already know by Lemma \ref{31} that (3)$\implies$(1) holds true and then we are done.
\end{proof}

\begin{invisible}
Furthermore, in the case of kernels we can prove something more as we will show in Proposition \ref{kernel}.

\begin{lem}
    Given a cocommutative color Hopf algebra $D$ and a color Hopf subalgebra $C\subseteq D$, if $F(C)$ is a normal super Hopf subalgebra of $F(D)$, then $D/DC^{+}$ is a quotient color Hopf algebra.
\end{lem}

\begin{proof}
    We know that $D/DC^{+}$ is in Vec$_{G}$ and that $DC^{+}$ is a left ideal and a two-sided coideal since $DC^{+}=\mathrm{Im}(m\circ(\mathrm{Id}\otimes i)-l\circ(\mathrm{Id}\otimes\epsilon))$, where $i:C\to D$ is the inclusion (see \cite[Proposition 1.4.8]{book5}), thus $D/DC^{+}$ is in $\mathrm{Comon}_{\mathrm{coc}}(\mathrm{Vec}_{G})$. Now, since $F(C)$ is a normal super Hopf subalgebra of $F(D)$, we have that $F(D)/F(D)F(C)^{+}=F(D)/F(DC^{+})$ is in $\mathrm{Hopf}_{\mathrm{coc}}({\mathrm{Vec}}_{\mathbb{Z}_{2}})$ by Corollary \ref{c1}, then $D/DC^{+}$ is in $\Hc$ by Lemma \ref{lemmino}. 
\end{proof}

\begin{prop}
An inclusion $i:C\to D$ in $\Hc$ is a kernel if and only if $F(i):F(C)\to F(D)$ is a kernel in $\mathrm{Hopf}_{\mathrm{coc}}(\mathrm{Vec}_{\mathbb{Z}_{2}})$. 
\end{prop}

\begin{proof}
We know that if $i$ is a kernel in $\Hc$ then $F(i)$ is a kernel in $\mathrm{Hopf}_{\mathrm{coc}}(\mathrm{Vec}_{\mathbb{Z}_{2}})$ by Proposition \ref{ossi2}, so we prove the other direction. If the inclusion $F(i):F(C)\to F(D)$ is a kernel in Hopf$_{\mathrm{coc}}(\mathrm{Vec}_{\mathbb{Z}_{2}})$, then by Corollary \ref{c1} we obtain that $F(C)$ is a normal super Hopf subalgebra of $F(D)$, so that $\pi:D\to D/DC^{+}$ is in $\Hc$ by the previous lemma and, by the proof of Corollary \ref{c1}, we obtain that $F(i)$ is the kernel of $F(\pi)$ in $\mathrm{Hopf}_{\mathrm{coc}}(\mathrm{Vec}_{\mathbb{Z}_2})$. Now, we know that $C=C^{+}\oplus\Bbbk1_{D}$, since $x=x-\epsilon_{D}(x)1_{D}+\epsilon_{D}(x)1_{D}$ for $x\in C$, and then
\[
\pi(x)=\pi(1_{D}(x-\epsilon_{D}(x)1_{D})+\epsilon_{D}(x)1_{D})=\epsilon_{D}(x)\pi(1_{D})=\epsilon_{D}(x)1_{D/DC^{+}}=u_{D/DC^{+}}\epsilon_{D}(x),
\]
i.e. $\pi\circ i=u_{D/DC^{+}}\circ\epsilon_{D}\circ i$, hence $i$ is the kernel of $\pi$ in $\Hc$ by Proposition \ref{ossi2}. 
\end{proof}

Now we can prove (1) and (2) of Lemma \ref{l1}. So let $f:A\to B$ be a morphism in $\Hc$. 
\end{invisible}

\begin{invisible}
Since $\mathrm{ker}(f)$ is graded by Remark \ref{Im and qu}, the same is true for $F(\mathrm{ker}(f))=\mathrm{ker}(F(f))$, which is also a two-sided ideal and coideal of $F(A)$ and closed under the antipode, so that $F(A)/\mathrm{ker}(F(f))=F(A/\mathrm{ker}(f))$ in Hopf$_{\mathrm{coc}}(\text{Vec}_{\mathbb{Z}_{2}}$) by Lemma \ref{lemmino}. Now $F(A)/\mathrm{ker}(F(f))$ is a quotient super left $F(A)$-module coalgebra; indeed it is a quotient super coalgebra and a super left $F(A)$-module since it is a super algebra and $F(\pi):F(A)\to F(A)/\mathrm{ker}(F(f))$ of super algebras, with $\pi:A\to A/\mathrm{ker}(f)$ in $\Hc$. The left $F(A)$-module action $\mu_{F(A)/\mathrm{ker}(F(f))}:F(A)\otimes F(A)/\mathrm{ker}(F(f))\to F(A)/\mathrm{ker}(F(f))$ is given by $h\otimes(a+\mathrm{ker}(F(f)))\mapsto h\cdot(a+\mathrm{ker}(F(f))):=h\cdot a+\mathrm{ker}(F(f))$. Furthermore we show that $\epsilon_{F(A)/\mathrm{ker}(F(f))}$ and $\Delta_{F(A)/\mathrm{ker}(F(f))}$ are morphisms of left $F(A)$-modules. 
\end{invisible}

\begin{invisible}
We know that $A/\mathrm{ker}(f)$ is in $\Hc$ (see Remark \ref{f(A)}), then $F(A/\mathrm{ker}(f))=F(A)/\mathrm{ker}(F(f))$ is in $\mathrm{Hopf}_{\mathrm{coc}}(\mathrm{Vec}_{\mathbb{Z}_2})$. In particular, it is a quotient super left $F(A)$-module coalgebra, with left $F(A)$-action given by $h\cdot(a+\mathrm{ker}(F(f))):=h\cdot a+\mathrm{ker}(F(f))$ (it is not difficult to see that $\epsilon_{F(A)/\mathrm{ker}(F(f))}$ and $\Delta_{F(A)/\mathrm{ker}(F(f))}$ are left $F(A)$-linear, recalling the monoidal structure of $_{F(A)}\mathrm{Vec}_{\mathbb{Z}_2}$ given in Remark \ref{ms}). In the following, we sometimes avoid to put $F$ for objects and morphisms in order to simplify the notation, denoting differently only the multiplication and the comultiplication of $F(A)$. 
\end{invisible}

\begin{invisible}
Recalling the monoidal structure of $_{A}\mathrm{Vec}_{\mathbb{Z}_2}$ given in Remark \ref{ms}, we have that
\[
\epsilon_{A/\mathrm{ker}(f)}(h\cdot(a+\mathrm{ker}(f)))=\epsilon_{A/\mathrm{ker}(f)}(h\cdot a+\mathrm{ker}(f))=\epsilon_{A}(h\cdot a)=\epsilon_{A}(h)\cdot\epsilon_{A}(a)=h\cdot\epsilon_{A/\mathrm{ker}(f)}(a+\mathrm{ker}(f))
\]
and
\[
\begin{split}
\Delta_{A/\mathrm{ker}(f)}(h\cdot(a+\mathrm{ker}(f)))&=\Delta_{A/\mathrm{ker}(f)}(h\cdot a+\mathrm{ker}(f))=(\pi\otimes\pi)\Delta_{F(A)}(h\cdot a)\\&=(-1)^{|\bar{h}_{2}||\bar{a}_{1}|}(\bar{h}_{1}\cdot\bar{a}_{1}+\mathrm{ker}(f))\otimes(\bar{h}_{2}\cdot\bar{a}_{2}+\mathrm{ker}(f))\\&=(\mu_{A/\mathrm{ker}(f)}\otimes\mu_{A/\mathrm{ker}(f)})((-1)^{|\bar{h}_{2}||\bar{a}_{1}+\mathrm{ker}(f)|}\bar{h}_{1}\otimes(\bar{a}_{1}+\mathrm{ker}(f))\otimes\bar{h}_{2}\otimes(\bar{a}_{2}+\mathrm{ker}(f)))\\&=(\mu_{A/\mathrm{ker}(f)}\otimes\mu_{A/\mathrm{ker}(f)})(\mathrm{Id}\otimes c_{A,A/\mathrm{ker}(f)}\otimes\mathrm{Id})(\bar{h}_{1}\otimes \bar{h}_{2}\otimes(\bar{a}_{1}+\mathrm{ker}(f))\otimes(\bar{a}_{2}+\mathrm{ker}(f)))\\&=(\mu_{A/\mathrm{ker}(f)}\otimes\mu_{A/\mathrm{ker}(f)})(\mathrm{Id}\otimes c_{A,A/\mathrm{ker}(f)}\otimes\mathrm{Id})(\Delta_{F(A)}\otimes\Delta_{A/\mathrm{ker}(f)})(h\otimes(a+\mathrm{ker}(f)))\\&=h\cdot\Delta_{A/\mathrm{ker}(f)}(a+\mathrm{ker}(f)).
\end{split}
\]
\end{invisible}

\begin{invisible}
We can apply Theorem \ref{ma} in order to obtain that $A/A(A^{\mathrm{co}\frac{A}{\mathrm{ker}(f)}})^{+}=A/\mathrm{ker}(f)$, where 
\[
A^{\mathrm{co}\frac{A}{\mathrm{ker}(f)}}=\{a\in A\ |\ (\mathrm{Id}_{A}\otimes\pi)\Delta_{F(A)}(a)=(\mathrm{Id}_{A}\otimes\pi u_{F(A)}\epsilon_{F(A)})\Delta_{F(A)}(a)\}=\mathrm{Hker}(F(\pi)), 
\]
since $F(\pi)$ is a morphism of super algebras. Furthermore, there is a unique map of super vector spaces $\bar{f}:F(A/\mathrm{ker}(f))\to F(B),\ a+\mathrm{ker}(f)\mapsto f(a)$ such that $\bar{f}\circ F(\pi)=F(f)$ and $\bar{f}$ is injective. 
\[
\begin{tikzcd}
	F(A) & F(A/\mathrm{ker}(f))\\
	F(B)
	\arrow[from=1-1, to=2-1, "F(f)"']
	\arrow[from=1-1, to=1-2, "F(\pi)"]
	\arrow[from=1-2, to=2-1,dotted, "\bar{f}"]
\end{tikzcd}
\]
Now if $a\in A^{\mathrm{co}\frac{A}{\mathrm{ker}(f)}}$ we obtain that $(\mathrm{Id}_{A}\otimes f)\Delta_{F(A)}(a)=(\mathrm{Id}_{A}\otimes u_{F(B)}\epsilon_{F(A)})\Delta_{F(A)}(a)$ by composing with $\mathrm{Id}_{A}\otimes\bar{f}$ and using that $f\circ u_{F(A)}=u_{F(B)}$, so $a\in\mathrm{Hker}(F(f))$. On the other hand if $a\in\mathrm{Hker}(F(f))$, then $(\mathrm{Id}_{A}\otimes\bar{f})(\mathrm{Id}_{A}\otimes\pi)\Delta_{F(A)}(a)=(\mathrm{Id}_{A}\otimes\bar{f})(\mathrm{Id}_{A}\otimes\pi u_{F(A)}\epsilon_{F(A)})\Delta_{F(A)}(a)$ using again $f\circ u_{F(A)}=u_{F(B)}$ and we obtain that $a\in A^{\mathrm{co}\frac{A}{\mathrm{ker}(f)}}$ by injectivity of $\bar{f}$, thus $A^{\mathrm{co}\frac{A}{\mathrm{ker}(f)}}=\mathrm{Hker}(F(f))$. Hence we have that $F(A/\mathrm{ker}(f))=F(A)/F(A)(\mathrm{Hker}(F(f)))^{+}$ in Hopf$_{\mathrm{coc}}(\text{Vec}_{\mathbb{Z}_{2}})$, which is also $F(A)/F(A)(\mathrm{Hker}(F(f)))^{+}F(A)$. $\\$
\end{invisible}

Now we can prove (1) and (2) of Lemma \ref{l1}. So let $f:A\to B$ be a morphism in $\Hc$ and consider the factorization $f=i\circ p$ in $\Hc$ obtained by taking $p:A\to A/A(\mathrm{Hker}(f))^{+}A$ as the cokernel of the kernel of $f$ in $\Hc$. We have that 
$A/\mathrm{ker}(f)$ is a quotient color Hopf algebra by Remark \ref{f(A)}. Thus, since $\pi:A\to A/\mathrm{ker}(f)$ is in $\Hc$, we have that $A^{\mathrm{co}\frac{A}{\mathrm{ker}(f)}}=\mathrm{Hker}(\pi)$ in $\Hc$. Furthermore, there exists a unique map $\bar{f}:A/\mathrm{ker}(f)\to B,\ a+\mathrm{ker}(f)\mapsto f(a)$ in Vec$_{G}$ and then in $\Hc$ such that $\bar{f}\circ\pi=f$ and $\bar{f}$ is injective. Thus $\mathrm{Hker}(\pi)=\mathrm{Hker}(f)$ in $\Hc$ and so we obtain
\[
A/\mathrm{ker}(f)=\phi_{A}(\psi_{A}(A/\mathrm{ker}(f)))=\phi_{A}(A^{\mathrm{co}\frac{A}{\mathrm{ker}(f)}})=\phi_{A}(\mathrm{Hker}(f))=A/A(\mathrm{Hker}(f))^{+}.
\]
Then $\mathrm{ker}(f)=A(\mathrm{Hker}(f))^{+}$ and so also $\mathrm{ker}(f)=A(\mathrm{Hker}(f))^{+}A$, hence $p=\pi$ and $i=\bar{f}$ is injective and then a monomorphism in $\Hc$. Observe that, clearly, $\mathrm{Hker}(f)^{+}\subseteq\mathrm{ker}(f)$ since with $x\in\mathrm{Hker}(f)$ we have that $f(x)=\epsilon(x)1_{B}$. Thus we have obtained a decomposition $f=i\circ p$ in $\Hc$ with $i$ a monomorphism in $\Hc$ and $p$ a regular epimorphism in $\Hc$, so (1) of Lemma \ref{l1} is proved. 

\begin{invisible}
Since $F$ preserves kernels and cokernels by Propositions \ref{ossi2} and \ref{coequalizer}, we obtain that $F(\pi)=F(p)$, then $F(i)=\bar{f}$ by uniqueness, which is injective and so $i$ is injective. Hence we have obtained a decomposition $f=i\circ p$ in $\Hc$ with $i$ injective and then a monomorphism in $\Hc$ and $p$ a cokernel and then a regular epimorphism in $\Hc$, thus every $f$ in $\Hc$ has a factorization reg epi-mono. Hence (1) of Lemma \ref{l1} is proved. Observe that, clearly, $\mathrm{Hker}(f)^{+}\subseteq\mathrm{ker}(f)$ since with $x\in\mathrm{Hker}(f)$ we have that $f(x)=\epsilon(x)1_{B}$, thus $A(\mathrm{Hker}(f))^{+}A\subseteq\mathrm{ker}(f)$ and now, since as vector spaces $A/\mathrm{ker}(f)=A/A(\mathrm{Hker}(f))^{+}A$, we have that $\mathrm{ker}(f)=A(\mathrm{Hker}(f))^{+}A$. 
\end{invisible}

\begin{lem}\label{monepi}
In $\Hc$ regular epimorphisms are exactly the surjective morphisms and monomorphisms are exactly the injective morphisms.
\end{lem}

\begin{proof}
A regular epimorphism in $\Hc$ is surjective since in $\Hc$ every coequalizer is a projection and if $f$ in $\Hc$ is surjective, by its decomposition $f=i\circ p$ in $\Hc$, we obtain that $i$ is surjective and then it is an isomorphism in $\Hc$, hence $f$ is a cokernel, so a regular epimorphism. Furthermore, in $\Hc$ an injective map is clearly a monomorphism and also the vice versa is true. Indeed, given $f:A\to B$ a monomorphism in $\Hc$ and $j:\mathrm{Hker}(f)\to A$ the categorical kernel of $f$, from $f\circ j=f\circ u_{A}\circ\epsilon_{A}\circ j$ we obtain that $j=u_{A}\circ\epsilon_{A}\circ j$, i.e., if $a\in\mathrm{Hker}(f)$ then $a=\epsilon_{A}(a)1_{A}$, hence $a\in\Bbbk1_{A}$. Clearly every $x\in\Bbbk1_{A}$ is in $\mathrm{Hker}(f)$ and then $\mathrm{Hker}(f)=\Bbbk1_{A}$, so $\mathrm{Hker}(f)^{+}=0$ and then, since $\mathrm{ker}(f)=A(\mathrm{Hker}(f))^{+}A$, we have that $\mathrm{ker}(f)=0$, thus $f$ is injective. 
\end{proof}

Now condition (2) of Lemma \ref{l1} is easily satisfied. In $\Hc$ binary products are tensor products, so $E\times A=E\otimes A$ and $E\times B=E\otimes B$ and the induced arrow $\mathrm{Id}_{E}\times f$ is the map $\mathrm{Id}_{E}\otimes f$. Indeed
\[
l_{B}\circ(\epsilon_{E}\otimes\mathrm{Id}_{B})\circ(\mathrm{Id}_{E}\otimes f)=l_{B}\circ(\mathrm{Id}_{\Bbbk}\otimes f)\circ(\epsilon_{E}\otimes\mathrm{Id}_{A})=f\circ l_{A}\circ(\epsilon_{E}\otimes\mathrm{Id}_{A})
\]
and $r_{E}\circ(\mathrm{Id}_{E}\otimes\epsilon_{B})\circ(\mathrm{Id}_{E}\otimes f)=\mathrm{Id}_{E}\circ r_{E}\circ(\mathrm{Id}_{E}\otimes\epsilon_{A})$. So, given a regular epimorphism $f:A\to B$ in $\Hc$, this is surjective, then $\mathrm{Id}_{E}\otimes f$ is surjective, hence a regular epimorphism by Lemma \ref{monepi}. Now, in order to obtain the regularity of $\Hc$, we show (3) of Lemma \ref{l1} by proving the stability of surjective maps along injective ones under pullbacks. $\\$

If we have a morphism $p:A\to B$ in $\Hc$ and we consider $C\subseteq B$ a color Hopf subalgebra of $B$, then the subspace $p^{-1}(C)$ of $A$ defined as in \cite{article23} by $$p^{-1}(C)=\{x\in A|\ (p\otimes\mathrm{Id}_{A})\Delta(x)\in C\otimes A\}$$ is a color Hopf subalgebra of $A$. Indeed observe that
\[
p^{-1}(C)=((p\otimes\mathrm{Id}_{A})\Delta)^{-1}(C\otimes A)=((p\otimes\mathrm{Id}_{A})\Delta)^{-1}(\bigoplus_{g\in G}{(C\otimes A)_{g}})=\bigoplus_{g\in G}{((p\otimes\mathrm{Id}_{A})\Delta)^{-1}((C\otimes A)_{g})}
\]
since the maps preserve gradings. Hence $p^{-1}(C)=\bigoplus_{g\in G}{P_{g}}$ is a graded vector space where 
\[
P_{g}=\{x\in A_{g}\ |\ (p\otimes\mathrm{Id}_{A})\Delta(x)\in (C\otimes A)_{g}\}=((p\otimes\mathrm{Id}_{A})\Delta)^{-1}((C\otimes A)_{g}).
\]
By Remark \ref{o1}, we only have to show that $p^{-1}(C)$ is closed under $\Delta_{A}$, $m_{A}$ and $S_{A}$ and that it contains $1_{A}$. Clearly it contains $1_{A}$ and it is closed under $m_{A}$ since $(p\otimes\mathrm{Id}_{A})\circ\Delta$ is in Mon(Vec$_{G}$), indeed
\[
(p\otimes\mathrm{Id}_{A})\Delta(xy)=m_{C\otimes A}((p\otimes\mathrm{Id}_{A})\Delta\otimes(p\otimes\mathrm{Id}_{A})\Delta)(x\otimes y)\in C\otimes A
\]
with $x,y\in p^{-1}(C)$. It is also easy to see closure under antipode since, in the cocommutative case, we have that $\Delta(S(x))=(S\otimes S)\Delta(x)$, so with $x\in p^{-1}(C)$ we have
\[
(p\otimes\mathrm{Id}_{A})\Delta(S_{A}(x))=(p\otimes\mathrm{Id}_{A})(S_{A}\otimes S_{A})\Delta(x)=(S_{B}\otimes S_{A})(p\otimes\mathrm{Id}_{A})\Delta(x)\in S_{B}(C)\otimes S_{A}(A)\subseteq C\otimes A.
\]
Finally we have to show that $\Delta(p^{-1}(C))\subseteq p^{-1}(C)\otimes p^{-1}(C)$ and, since $A$ is cocommutative, we only have to prove that $\Delta(p^{-1}(C))\subseteq p^{-1}(C)\otimes A$. But now, given $x\in p^{-1}(C)$, we have that
\[
(p\otimes\mathrm{Id}_{A})\Delta(x_{1})\otimes x_{2}=p(x_{1})\otimes x_{2_{1}}\otimes x_{2_{2}}=(\mathrm{Id}_{B}\otimes\Delta)(p\otimes\mathrm{Id}_{A})\Delta(x)\in(\mathrm{Id}_{B}\otimes\Delta)(C\otimes A)\subseteq C\otimes A\otimes A
\]
and then $\Delta(x)\in p^{-1}(C)\otimes A$, so that $p^{-1}(C)$ is closed under $\Delta_{A}$ and hence it is a color Hopf subalgebra of $A$. Now here we show some results which generalize those given in \cite{article1} for the case of Hopf$_{\Bbbk,\mathrm{coc}}$. The following Lemma \ref{color} and Lemma \ref{l11} correspond to \cite[Lemma 2.5]{article1} and \cite[Lemma 2.6]{article1} respectively and they have the same proof, which we report for the sake of completeness and in order to show that there are no problems with the respective generalizations. 
\begin{invisible}
First observe that, clearly, with $A=\bigoplus_{g\in G}{A_{g}}$ in $\Hc$ then $p(A)=\bigoplus_{g\in G}{p(A_{g})}$ is in $\Hc$ since it is a (graded) subalgebra and a (graded) subcoalgebra of $B$ and it is closed under the antipode of $B$.
\end{invisible}

\begin{lem}\label{color}
Given a morphism $p:A\to B$ in $\Hc$, we have the following facts: $\\$ $\\$
1) For all color Hopf subalgebras $C\subseteq B$, $p(p^{-1}(C))\subseteq C$. $\\$ $\\$
2) For all color Hopf subalgebras $D\subseteq A$, $D\subseteq p^{-1}(p(D))$. $\\$ $\\$
3) For all color Hopf subalgebras $C\subseteq B$, then $C=p(p^{-1}(C))$ if and only if $C=p(D)$, for some $D\subseteq A$ color Hopf subalgebra.
\end{lem}

\begin{proof}
If $x\in p^{-1}(C)$, i.e. $ p(x_{1})\otimes x_{2}\in C\otimes A$, then $p(x)\in C$, so 1) is shown. Recall that if $D$ is a color Hopf subalgebra of $A$ then $p(D)$ is a color Hopf subalgebra of $B$ by Remark \ref{f(A)}. If $d\in D$ we have that $(p\otimes\mathrm{Id}_{A})\Delta(d)\in p(D)\otimes D\subseteq p(D)\otimes A$ and then also 2) is proved. Finally if $C=p(p^{-1}(C))$ clearly we can take $D=p^{-1}(C)$ while if $C=p(D)$ for some color Hopf subalgebra $D$ of $A$ then $D\subseteq p^{-1}(p(D))=p^{-1}(C)$ by 2) and by applying $p$ one gets $C=p(D)\subseteq p(p^{-1}(C))$ and, since $p(p^{-1}(C))\subseteq C$ by 1), we have $C=p(p^{-1}(C))$ and we obtain 3).
\end{proof}

\begin{lem}\label{l11}
Given $p:A\to B$ in $\Hc$ and an inclusion $i:C\to B$ in $\Hc$, then the diagram
\[
\begin{tikzcd}
  p^{-1}(C) \arrow[r, "\tilde{p}"] \arrow["j"',d]
    & C \arrow[d, "i"] \\
  A \arrow[r,"p"']
& B 
\end{tikzcd}
\]
is a pullback in $\Hc$, where $j$ is the inclusion and $\tilde{p}$ is the restriction of $p$ to $p^{-1}(C)$.
\end{lem}

\begin{proof}
By 1) of Lemma \ref{color} the diagram is commutative. To check the universal property, consider two morphisms $\alpha:T\to A$ and $\beta:T\to C$ in $\Hc$ such that $p\circ\alpha=i\circ\beta$ and let us show that $\alpha(T)\subseteq p^{-1}(C)$. Then, taken $c:T\to p^{-1}(C)$ as $\alpha$ with codomain $p^{-1}(C)$ we have $j\circ c=\alpha$ and $i\circ\tilde{p}\circ c=p\circ j\circ c=p\circ\alpha=i\circ\beta$, hence $\tilde{p}\circ c=\beta$ since $i$ injective; clearly this $c$ is unique since we must have $j\circ c=\alpha$. 
\[
\begin{tikzcd}
  T
  \arrow[drr,bend left, "\beta"]
  \arrow[ddr, bend right, "\alpha"']
  \arrow[dr, dotted, "c" description] & & \\
    & p^{-1}(C) \arrow[r, "\tilde{p}"] \arrow[d, "j"']
      & C \arrow[d, "i"] \\
& A \arrow[r, "p"'] & B
\end{tikzcd}
\]
Thus we show that $\alpha(T)\subseteq p^{-1}(C)$. Given $t\in T$, since $\alpha$ is a morphism of coalgebras, we have
\[
(p\otimes\mathrm{Id}_{A})\Delta_{A}(\alpha(t))=(p\otimes\mathrm{Id}_{A})(\alpha\otimes\alpha)\Delta_{T}(t)=(i\otimes\mathrm{Id}_{A})(\beta\otimes\alpha)\Delta_{T}(t)\in C\otimes A,
\]
then the diagram is a pullback.
\end{proof}

\begin{invisible}
\begin{lem}\label{pull}
    The functor $F:\Hc\to\mathrm{Hopf}_{\mathrm{coc}}(\mathrm{Vec}_{\mathbb{Z}_{2}})$ preserves the pullback of Lemma \ref{l11}.
\end{lem}

\begin{proof}
    By Lemma \ref{l11} we know that, given $p:A\to B$ and an inclusion $i:C\to B$ in $\Hc$, then $(p^{-1}(C),j,\tilde{p})$ is the pullback of the pair $(p,i)$ in $\Hc$. Since by Lemma \ref{l11} applied for Hopf$_{\mathrm{coc}}(\mathrm{Vec}_{\mathbb{Z}_2})$ we have that $(p^{-1}(F(C)),\bar{j},\bar{p})$ is the pullback of the pair $(p,i)$ in Hopf$_{\mathrm{coc}}(\mathrm{Vec}_{\mathbb{Z}_2})$ with $\bar{j}:p^{-1}(F(C))\to F(A)$ the inclusion and $\bar{p}$ the restriction of $p$ to $p^{-1}(F(C))$, if we show that $F(p^{-1}(C))=p^{-1}(F(C))$ in $\mathrm{Hopf}_{\mathrm{coc}}(\mathrm{Vec}_{\mathbb{Z}_{2}})$ then we are done. Observe that $p^{-1}(F(C))$ is a super Hopf subalgebra of $F(A)$ and, since $p^{-1}(C)$ is a color Hopf subalgebra of $A$, also $F(p^{-1}(C))$ is a super Hopf subalgebra of $F(A)$ by Lemma \ref{sub}.

Now $p^{-1}(C)=\bigoplus_{g\in G}{P_{g}}$ with $P_{g}=((p\otimes\mathrm{Id})\Delta)^{-1}((C\otimes A)_{g})$, then $F(p^{-1}(C))=P_{0}\oplus P_{1}$ with $P_{0}=\bigoplus_{g\in G_{0}}{P_{g}}$ and $P_{1}=\bigoplus_{g\in G_{1}}{P_{g}}$. Thus we have 
\[
\begin{split}
P_{0}&=\bigoplus_{g\in G_{0}}{((p\otimes\mathrm{Id})\Delta)^{-1}((C\otimes A)_{g})}=((p\otimes\mathrm{Id})\Delta)^{-1}(\bigoplus_{g\in G_{0}}{(C\otimes A)_{g}})=((p\otimes\mathrm{Id})\Delta)^{-1}(F(C\otimes A)_{0})\\&=(F(p\otimes\mathrm{Id})(\phi^{2}_{A,A})^{-1}\Delta_{F(A)})^{-1}(F(C\otimes A)_{0})=((\phi^{2}_{B,A})^{-1}(F(p)\otimes\mathrm{Id})\Delta_{F(A)})^{-1}(F(C\otimes A)_{0})\\&=((F(p)\otimes\mathrm{Id})\Delta_{F(A)})^{-1}\phi^{2}_{C,A}(F(C\otimes A)_{0})=((p\otimes\mathrm{Id})\Delta_{F(A)})^{-1}((F(C)\otimes F(A))_{0})=p^{-1}(F(C))_{0}
\end{split}
\]
and similar $P_{1}=p^{-1}(F(C))_{1}$, hence $F(p^{-1}(C))=p^{-1}(F(C))$.

Now $p^{-1}(F(C))$ is composed by those $x\in F(A)$ such that $(F(p)\otimes\mathrm{Id}_{F(A)})\Delta_{F(A)}(x)\in F(C)\otimes F(A)$ and
\[
(F(p)\otimes\mathrm{Id}_{F(A)})\circ\Delta_{F(A)}=(F(p)\otimes\mathrm{Id}_{F(A)})\circ\phi^2_{A,A}\circ F(\Delta_{A})=\phi^2_{B,A}\circ F((p\otimes\mathrm{Id}_{A})\Delta_{A})
\]
and then, equivalently, by those $x\in F(A)$ such that $F((p\otimes\mathrm{Id}_{A})\Delta_{A})(x)\in F(C\otimes A)$ and these are exactly the elements of $F(p^{-1}(C))$.
\end{proof}
\end{invisible}

\begin{prop}\label{surjective}
Consider a surjective morphism $p:A\to B$ in $\Hc$ and an inclusion $i:C\to B$ in $\Hc$. Then the morphism $\tilde{p}$ in the pullback of Lemma \ref{l11} is also surjective.
\end{prop}

\begin{proof}
If we compute the pullback of the pair $(p,i)$ in $\Hc$ we obtain $(p^{-1}(C),j,\tilde{p})$ as in Lemma \ref{l11} and we want to show that $\tilde{p}$ is surjective if $p$ is surjective. 
Since $\tilde{p}$ is just given by the restriction of $p$, we have that $\tilde{p}$ is surjective if and only if $C=p(p^{-1}(C))$ and this is equivalent, with $C$ a color Hopf subalgebra of $B$, to prove that $C=p(D)$ for some color Hopf subalgebra $D$ of $A$ by 3) of Lemma \ref{color}. 
We know that the canonical projection $\pi:B\to B/BC^{+}$ is a quotient color left $B$-module coalgebra 
and, since $p$ is a morphism of color Hopf algebras, we have that $\pi\circ p$ is a morphism of color left $A$-module coalgebras, so that $A/\mathrm{ker}(\pi\circ p)$ is a quotient color left $A$-module coalgebra. We set $D:=A^{\mathrm{co}\frac{A}{\mathrm{ker}(\pi\circ p)}}=\psi_{A}(A/\mathrm{ker}(\pi\circ p))$, which is a color Hopf subalgebra of $A$ by Theorem \ref{magen}. Then we obtain
\[
A/AD^{+}=\phi_{A}(D)=\phi_{A}(\psi_{A}(A/\mathrm{ker}(\pi\circ p)))=A/\mathrm{ker}(\pi\circ p)
\]
by Theorem \ref{magen}, hence $AD^{+}=\mathrm{ker}(\pi\circ p)$. 
\begin{invisible}
Let $\pi_{1}:A\to A/\mathrm{ker}(\pi\circ p)$, then we have that there is a unique linear map $\bar{f}:A/\mathrm{ker}(\pi\circ p)\to B/BC^{+}$ such that $\bar{f}\circ\pi_{1}=\pi\circ p$ given by $\bar{f}(a+\mathrm{ker}(\pi\circ p))=p(a)+BC^{+}$, for every $a\in A$. 
\[
\begin{tikzcd}
  A \arrow[r, "\pi_1"] \arrow["p"',d]
    & \frac{A}{\mathrm{ker}(\pi\circ p)} \arrow[d,dotted,"\bar{f}"] \\
  B\arrow[r,"\pi"']
& \frac{B}{BC^{+}} 
\end{tikzcd}
\]
Now $a\in D$ is such that $(\mathrm{Id}_{A}\otimes\pi_{1})\Delta_{F(A)}(a)=a\otimes\pi_{1}(a)$ so, by composing with $l_{A}\circ(\epsilon_{A}\otimes\mathrm{Id})$, one obtain $\pi_{1}(a)=\pi_{1}(\epsilon_{A}(a)1_{A})$, thus if $a\in D^{+}$ then $\pi_{1}(a)=0$, i.e. $a\in\mathrm{ker}(\pi\circ p)$. Hence $D^{+}\subseteq\mathrm{ker}(\pi\circ p)$ and then $AD^{+}\subseteq\mathrm{ker}(\pi\circ p)$ so that, by the previous equality, $AD^{+}=\mathrm{ker}(\pi\circ p)$. 
\end{invisible}
Thus, since $p$ is a surjective morphism of algebras, we obtain 
\[
Bp(D)^{+}=p(A)p(D^{+})=p(AD^{+})=p(\mathrm{ker}(\pi\circ p))=\mathrm{ker}(\pi)=BC^{+}
\]
and then
\[
\phi_{B}(C)=B/BC^{+}=B/Bp(D)^{+}=\phi_{B}(p(D))
\]
so that, by applying $\psi_{B}$ and by using Theorem \ref{magen} again, we obtain that $C=p(D)$.
\end{proof}

We have shown the stability of surjective morphisms (i.e. regular epimorphisms by Lemma \ref{monepi}) along inclusions under pullbacks in $\Hc$. But now every injective morphism $f:C\to B$ in $\Hc$ can be decomposed as $i\circ\phi$ with $\phi$ an isomorphism between $C$ and $f(C)$ and $i$ the inclusion of $f(C)$ into $B$. Now, if we consider the pullback of $p$ along $f$ we have that, since the inner right square is a pullback too by Lemma \ref{l11}, also the left square is a pullback by \cite[Proposition 2.5.9]{book1}.
\[
\begin{tikzcd}
A\times_{B}C\arrow[r, "\tilde{\phi}"] \arrow[d,"\alpha"'] & p^{-1}(f(C))\arrow[r, "j"] \arrow[d,"\tilde{p}"]  & A \arrow[d,"p"] \\
C\arrow[r, "\phi"']          & f(C)\arrow[r, "i"']                                         & B
\end{tikzcd}
\]
Then, since $\phi$ is an isomorphism so is $\tilde{\phi}$ and from $\tilde{p}\circ\tilde{\phi}=\phi\circ\alpha$, we obtain that $\alpha=\phi^{-1}\circ\tilde{p}\circ\tilde{\phi}$. But now, since $\tilde{p}$ is surjective by Proposition \ref{surjective}, then also $\alpha$ is surjective, thus regular epimorphisms are stable under pullbacks along injective morphisms (i.e. monomorphisms by Lemma \ref{monepi}) in $\Hc$ and (3) of Lemma \ref{l1} is proved. We have obtained the following result:

\begin{prop}
    If $G$ is a finitely generated abelian group and char$\Bbbk\not=2$ then the category $\Hc$ is regular.
\end{prop}

\section{Semi-abelian condition for $\Hc$}

Let $G$ be a finitely generated abelian group and char$\Bbbk\not=2$. In \cite[3.7]{article7} an equivalent characterization for semi-abelian categories is given. It is required that $\Cc$ satisfies the following properties: $\\$ $\\$
1) $\Cc$ has binary products and coproducts and a zero object; $\\$ $\\$
2) $\Cc$ has pullbacks of (split) monomorphisms; $\\$ $\\$
3) $\Cc$ has cokernels of kernels and every morphism with zero kernel is a monomorphism; $\\$ $\\$
4) the Split Short Five Lemma holds true in $\Cc$; $\\$ $\\$
5) cokernels are stable under pullback; $\\$ $\\$
6) images of kernels by cokernels are kernels. $\\$

For the second part of 3) we observe that, since the categorical kernel of a morphism $f:A\to B$ in $\Hc$ is given by the inclusion $i:\mathrm{Hker}(f)\to A$, if this is the zero morphism $u_{A}\circ\epsilon_{\mathrm{Hker}(f)}$ then, given $x\in\mathrm{Hker}(f)$, we have $x=\epsilon(x)1_{A}$ and again $\mathrm{Hker}(f)=\Bbbk1_{A}$ and then $\mathrm{Hker}(f)^{+}=0$. Hence, since we know that the vector space $\mathrm{ker}(f)=A(\mathrm{Hker}(f))^{+}A$, then $f$ is injective or equivalently a monomorphism in $\Hc$ by Lemma \ref{monepi}. Since we have shown that $\Hc$ is pointed, finitely complete (also complete by Remark \ref{lp}), cocomplete, protomodular and regular, properties 1)-5) follow (recall that with $\Cc$ a pointed and finitely complete category 4) is equivalent to the protomodularity of $\Cc$) and then it remains only to prove that the image of a kernel by a cokernel is a kernel. Precisely we want to show that, given $j:\mathrm{Hker}(g)\to X$ a kernel of a mophism $g:X\to Z$ in $\Hc$ and $\mu:X\to X/Xf(A)^{+}X$ a cokernel of a morphism $f:A\to X$ in $\Hc$, there exist a morphism $p:\mathrm{Hker}(g)\to H$ in $\Hc$ and a kernel $\iota:H\to X/Xf(A)^{+}X$ in $\Hc$ such that the following diagram commutes.
\[
  \xymatrix{
     &  A\ar[d]^{f}\\
    \mathrm{Hker}(g)\ar[r]^{j}\ar[d]_{p} &  X \ar[r]^{ g }\ar[d]^{\mu} &Z\\ 
 H\ar[r]_{\iota} &  \frac{X}{Xf(A)^{+}X}}
\]
Now, if we consider the morphism $\mu\circ j$ we know that it has a factorization regular epimorphism-monomorphism in $\Hc$ since this category is regular, i.e. there exist a regular epimorphism $p$ and a monomorphism $\iota$ in $\Hc$ such that $\mu\circ j=\iota\circ p$. But now it is not true in general that every monomorphism is a kernel and then we do not have that $\iota$ is a kernel automatically. We know that $p$ is surjective and $\iota$ is injective by Lemma \ref{monepi}, then we have $\iota=i\circ\iota'$ where $i$ is an inclusion and $\iota'$ is an isomorphism between $p(\mathrm{Hker}(g))$ and $\iota(p(\mathrm{Hker}(g)))=\mu(j(\mathrm{Hker}(g)))=\mu(\mathrm{Hker}(g))$.
\[
\begin{tikzcd}
	\mathrm{Hker}(g) &&& X \\
	& \mu(\mathrm{Hker}(g)) \\
	p(\mathrm{Hker}(g)) &&& \frac{X}{Xf(A)^{+}X}
	\arrow[from=3-1, to=2-2, "\iota'"]
	\arrow[from=2-2, to=3-4, "i"]
	\arrow[from=1-1, to=3-1, "p"']
	\arrow[from=1-1, to=1-4, "j"]
	\arrow[from=1-4, to=3-4, "\mu"]
	\arrow[from=3-1, to=3-4, "\iota"']
\end{tikzcd}
\]
By Corollary \ref{kernel} we have that $\mathrm{Hker}(g)$ is a normal color Hopf subalgebra of $X$ and then $\mu(\mathrm{Hker}(g))$ is a normal color Hopf subalgebra of $X/Xf(A)^{+}X$ by 2) of Lemma \ref{xi}, since $\mu$ is surjective.
So the inclusion $i$ is a kernel again by Corollary \ref{kernel} and, since $\iota'$ is an isomorphism, also $\iota$ is a kernel in $\Hc$ and we are done. Finally we have obtained the following result: 

\begin{thm}
The category $\Hc$ is semi-abelian, if $G$ is a finitely generated abelian group and char$\Bbbk\not=2$.
\end{thm}

\subsection{Some consequences}
Let $G$ be a finitely generated abelian group and char$\Bbbk\not=2$. We can go further and prove something else about $\Hc$, which we know is semi-abelian. Recall \cite[Proposition 5.1.2]{book2} that every semi-abelian category is Mal'cev, so $\Hc$ is an exact Mal'cev category with coequalizers and zero object, thus we know by \cite[Corollary 4.2]{article16} that the category of abelian objects in $\Hc$, which we denote by Ab($\Hc$), is abelian. Hence we want to determine the category Ab($\Hc$) and we use the the characterization given in \cite[Theorem 6.9]{article17} (see also \cite[Proposition 9]{lo}), which states that an object $C$ in a semi-abelian category $\Cc$ is abelian if and only if its diagonal $\langle\mathrm{Id}_{C},\mathrm{Id}_{C}\rangle:C\to C\times C$ is a normal monomorphism, i.e. it is the kernel of some morphism in $\Cc$. But, in $\Hc$, we have that $C\times C=C\otimes C$ and $\langle\mathrm{Id}_{C},\mathrm{Id}_{C}\rangle=(\mathrm{Id}_{C}\otimes\mathrm{Id}_{C})\circ\Delta_{C}=\Delta_{C}$. Hence we have that Ab($\Hc$) is given by those cocommutative color Hopf algebras whose comultiplication is a kernel in $\Hc$.

\begin{thm}
$\mathrm{Ab}(\Hc)$ is the category of commutative and cocommutative color Hopf algebras. In particular this category is abelian.
\end{thm}

\begin{proof}
Given $C$ in $\Hc$, we know that $\Delta$ is injective and then, if we write $\Delta=i\circ\Delta'$ with $\Delta':C\to\mathrm{Im}(\Delta)$ an isomorphism and $i:\mathrm{Im}(\Delta)\to C\otimes C$ the inclusion, we have that $\Delta$ is a kernel in $\Hc$ if and only if $i$ is a kernel in $\Hc$. Furthermore, $i$ is a kernel in $\Hc$ if and only if $\mathrm{Im}(\Delta)$ is a normal color Hopf subalgebra of $C\otimes C$ by Corollary \ref{kernel}. 
If $C$ is a commutative color Hopf algebra, then also $C\otimes C$ is a commutative color Hopf algebra and then $\mathrm{Im}(\Delta)$, which is a color Hopf subalgebra of $C\otimes C$ by Remark \ref{f(A)}, is normal by 3) of Lemma \ref{xi}.
On the other hand, suppose that $\mathrm{Im}(\Delta)$ is a normal color Hopf subalgebra of $C\otimes C$, i.e. that we have $\xi_{C\otimes C}(C\otimes C\otimes\mathrm{Im}(\Delta))\subseteq\mathrm{Im}(\Delta)$. 
For every $a,b,c\in C$ we have that
\[
\begin{split}
\xi_{C\otimes C}(c\otimes1_{C}\otimes a\otimes b)&=m_{C\otimes C}(m_{C\otimes C}\otimes S_{C}\otimes S_{C})(\mathrm{Id}\otimes c_{C\otimes C,C\otimes C})(\Delta_{C\otimes C}\otimes\mathrm{Id})(c\otimes1_{C}\otimes a\otimes b)\\&=m_{C\otimes C}(m_{C\otimes C}\otimes S_{C}\otimes S_{C})(\mathrm{Id}\otimes c_{C\otimes C,C\otimes C})(\phi(|c_{2}|,1_{G})c_{1}\otimes1_{C}\otimes c_2\otimes1_{C}\otimes a\otimes b)\\&=m_{C\otimes C}(m_{C\otimes C}\otimes S_{C}\otimes S_{C})(\phi(|c_{2}\otimes1_{C}|,|a\otimes b|)c_{1}\otimes1_{C}\otimes a\otimes b\otimes c_{2}\otimes1_{C})\\&=m_{C\otimes C}(\phi(|c_{2}|,|a|)\phi(|c_{2}|,|b|)\phi(1_{G},|a|) c_{1}a\otimes b\otimes S_{C}(c_{2})\otimes S_{C}(1_{C}))\\&=\phi(|c_{2}|,|a|)\phi(|c_2|,|b|)\phi(|b|,|c_{2}|)c_{1}aS_{C}(c_{2})\otimes b=\phi(|c_{2}|,|a|)c_{1}aS_{C}(c_{2})\otimes b\\&=(\xi_{C}\otimes\mathrm{Id}_{C})(c\otimes a\otimes b).
\end{split}
\]
Hence, for every $a,c\in C$, we have that
\[
\mathrm{Im}(\Delta)\ni\xi_{C\otimes C}(c\otimes1_{C}\otimes\Delta(a))=(\xi_{C}\otimes\mathrm{Id}_{C})(\mathrm{Id}_{C}\otimes\Delta)(c\otimes a)
\]
and then there exists $x\in C$ such that
\[
\Delta(x)=(\xi_{C}\otimes\mathrm{Id}_{C})(\mathrm{Id}_{C}\otimes\Delta)(c\otimes a).
\]
Thus we obtain 
\[
\begin{split}
x&=l_{C}(\epsilon_{C}\otimes\mathrm{Id}_{C})\Delta(x)=l_{C}(\epsilon_{C}\otimes\mathrm{Id}_{C})(\xi_{C}\otimes\mathrm{Id}_{C})(\mathrm{Id}_{C}\otimes\Delta)(c\otimes a)\\&=l_{C}(\epsilon_{C\otimes C}\otimes\mathrm{Id}_{C})(\mathrm{Id}_{C}\otimes\Delta)(c\otimes a)=l_{C}(\epsilon_{C}\otimes\mathrm{Id}_{C})(c\otimes a)
\end{split}
\]
but we also have that
\[
x=r_{C}(\mathrm{Id}_{C}\otimes\epsilon_{C})\Delta(x)=r_{C}(\mathrm{Id}_{C}\otimes\epsilon_{C})(\xi_{C}\otimes\mathrm{Id}_{C})(\mathrm{Id}_{C}\otimes\Delta)(c\otimes a)=\xi_{C}(c\otimes a_{1})\epsilon_{C}(a_{2})=\xi_{C}(c\otimes a)
\]
and then 
\[
\xi_{C}(c\otimes a)=l_{C}(\epsilon_{C}\otimes\mathrm{Id}_{C})(c\otimes a) \text{ for every }a,c\in C.
\]
Hence $C$ is a commutative color Hopf algebra by 3) of Lemma \ref{xi}. 
\end{proof}

The notion of semi-abelian category was introduced to capture typical algebraic properties of groups but it was noted that there are many significant aspects of groups which are not captured in this more general context, then reinforcements of this notion were born. We recall that a category with finite limits $\Cc$ is called $\textit{algebraically coherent}$ if for each morphism $f:X\to Y$ in $\Cc$ the change-of-base functor $f^{*}:\mathrm{Pt}_{Y}(\Cc)\to\mathrm{Pt}_{X}(\Cc)$ is coherent, i.e. it preserves finite limits and jointly strongly epimorphic pairs (see \cite[Definition 3.1]{article18}). In \cite[Theorems 6.18 and 6.24]{article18} it is shown that semi-abelian categories which are algebraically coherent satisfy both the condition (SH) and (NH) and are peri-abelian and strongly protomodular, thus they are significantly stronger than general semi-abelian categories. So it is interesting to understand if the category $\Hc$ is algebraically coherent, still with $G$ a finitely generated abelian group and char$\Bbbk\not=2$. We recall from \cite{articlegray} that a finitely complete category $\Cc$ is said to be \textit{locally algebraically cartesian closed} when, for every $f:X\to Y$ in $\Cc$, the change-of-base functor $f^{*}:\mathrm{Pt}_{Y}(\Cc)\to\mathrm{Pt}_{X}(\Cc)$ is a left adjoint and that if $\Cc$ is locally algebraically cartesian closed then it is algebraically coherent by \cite[Theorem 4.5]{article18}.
We conclude with the following result:

\begin{prop}
The category of cocommutative color Hopf algebras is action representable and locally algebraically cartesian closed.
\end{prop}

\begin{proof}
By \cite[Proposition 3.2]{articlepo} (see also \cite{book7}), the category $\mathrm{Comon}_{\mathrm{coc}}(\mathrm{Vec}_{G})$ is cartesian closed since $\mathrm{Vec}_{G}$ is a symmetric monoidally closed category (see e.g. \cite{Can}). Thus, since $\Hc=\mathrm{Grp}(\mathrm{Comon}_{\mathrm{coc}}(\mathrm{Vec}_{G}))$, we have that $\Hc$ is locally algebraically cartesian closed by \cite[Proposition 5.3]{article30}. Furthermore, the category of internal groups in a cartesian closed category is always action representable, provided it is semi-abelian, as it is shown in \cite[Theorem 4.4]{article19} and then $\Hc$ is also action representable.
\end{proof}

\textbf{Acknowledgments}. This work started with the master thesis of the author under the supervision of A. Ardizzoni, in which the completeness, cocompleteness and protomodularity of $\Hc$ were proved extending the result achieved in \cite{article1} for Hopf$_{\Bbbk,\mathrm{coc}}$. I would like to thank A. Ardizzoni for the careful reading of this paper and for many meaningful remarks during the development of this work. I would also like to thank A. S. Cigoli for some helpful comments.


\begin{thebibliography}{}

\bibitem{book6}
J. Adámek, J. Rosický,  
\emph{Locally presentable and accessible categories}. 
London Mathematical Society Lecture Note Series, 189. Cambridge University Press, Cambridge, 1994.

\bibitem{article3}
A. L. Agore,
\emph{Limits of coalgebras, bialgebras and Hopf algebras}.
Proc. Amer. Math. Soc. 139 (2011), no. 3, 855–863. 

\bibitem{article4}
A. L. Agore,
\emph{Categorical constructions for Hopf algebras}.
Comm. Algebra 39 (2011), no. 4, 1476–1481.

\bibitem{l}
M. Aguiar, S. Mahajan,  
\emph{Monoidal functors, species and Hopf algebras}. 
With forewords by Kenneth Brown and Stephen Chase and André Joyal. CRM Monograph Series, 29. American Mathematical Society, Providence, RI, 2010. 

\bibitem{articlee2}
H. Albuquerque, S. Majid,  
\emph{Quasialgebra structure of the octonions}. 
J. Algebra 220 (1999), no. 1, 188–224.

\bibitem{a}
E. Aljadeff, O. David,
\emph{On regular G-gradings}. 
Trans. Amer. Math. Soc. 367 (2015), no. 6, 4207–4233.

\bibitem{article5}
N. Andruskiewitsch, I. Angiono, D. Bagio,
\emph{Examples of pointed color Hopf algebras}. J. Algebra Appl. 13 (2014), no. 2, 1350098, 28 pp.

\bibitem{book7}
M. Barr, 
\emph{Coalgebras over a Commutative Ring}, 
J. Algebra 32, 600–610 (1974).

\bibitem{article17}
F. Borceux,
\emph{A survey of semi-abelian categories}. 
In Galois theory, Hopf algebras, and semiabelian categories, volume 43 of Fields Inst. Commun., pages 27–60. Amer. Math. Soc., Providence, RI, 2004.

\bibitem{book1}
F. Borceux,
\textit{Handbook of categorical algebra. 1}.
Basic category theory. Encyclopedia of Mathematics and its Applications, 50. Cambridge University Press, Cambridge, 1994.

\bibitem{book2}
F. Borceux, D. Bourn, 
\textit{Mal'cev, protomodular, homological and semi-abelian categories}. Mathematics and its Applications, 566. Kluwer Academic Publishers, Dordrecht, 2004.

\bibitem{article19}
F. Borceux, G. Janelidze, G.M. Kelly, 
\emph{Internal object actions}, 
Comment. Math. Univ. Carol. 46 (2) (2005) 235–255.

\bibitem{lo}
D. Bourn, 
\emph{Normal subobjects and abelian objects in protomodular categories}, 
J. Algebra 228 (2000) 143–164.

\bibitem{article6}
D. Bourn, M. Gran,
\textit{Regular, protomodular, and abelian categories}.
Categorical foundations, 165–211, Encyclopedia Math. Appl., 97, Cambridge Univ. Press, Cambridge, 2004.

\bibitem{articlegray}
D. Bourn, J.R.A. Gray. 
\emph{Aspects of algebraic exponentiation}. Bull. Belg. Math. Soc. Simon Stevin 19 (2012), no. 5, 823–846.

\bibitem{bre}
T. Brzezinski, R. Wisbauer, 
\emph{Corings and comodules}. 
London Mathematical Society Lecture Note Series, 309. Cambridge University Press, Cambridge, 2003.

\bibitem{Can}
S. Caenepeel, T. Guédénon, 
\emph{On the cohomology of relative Hopf modules}. Comm. Algebra 33 (2005), no. 11, 4011–4034.

\bibitem{article18}
A.S. Cigoli, J.R.A. Gray, T. Van der Linden, 
\emph{Algebraically coherent categories}, 
Theory Appl. Categ. 30 (54) (2015) 1864–1905.

\bibitem{po}
P. Etingof, S. Gelaki,  D. Nikshych, Dmitri, V. Ostrik,
\emph{Tensor categories}. 
Mathematical Surveys and Monographs, 205. American Mathematical Society, Providence, RI, 2015.

\bibitem{article16}
M. Gran, 
\emph{Central extensions and internal groupoids in Maltsev categories}, 
J. Pure Appl. Algebra 155 (2001) 139–166.

\bibitem{articlee}
 M. Gran, G. Kadjo, J. Vercruysse, 
 \emph{A torsion theory in the category of cocommutative Hopf algebras}. Appl. Categ. Structures 24 (2016), no. 3, 269-282.

\bibitem{article1}
M. Gran, F. Sterck, J. Vercruysse, 
\emph{A semi-abelian extension of a theorem by Takeuchi}. J. Pure Appl. Algebra 223 (2019), no. 10, 4171–4190.

\bibitem{article30}
J.R.A. Gray, 
\emph{Algebraic exponentiation in general categories}, Appl. Categ. Structures 20 (2012), 543–567.

\bibitem{article23}
Y. Hiroshi, 
\emph{On group-theoretic properties of cocommutative Hopf algebras}. Hiroshima Math. J. 9 (1979), no. 1, 179–200.

\bibitem{article7}
G. Janelidze, L. Márki, W. Tholen,  
\textit{Semi-abelian categories}. 
Category theory 1999 (Coimbra). J. Pure Appl. Algebra 168 (2002), no. 2-3, 367–386. 

\bibitem{book3}
S. Mac Lane, 
\textit{Categories for the working mathematician}. 
Second edition. Graduate Texts in Mathematics, 5. Springer-Verlag, New York, 1998.

\bibitem{article2}
A. Masuoka, 
\emph{The fundamental correspondences in super affine groups and super formal groups}. J. Pure Appl. Algebra 202 (2005), no. 1-3, 284–312. 

\bibitem{Na}
C. Năstăsescu, F. van Oystaeyen,
\emph{Graded ring theory}. 
North-Holland Mathematical Library, 28. North-Holland Publishing Co., Amsterdam-New York, 1982.

\bibitem{Na2}
C. Năstăsescu, B. Torrecillas,
\emph{Graded coalgebras}. 
Tsukuba J. Math. 17 (1993), no. 2, 461–479.

\bibitem{article8}
K. Newman,
\textit{A correspondence between bi-ideals and sub-Hopf algebras in cocommutative Hopf algebras}.
J. Algebra 36 (1975), no. 1, 1–15. 

\bibitem{article9}
H-E. Porst,
\textit{The formal theory of Hopf algebras Part I: Hopf monoids in a monoidal category}.
Quaest. Math. 38 (2015), no. 5, 631–682. 

\bibitem{article10}
H-E. Porst,
\textit{On subcategories of the category of Hopf algebras}.
Arab. J. Sci. Eng. 36 (2011), no. 6, 1023–1029.

\bibitem{article12}
H-E. Porst, 
\emph{Universal constructions of Hopf algebras}. 
J. Pure Appl. Algebra 212 (2008), no. 11, 2547–2554. 

\bibitem{articlepo}
H-E. Porst,
\emph{On categories of monoids, comonoids, and bimonoids}. 
Quaest. Math. 31 (2008), no. 2, 127–139.

\bibitem{book4}
D. E. Radford,
\textit{Hopf algebras}.
Series on Knots and Everything, 49. World Scientific Publishing Co. Pte. Ltd., Hackensack, NJ, 2012.

\bibitem{Sc}
M. Scheunert, 
\emph{Generalized Lie algebras}. 
J. Math. Phys. 20, 712–720 (1979).

\bibitem{book5}
M. E. Sweedler,
\textit{Hopf algebras}.
Mathematics Lecture Note Series W. A. Benjamin, Inc., New York 1969. 

\bibitem{article11}
C. Vespa, M. Wambst, 
\textit{On some properties of the category of cocommutative Hopf algebras}.
North-West. Eur. J. Math. 4 (2018), 21–37.


\end{thebibliography}
\end{document}